\documentclass[10pt]{amsart}

\usepackage{amsmath,amsthm}
\usepackage[colorlinks=true,linkcolor=blue,urlcolor=blue,citecolor=blue]{hyperref}
\usepackage[all]{xy}

\usepackage{amsrefs} 
\usepackage{enumerate} 
\usepackage{amssymb}  
\usepackage{latexsym} 
\usepackage{comment}
\usepackage{color}
\usepackage{colonequals}
 \usepackage{epsfig}
\usepackage{tikz}
\usepackage{tikz-cd}
\usepackage{dsfont}
\usepackage{caption}

\usepackage{booktabs}

\def\arxiv#1{%
	\href{http://arxiv.org/abs/#1}{arXiv:#1}}%
\def\doi#1{%
	\href{https://doi.org/#1}{doi:#1}}%

\DeclareMathOperator{\Aut}{Aut}

\DeclareMathOperator{\GL}{GL}

\DeclareMathOperator{\Id}{Id}

\DeclareMathOperator{\Irr}{Irr}

\DeclareMathOperator{\PSL}{PSL}

\DeclareMathOperator{\Sing}{Sing}
\DeclareMathOperator{\SL}{SL}

\DeclareMathOperator{\ord}{ord}

\begin{document}

\newtheorem{theorem}{Theorem}[section]
\newtheorem{lemma}[theorem]{Lemma}
\newtheorem{proposition}[theorem]{Proposition}
\newtheorem{corollary}[theorem]{Corollary}
\newtheorem*{theorem*}{Theorem}

\theoremstyle{definition}
\newtheorem{remark}[theorem]{Remark}
\newtheorem{definition}[theorem]{Definition}
\newtheorem{example}[theorem]{Example}
\newtheorem{notation}[theorem]{Notation}

\newtheorem*{problem*}{Problem}
\newtheorem*{question*}{Question}
\newtheorem*{acknowledgements}{Acknowledgements}

\numberwithin{equation}{section}

\renewcommand{\subjclassname}{%
	\textup{2020} Mathematics Subject Classification}

\title[On the classification of regular product-quotient surfaces]{On the classification of product-quotient surfaces with $q=0$, $p_g=3$ and their canonical map}

\author{Federico Fallucca}
\address{Dipartimento di Matematica e Applicazioni\\
	Universit\`a degli Studi di Milano-Bicocca\\
	Via Roberto Cozzi 55\\
	20126 Milano\\ Italy.}
\email{federico.fallucca@unimib.it}

\keywords{surfaces of general type, product-quotient, canonical maps}
\subjclass[2020]{14J29}

\begin{abstract}
	In this work we	present new results to produce an algorithm that returns, for any fixed pair of natural integers $K^2$ and $\chi$, all regular surfaces $S$ of general type with self-intersection $K_S^2=K^2$ and Euler characteristic $\chi(\mathcal O_S)=\chi$, that are product-quotient surfaces. 
\\
The key result we obtain is an algebraic characterization of all families of regular product-quotients surfaces, up to isomorphism, arising from a pair of $G$-coverings of $\mathbb P^1$.
\\
As a consequence of our work, we provide a classification of all regular product-quotient surfaces of general type with $23\leq K^2\leq 32$ and $\chi(\mathcal O_S)=4$. 
\\
Furthermore, we study their canonical map and present several new examples of surfaces of general type with a high degree of the canonical map. 
	\end{abstract}

\maketitle

\tableofcontents

\section*{Introduction}\label{sec: Introduction}

The history of the canonical map of surfaces of general type is long more than $45$ years and it has been recently revived after the beautiful survey \cite{surveyMLP}, where the authors provide an overview of the current state of knowledge on the topic, also outlining a series of still-open questions.

In $1978$, Persson proved that the degree of the canonical map of surfaces of general type is bounded from above by $36$, see \cite{Ulf}. Furthermore, it is known since \cite{beauville} that if the degree is more than $27$, then $q=0$ and $p_g=3$. 
\\
For a long time, the only examples with a high degree of the canonical map were the surfaces of Persson \cite{Ulf} with degree $16$ and Tan \cite{Tan03}, with degree $12$, proving how much challenging can be the construction of new examples. Recently it has been proved that the bound given by Persson is sharp, see \cite{yeung}, \cite{yeung2}, \cite{carlos2}. As a consequence of this, M. Mendes Lopes and R. Pardini revived the topic of the degree of the canonical map and posed in their survey, among other things, the natural question \cite[Question 5.2]{surveyMLP} if all the integers between $2$ and $36$ can be the degree of the canonical map of some surfaces of general type having $q=0$ and $p_g=3$. 
\\
We also noteworthy, as mentioned in \cite{beauville}, that the degree of the canonical map is bounded from above by $K^2_S$, so that minimal surfaces with a high degree of the canonical map not only have $q=0$ and $p_g=3$ but also high values of $K^2_S$. 

In this paper, we construct product-quotient surfaces of general type with $q=0$, $p_g=3$, and $23 \leq K^2_S\leq 32 $ with the ultimate goal to compute the degree of their canonical map and give new examples. 
\\
We remind that $K^2_S=32$ is the highest possible value for product-quotient surfaces with $q=0$ and $p_g=3$, see Theorem \ref{thm: invariants_PQ}. 

We consider product-quotient surfaces as they have proven to be highly useful tools in investigating unresolved conjectures in Algebraic Geometry. As a series of examples, that only deal with regular surfaces, we mention the rigid but not infinitesimally rigid manifolds \cite{rigidity} constructed by Bauer and Pignatelli that gave a negative answer to a question of Kodaira and Morrow \cite[p. 45]{KM71}, the families of surfaces with $p_g=q=0$ constructed in \cite{BP12} realizing $13$ new topological types and for which Bloch's conjecture \cite{Blo75} holds, and the series of papers  \cite{BP16}, \cite{BP12}, \cite{BCGP12}, \cite{BCG08}, \cite{BC04} providing a classification of minimal product-quotient surfaces of general type with  $p_g=q=0$ to give a partial answer to a still-open problem posed by Mumford in 1980, see \cite[p. 551]{BCG08}.

As a first result of this paper, we refine the MAGMA code of \cite{BP12} and we present a new version of it which, taking as input a pair of natural integers $K^2$ and $\chi$, returns all regular surfaces $S$ of general type with self-intersection $K^2_S=K^2$ and Euler characteristic $\chi(\mathcal O_S)=\chi$, that are product-quotient surfaces. 

Although the original script is relatively easy to be adapted to any fixed value of $\chi$ and not only for $\chi=1$ as in \cite{BP12}, it still presents computational time problems as the value of $\chi$ increases. We make the code more performant
by giving new improvements. 

A first novelty is the implementation of the database and the script \textit{FindGenerators} developed in \cite{CGP23}. Such database contains one spherical system of generators of a finite group $G$ for each family of pairwise topologically equivalent $G$-coverings $C$ of $\mathbb P^1$, where the genus of $C$ is $g(C)\leq 27$. We use these tools from \cite{CGP23} to speed up \textbf{Step 3} in Subsection \ref{subsec: descr_implem_class_algorithm} as well. 
\\
The second main novelty is given from the following new result: 
\begin{theorem}\label{thm: short_version_main_thm}
	Let $V_1, V_2$ be two spherical systems of generators of a finite group $G$. Assume that the associated topological types of G-coverings of $\mathbb P^1$ are different.
	The families of product-quotient surfaces associated to this pair of topological types of $G$-coverings is in natural bijection with the set of double cosets 
	\[
	  \mathcal B\Aut(G, V_1)\backslash \Aut(G)/  \mathcal B\Aut(G, V_2).
	\]
\end{theorem}
This is a short version of the main Theorem \ref{thm: comp_inverse_image}. We also remind to the definitions in Subsection \ref{subsec: families_of_product-quotient_surfaces} that make clear the objects presented in Theorem  \ref{thm: short_version_main_thm}. The analogous case of Theorem \ref{thm: short_version_main_thm} where $V_1$ and $V_2$ have topological equivalent associated $G$-coverings of $\mathbb P^1$ is discussed in Corollary \ref{cor: card_set_of_families}. 
\vspace{0,3cm}

As a consequence of these improvements, we run the above-mentioned script to obtain a classification of regular product-quotient surfaces $S$ with $23\leq K_S^2\leq 32$ and $\chi(\mathcal O_S)=4$. What we obtain is the following
\begin{theorem}\label{thm: intro_teo_classif_pg=3}
	Let $S$ be a regular product-quotient surface with $23\leq K_S^2\leq 32$ and $\chi(\mathcal O_S)=4$.  Then $S$ is a surface of general type and it realizes one of the families of surfaces described in tables \ref{table: class_PQ_pg=3} to \ref{table: class_PQ_pg=3_excep_minimality}  in the appendix of this paper. Moreover, surfaces in tables  \ref{table: class_PQ_pg=3} to \ref{table: class_PQ_pg=3_12} are minimal.  
\end{theorem}

We are  interested to compute the degree of the canonical map of product-quotient surfaces, with a particular focus to those with $p_g=3$.   
\\
Let $S$ be a product-quotient surface given by a pair of curves $C_1$ and $C_2$ and a finite group $G$. We prove that the degree of the canonical map of $S$ is determined automatically whenever we compute the (schematic) base locus of the linear subsystem 
associated to the subspace $H^{2,0}(C_1\times C_2)^G$ of invariant $2$-forms of $C_1\times C_2$. 

Such subspace splits as a direct sum of subspaces $\left(H^{1,0}(C_1)^{\chi}\otimes H^{1,0}(C_2)^{\overline{\chi}}\right)^G$, denoted for short by $V_\chi$, one for each irreducible character $\chi \in \Irr(G)$. We remind to Section \ref{sec: degree_can_map} where it makes more clear the meaning of these subspaces. We need the following 

\begin{description}
	\item[\hypertarget{Property $\left(\#\right)$}{Property $\left(\#\right)$}] A product-quotient surface $S$ satisfies Property $\left(\#\right)$ if 
	\[
	\dim V_\chi \neq 0 \implies \deg(\chi)=1
	\]
	for each $\chi \in \Irr(G)$.
\end{description}
\begin{remark}
	\hyperlink{Property $\left(\#\right)$}{Property $\left(\#\right)$} always holds for $G$ abelian group, since each irreducible character of $G$ has degree $1$. 
\end{remark}
Assume that $S$ satisfies \hyperlink{Property $\left(\#\right)$}{Property $\left(\#\right)$}. Then Corollary \ref{cor: base_locus_Kc1c2_chi_G} gives a formula of the base locus of each linear subsystem associated to the subspace $V_\chi$, $\chi\in \Irr(G)$, and so of the base locus of $H^{2,0}(C_1\times C_2)^G$ by intersecting them.
In other words, \hyperlink{Property $\left(\#\right)$}{Property $\left(\#\right)$} makes computable the degree of the canonical map of $S$ automatically, see Subsection \ref{subsec: example_comp_deg_can_map} for an example.

Furthermore, Corollary \ref{cor: base_locus_Kc1c2_chi_G} also implies that \hyperlink{Property $\left(\#\right)$}{Property $\left(\#\right)$} makes the canonical system $\vert K_S\vert$ of $S$ spanned by $p_g$ divisors that are union of fibres (with multiplicity) for the natural fibrations $S\to C_i$, $i=1,2$. 

Thus, we have used the results obtained in Section \ref{sec: degree_can_map} to produce a MAGMA code that automatically computes the degree of the canonical map of a product-quotient surface $S$ with $q=0$ and $p_g=3$ satisfying  \hyperlink{Property $\left(\#\right)$}{Property $\left(\#\right)$}.

We have then selected those surfaces in Theorem \ref{thm: intro_teo_classif_pg=3} satisfying \hyperlink{Property $\left(\#\right)$}{Property $\left(\#\right)$} and we have computed the degree of their canonical map. We have obtained  a series of examples that are listed in Table \ref{table: class_PQ_pg=3_with_deg_can_map}. The numbers of column $no.$ of Table  \ref{table: class_PQ_pg=3_with_deg_can_map} are referred to the row number of tables \ref{table: class_PQ_pg=3} to \ref{table: class_PQ_pg=3_excep_minimality} in the appendix.  
We refer to the appendix of this paper where we explain in details all the other information contained in the columns of  Table \ref{table: class_PQ_pg=3_with_deg_can_map}. 
\begin{table}
	\centering
	\hspace*{-\leftmargin}
	\begin{tabular}{| c|c| c |c|c| l| c| c|c | }
		\hline
		$no.$ &	 $K^2_S$ & Sing($X$) & $t_1$ & $t_2$ & $G$ & $Id$ &$N$ & $\text{deg}(\Phi_S)$\\
		\hline
		\hline
		$1$ & 32 & \empty & $2^{6}$ &	$2^8$ & $\mathbb Z_2^3$ & $\langle 8, 5\rangle $ & 3 &$8, 16^2$ \\ 
		$2$ & 32 & \empty  & $2^{5}$ &	$2^{12}$ & $\mathbb Z_2^3$ & $\langle 8, 5\rangle $ & 3 & 0,4,8\\ 
		$3$ & 32 & \empty & $3^4$ &	$3^7$ & $\mathbb Z_3^2$ & $\langle 9, 2\rangle $ & 2 & 6, 12  \\
		$4$ & 32 & \empty &$3^5$ &	$3^5$ & $\mathbb Z_3^2$ & $\langle 9, 2\rangle $ & 1 & 9 \\ 
		$5$ & 32 & \empty & $2^{3},4^2$ &	$2^3,4^2$ & $G(16,3)$ & $\langle 16, 3\rangle $ & 2 & 16\\
		$7$ & 32 & \empty & $2^2,4^2$ &	$2^5,4^2$ & $G(16,3)$ & $\langle 16, 3\rangle $ & 6 & 8\\ 
		$9$ & 32 & \empty & $2^{3},4$ &	$2^{12}$ & $\mathbb Z_2\times D_4$ & $\langle 16, 11\rangle $ & 6 & 0\\ 
		$12$ & 32 & \empty & $2^{6}$ &	$2^6$ & $\mathbb Z_2\times D_4$ & $\langle 16, 11\rangle $ & 1 & 32 \\ 
		$13$ & 32 & \empty & $2^{5}$ &	$2^8$ & $\mathbb Z_2^4$ & $\langle16, 14\rangle $ & 13 & $0, 8^5, 16^7$\\ 
		$14$ & 32 & \empty & $2^{6}$ &	$2^6$ & $\mathbb Z_2^4$ & $\langle 16, 14\rangle $ & 6 & $8,16^3,32^2$\\ 
		$21$ &32 & \empty & $2^{2},4^2$ &	$2^3,4^2$ & $G(32,22)$ & $\langle 32, 22\rangle $ & 7 & 16\\  
		$28$ & 32 & \empty & $2^5$ &	$2^6$ & $\mathbb Z_2^2\times D_4$ & $\langle 32, 46\rangle $ & 4 & 24\\ 
		$42$ & 32 & \empty & $7^3$ &	$7^3$ & $\mathbb Z_7^2$ & $\langle 49,2\rangle $ & 7 & $0,5,7,10,11,14^2$\\
		$48$ & 32 & \empty & $2^{2},4^2$ &	$2^2,4^2$ & $\mathbb Z_2^5\rtimes \mathbb Z_2 $ & $\langle 64, 60\rangle $ & 3 & 32 \\   
		\hline
		\hline
		$87$ &30& $1/2^2$& $2^3,4$ &	$2^{10},4$ & $\mathbb Z_2\times D_4$ & $\langle 16, 11\rangle $ &6 &0 \\ 
		$88$ &30& $1/2^2$& $2^{4},4$ &	$2^5,4$ & $\mathbb Z_2\times D_4$ & $\langle 16, 11\rangle $ & 2&  4\\ 
			\hline
			\hline
			$119$ &28& $1/2^4$ & $2^2,4^2$ &	$2^8,4^2$ & $\mathbb Z_2\times \mathbb Z_4$ & $\langle 8, 2\rangle $ & 1& 0\\ 
			$120$ &28& $1/2^4$& $2^5$ &	$2^{11}$ & $\mathbb Z_2^3$ & $\langle 8, 5\rangle $ & 6 & $0^2,4^3,8$\\
			$123$ &28& $1/2^4$& $2^{3},4$ &	$2^{11}$ & $\mathbb Z_2\times D_4$ & $\langle 16, 11\rangle $ &  14& 0\\
			$124$ &28& $1/2^4$& $2^{5}$ &	$2^6,4$ & $\mathbb Z_2\times D_4$ & $\langle 16, 11\rangle $ & 6 & 8\\ 
			$125$ &28& $1/2^4$& $2^2,3^2$ &	$3^4,6^2$ & $\mathbb Z_3\times S_3$ & $\langle 18, 3\rangle $ & 6 & $6^2$\\ 
			\hline\hline
			$198$ &26& $1/2^6$& $2^3,4$ &	$2^9,4$ & $\mathbb Z_2\times D_4$ & $\langle 16, 11\rangle $ &14  &0 \\ 
			\hline
				$225$ &26&  $1/3^2,2/3^2$& $3,9^2$ &	$3^2,9^2$ & $\mathbb Z_3\times \mathbb Z_9$ & $\langle 27, 2\rangle $ &6 &$6^3,7,9,10$ \\ 
				$237$ &26& $1/3^2,2/3^2$& $2,6^2$ &	$2^4,6^2$ & $\mathbb Z_2^2\times \mathcal A_4$ & $\langle 48, 49\rangle $ & 5&8 \\ 
			\hline
			\hline
			$ 283$ &24& $ 1/2^8 $ & $ 2^6 $ & $ 2^{10} $ & $ \mathbb Z_2^2$ & $\langle 4, 2 \rangle$ &1 &0 \\
			$ 284 $ &24& $ 1/2^8 $ & $ 2^3, 4^2 $ & $ 2^4, 4^2 $ & $ \mathbb Z_2\times \mathbb Z_4 $ & $\langle 8, 2 \rangle$ & 1&8 \\
			$ 285 $ &24& $ 1/2^8 $ & $ 2^2, 4^2 $ & $ 2^7, 4^2 $ & $  \mathbb Z_2\times \mathbb Z_4  $ & $\langle 8, 2 \rangle$ & 1&2 \\

			$ 286 $ &24& $ 1/2^8 $ & $ 2^2, 4^2 $ & $ 2^4, 4^4 $ & $ \mathbb Z_2\times \mathbb Z_4 $ & $\langle 8, 2 \rangle$ &2 &2, 8 \\
			$ 289 $ &24& $ 1/2^8 $ & $ 2^6 $ & $ 2^7 $ & $ \mathbb Z_2^3 $ & $\langle 8, 5 \rangle$ &11 & $4^3, 6^2, 8^3, 12^2, 16$ \\
			$ 290 $ &24& $ 1/2^8 $ & $ 2^5 $ & $ 2^{10} $ & $ \mathbb Z_2^3 $ & $\langle 8, 5 \rangle$ &14 & $0^4, 4^7, 6, 8^2$ \\
			$ 295 $ &24& $ 1/2^8 $ & $ 2, 4^3 $ & $ 4^4 $ & $\mathbb Z_4^2$ & $\langle 16, 2 \rangle$ & 1&12 \\
			$ 296 $ &24& $ 1/2^8 $ & $ 2^2, 4^2 $ & $ 2^4, 4^2 $ & $\mathbb Z_2^2\rtimes \mathbb Z_4 $ & $\langle 16, 3 \rangle$ &13 & $8^3$\\
			$ 298 $ &24& $ 1/2^8 $ & $ 2^2, 4^2 $ & $ 2^4, 4^2 $ & $ \mathbb Z_2^2\times \mathbb Z_4$ & $\langle 16, 10 \rangle$ &10  & $8^4, 12^4, 16^2$\\
				$ 303 $ &24& $ 1/2^8 $ & $ 2^3, 4 $ & $ 2^{10} $ & $ \mathbb Z_2\times D_4 $ & $\langle 16, 11 \rangle$ & 27& 0\\
			$ 304 $ &24& $ 1/2^8 $ & $ 2^5 $ & $ 2^7 $ & $\mathbb Z_2\times D_4 $ & $\langle 16, 11 \rangle$ &4 &16 \\
			$ 305 $ &24& $ 1/2^8 $ & $ 2^4, 4 $ & $ 2^6 $ & $ \mathbb Z_2\times D_4$ & $\langle 16, 11 \rangle$ & 14&$8^2$ \\
			$ 308 $ &24& $ 1/2^8 $ & $ 2^5 $ & $ 2^7 $ & $ \mathbb Z_2^4 $ & $\langle 16, 14 \rangle$ &13 &$8^5, 12^4, 16^4$  \\
			$ 309 $ &24& $ 1/2^8 $ & $ 2^2, 3^2 $ & $ 3, 6^4 $ & $ \mathbb Z_3\times S_3 $ & $\langle 18, 3 \rangle$ &3 &0, 6 \\
			$ 312 $ &24& $ 1/2^8 $ & $ 2, 3^4, 6 $ & $ 2^2, 3^2 $ & $ \mathbb Z_3\times S_3 $ & $\langle 18, 3 \rangle$ &3 & 6\\
			$ 376 $ &24& $ 1/2^8 $ & $ 2, 3^2, 6 $ & $ 3, 6^2 $ & $S_3\times \mathbb Z_3^2 $ & $\langle 54, 12 \rangle$ &9 &12, (16, 18), (13, 15), 18, 24 \\
			\hline
			$ 459 $ &24& $ 1/4^2, 3/4^2 $ & $ 2^3, 4 $ & $ 2^9, 4 $ & $ \mathbb Z_2\times D_4 $ & $\langle 16, 11 \rangle$ & 6& 0\\
			\hline
			\hline
			$ 475 $ &23& $ 1/3^3, 2/3^3 $& $ 3^4 $& $ 3^6 $& $ \mathbb Z_3^2 $& $\langle 9, 2 \rangle$ &6 &$6^5,9$ \\
			$ 477 $ &23& $ 1/3^3, 2/3^3 $& $ 2^2, 3^2 $& $ 2^4, 3, 6 $& $ \mathbb Z_2\times \mathcal A_4$& $\langle 24, 13 \rangle$ &2 &8 \\
			$ 486 $ &23& $ 1/3^3, 2/3^3 $& $ 2, 6^2 $& $ 2^4, 3, 6 $& $ \mathbb Z_2^2\times \mathcal A_4 $& $\langle 48, 49 \rangle$ &6 & 8\\
			\hline
	\end{tabular}
\caption{} \label{table: class_PQ_pg=3_with_deg_can_map}
\end{table}

 The paper is organized as follows: 
\\
In Section \ref{sec: from_data_coverings_of_P1_to} we discuss finite group actions on a product of Riemann surfaces. We  then present the main Theorem \ref{thm: comp_inverse_image}, the extended version of Theorem \ref{thm: short_version_main_thm} of the introduction, crucial to speed up the classification algorithm.

Roughly speaking, Theorem \ref{thm: comp_inverse_image} plays a crucial role in determining the irreducible families of surfaces of Theorem \ref{thm: intro_teo_classif_pg=3}. 
Techniques to establish whether a pair of product-quotient surfaces belong  to the same irreducible family have been extensively studied first in \cite[Thm. 1.3]{BC04} and \cite[Prop. 5.2]{BCG08} in the case of surfaces isogenous to a product, and then in the general case in \cite{BP12}. 
\\
Theorem \ref{thm: comp_inverse_image} seems to be a relevant new result on this problem, very useful in overcoming the huge amount of calculations that usually occur when adopting those techniques. We refer also to the last Section \ref{sec: comp_complexity} where, among other things, we show the efficiency of our MAGMA code by using Theorem \ref{thm: comp_inverse_image}. 

Apart from the rows where $N$ of families is denoted by $?$, and whose challenges are discussed in Remark \ref{rem: comp_diffi_Find_Surf}, the classification outlined in Theorem \ref{thm: intro_teo_classif_pg=3} yields a total of $1502$ irreducible families of minimal surfaces of general type. Additionally, each family with $K^2=32$ maps onto an irreducible component (in the Zariski topology) of the Gieseker moduli space $\mathfrak M_{(4, 32)}$, which consists of surfaces of general type with $K_S^2=32$ and $\chi(\mathcal O_S)=4$. The remaining cases, where $23\leq K^2\leq 30$, are more delicate and we refer to Subsection \ref{subsec: families_of_product-quotient_surfaces} and Remark \ref{rem: Gieseker_modui}.

In Section \ref{sec: class_algorithm} we
generalize \cite[Prop. 1.14]{BP12} to any $\chi \in \mathbb N$ and discuss the classification algorithm. 

In Section \ref{sec: class_chi=4} we prove Theorem \ref{thm: intro_teo_classif_pg=3}. In particular, we show that all surfaces of  Theorem \ref{thm: intro_teo_classif_pg=3} are of general type and those in tables \ref{table: class_PQ_pg=3} to \ref{table: class_PQ_pg=3_excep_minimality} are also minimal. We also discuss the exceptional cases arising from the secondary output of the function $ListGroups(K^2,4)$ for each $K^2\in\{23,\dots, 32\}$ in order to obtain the complete list of tables \ref{table: class_PQ_pg=3}  to \ref{table: class_PQ_pg=3_excep_minimality} of the appendix.

In Section \ref{sec: degree_can_map} we investigate the canonical map of product-quotient surfaces $S$ and we present an algorithm to compute its degree whenever $S$ satisfies \hyperlink{Property $\left(\#\right)$}{Property $\left(\#\right)$}. The examples of  \ref{table: class_PQ_pg=3_with_deg_can_map} with a high degree of the canonical map that are to our knowledge already discovered in the literature are the following:
\begin{itemize}
	\item  Surfaces of $no.42$ in Table \ref{table: class_PQ_pg=3_with_deg_can_map} are the examples presented in \cite{fede1};
	\item Families of surfaces $no. 376$ having a degree of the canonical map $12$, $(16, 18)$, $(13, 15)$, $18$ are all those in \cite{fede2}.
	Furthermore, we point out that only surfaces $no. \ \textit{1}$ of \cite[Thm. 2.3]{fede2} satisfy \hyperlink{Property $\left(\#\right)$}{Property $\left(\#\right)$} 
	thanks to which the degree of their canonical map was automatically  computable. 
	
\item	One of the families of $no. 28$ in Table \ref{table: class_PQ_pg=3_with_deg_can_map} having $24$ as degree of the canonical map is already studied by the author of the present paper and D. Frapporti and can be found in \cite[Sec. 6.3]{fede3}. We also mention that this family of surfaces does not satisfy  \hyperlink{Property $\left(\#\right)$}{Property $\left(\#\right)$}, 
	hence in this case we have had to find the equations of the pair of curves realizing that family of surfaces and then studied by hands the degree of their canonical map. 
	\item 
	Two of the six families of $no.14$ in Table \ref{table: class_PQ_pg=3_with_deg_can_map} having degree $32$ of the canonical map are discussed in \cite{gpr}. They are described there differently from us, as $\mathbb Z_2^4$-coverings of $\mathbb P^1\times \mathbb P^1$ using the language of Pardini's theory of abelian coverings. Surfaces of these families are the only examples in the literature with a canonical map of degree $32$, which is also the highest possible degree for product-quotient surfaces as observed in Remark \ref{rem: higest_value_Ksquare}. 
	
	Furthermore, 
	the authors proved in \cite[Prop. 5.3]{gpr} that these two examples are the only product-quotient surfaces with $G$ abelian having degree of the canonical map equal to $32$. The same question with $G$ not abelian was still-open and it find an answer in the present paper. Indeed, there are other families of surfaces in Table \ref{table: class_PQ_pg=3_with_deg_can_map} with a canonical map of degree 32. 
	
\end{itemize}

Section \ref{sec: comp_complexity} is devoted to comments about the computational complexity of the presented algorithms. 
\vspace{0,3cm}

An expanded version of Tables \ref{table: class_PQ_pg=3} to \ref{table: class_PQ_pg=3_excep_minimality} of the appendix describing all the needed data to work explicitly with one of the surfaces and a commented version of the MAGMA codes we used can be found here: 
\\
\url{https://github.com/Fefe9696/PQ_Surfaces_with_fixed_Ksquare_chi}

\section*{Notation}
We will use the basic notations of the theory of smooth complex projective surfaces, hence $K_S$ is the \textit{canonical class} of $S$, $p_g:=h^0(S, K_S)$ is the \textit{geometric genus}, $q(S):=h^1(S, \mathcal O_S)$ is the \textit{irregularity}, and $\chi(\mathcal{O}_S)=1-q+p_g$ is the \textit{Euler characteristic}.

\section{Algebraic characterization of families of product-quotient surfaces }
\label{sec: from_data_coverings_of_P1_to}
Let us consider a finite group $G$ acting on two smooth projective curves $C_1$ and $C_2$ of respective genera at least $2$. We consider the diagonal action of $G$ on $C_1\times C_2$. In this case the action of $G$ on $C_1\times C_2$ is called \textit{unmixed}. 
Following \cite[Remark 3.10]{Cat00}, we can assume that the action on $C_i$ is faithful. 
\begin{definition}\cite[Defn. 0.1]{BP12}
	The minimal resolution of singularities $S$ of $X=\left(C_1\times C_2\right)/G$ is called \textit{product-quotient} surface of the \textit{quotient model} $X$. 
\end{definition}
Let $S$ be a product-quotient surface of quotient model $\left(C_1\times C_2\right)/G$. From a theorem due to Serrano \cite[Prop. 2.2]{Ser96}, then $q(S)=0$ if and only if $C_i/G\cong \mathbb P^1$. 
\\
In other words, pairs of $G$-coverings of the projective line defines regular product-quotient surfaces. For this reason, let us briefly recall how coverings of $\mathbb P^1$ can be described. 

\subsection{Algebraic characterization of families of $G$-coverings of $\mathbb P^1$}\label{subsec: families_of_G_coverings}
\begin{definition}
	Let $G$ be a finite group. For a $G$-covering of $\mathbb P^1$ we mean a Riemann surface $C$ together with a (holomorphic) action $\phi$ of $G$ on $C$ such that the quotient $C/G$ is $\mathbb P^1$. Whenever we need to specify the action, we write $(C,\phi)$.
\end{definition}
There are two notions of equivalence among $G$-coverings of $\mathbb P^1$: 
we say that $C_1$ and $C_2$ are \textit{topologically equivalent} if there exists a orientation preserving homeomorphism $f\colon C_1\to C_2$ and an automorphism $\varphi \in \Aut(G)$ such that $f(g\cdot p)=\varphi(g)\cdot f(p)$ for any $g\in G$ and $p\in C_1$. We say that $C_1$ and $C_2$ are \textit{isomorphic} if moreover $f$ is a biholomorphism. 

Consider the set of G-coverings of $\mathbb P^1$ modulo isomorphism. The topological equivalence partitions it into equivalence classes, let $\mathcal C$ be one of them. Gonz\'alez-D\'iez and Harvey showed in \cite{GDH92} that $\mathcal C$ has a natural structure of connected complex manifold such that the natural map of $\mathcal C$ on the moduli space of curves mapping $(C, \phi)$ to $C$ is analytic. More precisely, the manifold $\mathcal C$ is the normalization of its image $\widetilde{\mathcal C}$. In particular, $\widetilde{\mathcal C} $ is always an irreducible subvariety of the moduli space of curves. 
\\
The manifold $\mathcal C$ can be realized by taking a $G$-covering $C\in \mathcal C$ and moving the branch points of its covering map $C\to \mathbb P^1$. Each move of the branch points of $C$ corresponds to a topological covering of a complement of $\mathbb P^1$ given by those points. Thus, from the Riemann Existence Theorem, $C$ can be endowed with a new structure as Riemann surface. Hence we obtain another $G$-covering of $\mathbb P^1$ always topological equivalent but possibly not isomorphic to $C$. More precisely, they are not isomorphic if the move of the branch points is not a projective transformation of $\mathbb P^1$. In other words, we are realizing $G$-coverings of $\mathcal C$ not isomorphic to $C$ whenever we fix three of the branch points of $C$ and move the remaining ones. 
\\
From this we may easily deduce the dimension of the complex manifold $\mathcal C$, which is $r-3$, where $r$ is the number of branch points of a $G$-covering $C\in \mathcal C$. 
\begin{definition}
	We set $\mathcal T^r(G)$ be the collection of all classes of $G$-coverings of  $\mathbb P^1$ ramified over $r$ points modulo topological equivalence. 
\end{definition}
From the above discussion, we invite the reader to think of each element of $\mathcal T^r(G)$ as a class $\mathcal C$ of families of $G$-coverings of $\mathbb P^1$ two-by-two not isomorphic but all topological equivalent to each other. 

We shall give an algebraic description of the elements of $\mathcal T^r(G)$. 
\begin{definition}\label{defn: ssg}
	A \textit{spherical system of generators} (of length $r$) of $G$ is a sequence $[g_1, \dots, g_r]\in G^r$ of elements of $G$ such that $g_i\neq 1$ for all $i$, and
	\begin{itemize}
		\item $G=\langle g_1, \dots, g_r\rangle$;
		\item $g_1\cdots g_r=1$. 
	\end{itemize}
	The sequence $[o(g_1), \dots, o(g_r)]$ is called \textit{signature} of $[g_1, \dots, g_r]$. 
\end{definition}
\begin{definition}
	We set $\mathcal D^r(G)\subset G^r$ be the collection of all spherical systems of generators of $G$ of length $r$. 
\end{definition}
\begin{remark}\label{rem: orbifold_morphism}
	For each signature $[m_1, \dots, m_r]$ consider the \textit{orbifold group} 
	\[\mathbb T(m_1,\dots, m_r):=\langle \gamma_1, \dots, \gamma_r\vert \gamma_1^{m_1}, \dots, \gamma_r^{m_r}, \gamma_1\cdots \gamma_r\rangle.\]
	There is a natural bijection among the  set of surjective homomorphisms $ \ \ \  \  \ \\ \ \ \mathbb T(m_1, \dots, m_r)\to G$ and the set of spherical systems of generators of signature $[m_1,\dots, m_r]$. 
	
	The bijection associates to any homomorphism $\varphi$ the spherical system of generators $[\varphi(\gamma_1), \dots, \varphi(\gamma_r)]$.
\end{remark}

Consider the braid group $\mathcal{B}_r$, whose presentation with generators $\sigma_1, \dots, \sigma_{r-1}$ is 
\[
\mathcal{B}_r=\left\langle \sigma_1, \dots, \sigma_{r-1}\colon \begin{split}
	\sigma_i\sigma_j & =\sigma_j\sigma_i,  \quad &\vert i-j\vert  >1 \\
	\sigma_i\sigma_j\sigma_i & =\sigma_j\sigma_i\sigma_j, \quad & \vert i-j\vert =1
\end{split} \right\rangle. 
\]
The elements of $\mathcal{B}_r$ are called \textit{Hurwitz moves}. We consider the following action of $\Aut(G)\times \mathcal{B}_r$ on $\mathcal{D}^r(G)$:
\[
\Psi\cdot [g_1, \dots , g_r]:= [\Psi(g_1), \dots, \Psi(g_r)], \qquad \Psi\in \Aut(G).
\]
\[
\sigma_i\cdot [g_1, \dots, g_r]:=[g_1, \dots, g_{i-1}, g_i\cdot g_{i+1}\cdot g_{i}^{-1}, g_i,  g_{i+2}, \dots g_{r}], \qquad \sigma_i\in \mathcal{B}_r.
\]
The action of the generators $\sigma_i$ extends to an action of the entire $\mathcal{B}_r$. 
\\
We finally have the following classical result 
\begin{theorem}\label{thm: bij_DrGmodulo_TrG}
	The collection of all classes of $G$-coverings of $\mathbb P^1$ ramified over $r$ points modulo topological equivalence is in bijection with  $\mathcal{D}^r(G)/\Aut(G)\times\mathcal{B}_r$:
	\begin{equation}\label{eq: correspo_coverings_ssg}
		\mathcal{T}^r(G)\cong \mathcal{D}^r(G)/\Aut(G)\times \mathcal{B}_r.
	\end{equation}
\end{theorem}
For a recent proof, we refer to \cite[Cor. 5.7]{GT22}. 

We briefly describe the bijection in Theorem \ref{thm: bij_DrGmodulo_TrG}. Take an element in the quotient  $\mathcal{D}^r(G)/\Aut(G)\times \mathcal{B}_r$, and choose a representative $[g_1, \dots, g_r]$ of it. From Remark \ref{rem: orbifold_morphism} we obtain a surjective morphism $\mathbb T(m_1,\dots, m_r)\to G$, with $m_i=o(g_i)$.
\\
We choose a finite set $X:=\lbrace q_1, \dots, q_r\rbrace$ on $\mathbb P^1$, a point $q_0\in \mathbb P^1\setminus X$, and a  \textit{geometric basis} of  the fundamental group of $\mathbb P^1\setminus X$ 
with base point $q_0$, which is a set of $r$ distinct homotopy class loops $\eta_i$ of $\mathbb P^1\setminus X$  starting at $q_0$ and travelling around $q_i$ in the counterclockwise order, $i=1,\dots, r$. 
\\
Observe that any class loop of $\pi_1(\mathbb P^1\setminus X,q_0)$ is generated by  $\eta_1, \dots, \eta_r$ and the only relation among $\eta_1, \dots, \eta_r$ is that their product $\eta_1\cdots \eta_r$ can be contracted to $q_0$. In other words, we have obtained $\pi_1(\mathbb P^1\setminus X,q_0)$ has presentation  $\langle \gamma_1, \dots, \gamma_r| \gamma_1\cdots\gamma_r\rangle$.

We observe then $\mathbb T(m_1, \dots, m_r)$ is a quotient of $\pi_1(\mathbb P^1\setminus X,q_0)$, and the kernel of the composition 
\[
\pi_1(\mathbb P^1\setminus X,q_0)\to \mathbb T(m_1, \dots, m_r)\to G
\]
defines a unique topological $G$-covering of $\mathbb P^1\setminus X$. By Riemann Existence theorem, this completes to a $G$-covering $C$ of $\mathbb P^1$. 

The bijection of Theorem \ref{thm: bij_DrGmodulo_TrG} maps the class of  $[g_1, \dots, g_r]$ modulo $\Aut(G)\times \mathcal B_r$ to the class of $C$ modulo topological equivalence.

In particular, Theorem \ref{thm: bij_DrGmodulo_TrG} says that 
\begin{enumerate}
	\item  if in the above construction we change 
	\begin{itemize}
		\item the set of spherical generators $[g_1, \dots, g_r]$ by a set in the same orbit for the action of $\Aut(G)\times \mathcal B_r$, or
		\item the points $q_0, q_1, \dots, q_r$ with other $r+1$ points of $\mathbb P^1$, or
		\item the choice of the geometric basis $\eta_1,\dots, \eta_r$
	\end{itemize}
	then the new obtained $G$-covering is topologically equivalent to $C$;
	\item if $C_1$ and $C_2$ are obtained by spherical systems of generators that are not in the same $\Aut(G)\times \mathcal B_r$ orbit then $C_1$ and $C_2$ are not topologically equivalent.
	\item every $G$-covering $C$ of $\mathbb P^1$ ramified over $r$ points up to topological equivalence is obtained in this way by a spherical system of generators of $G$. 
\end{enumerate}
We need to discuss only the third point above. Hence we show how to get a spherical system of generators from a $G$-covering of the projective line. 
\\
Given a $G$-covering $\lambda\colon C\to \mathbb P^1$ whose branch locus consists of $r$ points $X:=\lbrace q_1, \dots, q_r\rbrace$, then $\lambda\colon C\setminus \lambda^{-1}(X)\to \mathbb P^1\setminus X$ is a topological covering. Chosen a point $p_0$ over $q_0\in \mathbb P^1\setminus X$, we consider the monodromy map $\pi_1(\mathbb P^1\setminus X)\to \lambda^{-1}(p_0)$. Since $q_0$ is not a branch point  of the covering, then $\lambda^{-1}(q_0)$ consists of $\vert G\vert$ points so that we can identity $\lambda^{-1}(q_0)$ with $G$: $g\cdot p_0\mapsto g$. In this way, the monodromy map becomes a morphism of groups. 
\\
Notice that only the kernel of this map is uniquely determined by the covering. 
\\
Each class loop of the geometric basis $\eta_1, \dots, \eta_r$ is sent to elements $g_1, \dots, g_r$ via the monodromy map. Since the monodromy map is surjective, then $g_1,\dots, g_r$ generate $G$, whilst the contraction of $\eta_1\cdots\eta_r$ to the point $q_0$ translates as $g_1\cdots g_r=1_G$. Thus, $[g_1, \dots, g_r]$ is a spherical system of generators of $G$. 

This construction makes also very clear the geometrical meaning of the elements $g_1, \dots, g_r$. Indeed, we see that $g_i$ is a generator of the stabilizer subgroup of a point $p_i$ over $q_i$ and there is a local coordinate $z$ on $C$ around $p_i$ such that the action of $g_i$ sends a point of coordinate $z$ to the point of coordinate $e^{\frac{2\pi i}{m_i}} z$, where $m_i$ is the ramification index of $p_i$. 
\\
In other words,  $g_i$ is  the \textit{local monodromy} of a point over $q_i$, see \cite[Sec. 2.1.1]{fede3} for more details. 
\begin{remark}
	Notice that from the above discussion we have also obtained the order of $g_i$ is simply the ramification index $m_i$ of $q_i$. 
	\\
	Hence the bijection of Theorem \ref{thm: bij_DrGmodulo_TrG} sends the class $[g_1,\dots, g_r]$ modulo $\Aut(G)\times \mathcal{B}_r$ to the class of a $G$-covering $C$ modulo topological equivalence branched on $r$ points $\lbrace q_1, \dots, q_r\rbrace$ with ramification indices $o(g_1), \dots, o(g_r)$ respectively and such that the Hurwitz formula holds:
	\begin{equation}\label{eq: Hurwitz_formula}
		2g(C)-2=\vert G\vert\left(-2+\sum_{i=1}^r\left(1-\frac{1}{o(g_i)}\right)\right).
	\end{equation}
	Here the cyclic groups $\langle g_i\rangle$ (and their conjugates) provide the non-trivial stabilizers of the action of $G$ on $C$. 
\end{remark}
Let us give a example of how to use Theorem \ref{thm: bij_DrGmodulo_TrG}.
\begin{example}\label{ex: ssg_S3xZ3square_ramified_over_3_points}
	We are going to compute $\mathcal T^3(S_3\times \mathbb Z_3^2)$, the collection of the $S_3\times \mathbb Z_3^2$-coverings of $\mathbb P^1$ up to topological equivalence ramified over $3$ points. Up to apply suitable Hurwitz moves, we can assume that a spherical system of generators $[(g_1, v_1), (g_2, v_2), (g_3, v_3)]$ has $o(g_1)\leq o(g_2)\leq o(g_3)$. Observe $g_i\neq 1$, otherwise $S_3$ would be generated by only one element, and this is not possible since it is not cyclic. The same argument holds for $\mathbb Z_3^2$, so that $v_i\neq 0$. 
	This implies   $[g_1, g_2, g_3]\in \mathcal{D}^3(S_3)$, and $[v_1, v_2, v_3]\in \mathcal{D}^3(\mathbb Z_3^2)$. By Hurwitz formula \ref{eq: Hurwitz_formula}, then $\sum_{i=1}^3\frac{1}{o(g_i)}$ has to be an integer, which holds only either for  $o(g_1)=o(g_2)=o(g_3)=3$ or $o(g_1)=o(g_2)=2$, and $o(g_3)=3$. The first case can be excluded since there are no $g_1, g_2, g_3$ of order $3$ generating $S_3$.
	\\
	Let us focus on the second case, which gives $g(C)=0$, so $C\cong \mathbb P^1$. The elements of order $2$ of $S_3$ are $\tau, \tau\sigma$, and  $\tau\sigma^2$, where $\tau$ is a reflection (transposition) and $\sigma$ is a rotation (3-cycle) of $S_3$. 
	Since $g_3=g_2^{-1}g_1^{-1}$, then $g_1\neq g_2$ otherwise $g_3=1$ since $g_1$ and $g_2$ has order two. 
	\\
	Thus the list of spherical systems with ordered signature $[2,2,3]$ consists only of six elements obtained just by choosing a distinct pair of $g_1, g_2$ in the set $\lbrace \tau, \tau\sigma, \tau\sigma^2\rbrace$. From here it is easy to see that the action of $\Aut(S_3)$ on $\mathcal{D}^3(S_3)$ is transitive. 
	
	From the other side, it is clear that the action of $\GL_2(\mathbb Z_3)$ on $\mathcal{D}^3(\mathbb Z_3^2)$ is transitive.  Thus $\Aut(S_3\times \mathbb Z_3^2)$ acts transitively on $\mathcal{D}^3(S_3\times \mathbb Z_3^2)$ and from Theorem \ref{thm: bij_DrGmodulo_TrG} we obtain
	\[
	\begin{split}
		\mathcal{T}^3(S_3\times \mathbb Z_p^2)\cong   \frac{\mathcal{D}^3(S_3\times \mathbb Z_3^2)}{\Aut(S_3\times \mathbb Z_3^2)\times  \mathcal{B}_3} 
		=  \lbrace [(\tau, (1,0)), (\tau\sigma, (0,1)), (\sigma^2, (p-1, p-1))]\rbrace. 
	\end{split}
	\]
	By the Hurwitz formula \eqref{eq: Hurwitz_formula}, the genus of the corresponding  $G$-covering $C$ is $g(C)=10$. 
	\\
 Here $C$ may be described explicitly by equations as follows: we consider the projective space $\mathbb{P}^3$ with homogeneous coordinates $x_0, \ldots, x_3$ and define
	\[
	C \colon \begin{cases}
		x_2^3=x_0^3-x_1^3 \\
		x_3^3=x_0^3+x_1^3
	\end{cases}.
	\]
	The action $\phi\colon  S_3\times \mathbb Z_3^2 \to \Aut(C)$ is given by 
	\[
	\left(\sigma^i\tau^j, (a,b)\right) \mapsto [(x_0:x_1:x_2: x_3) \mapsto (\zeta_3^ix_{[j]}:x_{[j+1]} : (-1)^j\zeta_3^{2a+2i} x_2:  \zeta_3^{2b+2i} x_3)],
	\]
	where $\zeta_3:=e^{\frac{2\pi i}{3}}$ is the first $3$-root of the unity. 
	Finally, the covering map by this action is 
	\[
	\lambda \colon C \stackrel{9 : 1}{\longrightarrow} \mathbb P^1 \stackrel{6 : 1}{\longrightarrow} \mathbb P^1, \qquad (x_0:x_1:x_2:x_3) \mapsto (x_0:x_1)\mapsto \left(x_0^3x_1^3: (x_0^6+x_1^6)/2\right).
	\]
\end{example}
\begin{remark}
	As we could expect, it becomes soon computationally difficult getting the $\Aut(G)\times \mathcal{B}_r$-orbits of $\mathcal{D}^r(G)$, when $r$ or $G$ increase. For this reason, several authors put an increased effort into the development of an efficient algorithm to compute such orbits, usually with the helping also of a computational algebra system (e.g. MAGMA, \cite{BCP97}). A big step forward in this direction is given for instance in \cite{CGP23}, where the authors 
	collect in a database a representative for any orbit of spherical systems of generators of fixed genus $g\leq 40$, with a few exceptions. 
	\\
	We use this database and their script \textit{FindGenerators} to speed up \textbf{Step 3} of  Subsection \ref{subsec: descr_implem_class_algorithm} and in combination with  Theorem \ref{thm: comp_inverse_image} to give improvements in \textbf{Step 5}.
\end{remark}

\subsection{Families of product-quotient surfaces from a pair of coverings of $\mathbb P^1$}\label{subsec: families_of_product-quotient_surfaces}
In this subsection we study how to realize all families of product-quotient surfaces obtained by a pair of topological types of $G$-coverings of $\mathbb P^1$. 
\begin{definition}\label{defn: isomorphism_betweeen_PQs}
	Let us call by $\mathcal{T}^{r,s}(G)$ the collection of all families of regular product-quotient surfaces, whose natural  fibrations $\lambda_i$ are $G$-coverings $C_i$ of $\mathbb P^1$ branched over $r$ and $s$ points respectively. 
\end{definition}
\begin{remark}\label{rem: exchanging_factors}
	In the above definition the order of $C_1$ and $C_2$ is relevant. Thus exchanging them gives  a natural bijection $\iota \colon\mathcal T^{r,s}(G)\to \mathcal T^{s,r}(G)$ which sends families to isomorphic families of surfaces. 
\end{remark}
We give a generalization of Theorem \ref{thm: bij_DrGmodulo_TrG} for product-quotient surfaces (see \cite{BP12} and \cite{BCGP12} for more details). 
\begin{proposition}\label{prop: F_constants_on_orbits}
	There is a natural bijection among $\mathcal T^{r,s}(G)$ and 
	\[\frac{\mathcal D^r(G)\times \mathcal D^s(G)}{\Aut(G)\times \mathcal B_r\times \mathcal B_s},
	\] 
	where $\Aut(G)$ acts simultaneously on both factors, whilst $\mathcal B_r$ and $\mathcal B_s$ act on the first and second factor respectively. 
\end{proposition}
\begin{remark}\label{rem: Gieseker_modui}
	We point out that each of the families of $\mathcal T^{r,s}(G)$ maps onto an algebraic subset of the Gieseker moduli space, but the images of two different families may not be distinct. This because we are considering an equivalence relation among product-quotient surfaces that is weaker than the equivalence relation \textit{being isomorphic}. 
	\\
	However, as proved in \cite[Prop. 5.2]{BCG08}, regular product-quotient surfaces $S_1$ and $S_2$ isogenous to a product are isomorphic if either their algebraic data share the same orbit by the action of $\Aut(G)\times \mathcal B_r\times \mathcal B_s$ or the isomorphism is given by exchanging the factors  $C_1$ and $C_2$. 
\end{remark}
The  bijection  in Proposition \ref{prop: F_constants_on_orbits} is given by a map $\mathcal D^r(G)\times \mathcal D^s(G)\to T^{r,s}(G)$ as follows.

Consider a pair of spherical systems of generators $[g_1, \dots, g_r]$ and $\ \ $ $[h_1,\dots, h_s]$. 
We fix points $q_0, q_1, \dots, q_r \in  \mathbb P^1$ and a geometric basis $\eta_1, \dots, \eta_r$ as in the Subsection \ref{subsec: families_of_G_coverings}, where $\eta_i$ is a class loop based at $q_0$ travelling around $q_i$ in the counterclockwise order. In this way, following the description of Theorem \ref{thm: bij_DrGmodulo_TrG} in Subsection \ref{subsec: families_of_G_coverings} we get from the first spherical system $[g_1, \dots, g_r]$ a $G$-covering of the line $\lambda_1\colon C_1\to \mathbb P^1$ whose branch locus is $\{q_1, \dots, q_r\}$, where $q_i$ has ramification index $o(g_i)$ and $g_i$ is the local monodromy of a point over $q_i$. 
 
A similar argument holds for $[h_1,\dots, h_s]$, so we get another $G$-covering $\lambda_2\colon C_2\to \mathbb P^1$ branched over $s$ points with ramification indices $o(h_1), \dots, o(h_s)$ and $h_i$ is the local monodromy of a point over $q_i$. 

The diagonal action of $G$ on $C_1\times C_2$ gives a product-quotient surface $S$ whose quotient model is $\left(C_1\times C_2\right)/G$. 

The map $\mathcal D^r(G)\times \mathcal D^s(G)\to T^{r,s}(G)$ sends the pair of spherical systems $([g_1, \dots, g_r], [h_1,\dots, h_s])$ to the family of  $S$. 

Let us discuss how $\Aut(G)\times \mathcal B_r\times \mathcal B_s$ acts on this construction. 

We show that acting with  $\Psi\in \Aut(G)$ simultaneously on $[g_1, \dots, g_r]$ and $[h_1, \dots, h_s]$ the isomorphic class of $S$ does not change. Acting with $\Psi$, we obtain the same $G$-coverings $C_1$ and $C_2$, but the isomorphisms among $G$ and the automorphism groups of $C_1$ and $C_2$ are both modified by composition with $\Psi$.
Then we obtain the same product $C_1\times C_2$ and the action of $G\times G$ on it has been modified by composition with $\Psi\times \Psi$. Since $\Psi\times \Psi$ sends the diagonal to itself, then we obtain a surface isomorphic to $S$. 

The group $\mathcal B_r$ acts only on the first spherical system of generators $[g_1, \dots$ $, g_r]$ replacing $C_1$ with a topological equivalent $G$-covering $C_1'$ as described in Subsection \ref{subsec: families_of_G_coverings}. By the result of  Gonz\'alez-D\'iez and Harvey in \cite{GDH92} mentioned there, then $C_1$ and $C_1'$ are in the same irreducible connected family of $G$-coverings. In particular, the action of $\mathcal B_r$ on the given construction connects surfaces of the same family. 

An analogous statement holds for the action of $\mathcal B_s$ on a spherical system of generators $[h_1, \dots, h_s]$.  
\vspace{0,4cm}

To each family of product-quotient surfaces we have a naturally associated pair of topological types of $G$-coverings, thus giving a surjective map $ \mathcal T^{r,s}(G)\twoheadrightarrow \mathcal T^r(G)\times \mathcal T^s(G)$. By Proposition  \ref{prop: F_constants_on_orbits} and Theorem \ref{thm: bij_DrGmodulo_TrG} we obtain  the following commutative diagram
\begin{equation}\label{eq: defn_map_pi}
	\begin{tikzcd}
		{\mathcal T^{r,s}(G)} && {\frac{\mathcal D^r(G)\times \mathcal D^s(G)}{\Aut(G)\times \mathcal B_r\times \mathcal B_s}} \\
		\\
		{\mathcal T^r(G)\times \mathcal T^s(G)} && {\frac{\mathcal D^r(G)}{\Aut(G)\times \mathcal B_r}\times \frac{\mathcal D^s(G)}{\Aut(G)\times \mathcal B_s}}
		\arrow[leftrightarrow, from=1-1, to=1-3]
		\arrow[two heads, from=1-1, to=3-1]
		\arrow[leftrightarrow, from=3-1, to=3-3]
		\arrow["\pi", dashed, two heads, from=1-3, to=3-3]
	\end{tikzcd}
\end{equation}
Here $\pi$ is defined as the only map making the diagram commutative.
Such $\pi$ sends the class of a pair of spherical systems of generators $[V_1, V_2]$ to the pair of classes $([V_1], [V_2])$. 

We are going to find the inverse image of each point $([V_1],[V_2])$ by $\pi$, which translates in determining each family of product-quotient surfaces afforded by the pair of topological types of $G$-coverings, the first given by $[V_1]$, and the second by $[V_2]$.
\begin{definition}
	Let $V$ be a spherical system of generators of length $r$ of a finite group $G$. The group of automorphisms of \textit{braid type} on $V$  is the following subgroup of $\Aut(G)$
	\[
	\mathcal B\Aut(G, V):=\lbrace \varphi \in \Aut(G)\colon \exists \  \sigma \in \mathcal{B}_r\ \textrm{such that }\varphi \cdot V=\sigma\cdot V  \rbrace.
	\]
\end{definition}
Since the action of an automorphism of $G$ commutes with the action of a braid on a spherical system of generators, then is immediate to see $\mathcal B\Aut(G, V)$ is subgroup of $\Aut(G)$: 
given $\varphi_1,\varphi_2 \in \mathcal B\Aut(G, V)$, then 
\[
\left(\varphi_1\circ \varphi_2^{-1}\right) \cdot V=\varphi_1 (\sigma_2^{-1} \cdot V)=\sigma_2^{-1} \cdot (\varphi_1\cdot V)=(\sigma_2^{-1}\sigma_1) \cdot V
\]
for some $\sigma_1,\sigma_2 \in  \mathcal{B}_r$. Thus $\varphi_1\circ \varphi_2^{-1} \in \mathcal B\Aut(G, V)$. 
\begin{remark}\label{rem: how_change_V_if_change_ssg_on_the _orbit}
	If you replace $V$ by $V'$ on its $\Aut(G)\times \mathcal{B}_r$-orbit, let us say $V':=(\Psi, \sigma)\cdot V$, then the subgroup $\mathcal B\Aut(G, V')$ is conjugate to $\mathcal B\Aut(G, V)$:
	\[
	\mathcal B\Aut(G, V')=\Psi \circ \mathcal B\Aut(G, V)\circ \Psi^{-1}.
	\]
	Note that $\Psi\in \mathcal B\Aut(G, V)$ implies $\mathcal B\Aut(G, V')=\mathcal B\Aut(G, V)$.
\end{remark}
\begin{definition}
	Let $V_1$ and $V_2$ be a pair of spherical systems of generators of $G$. We will say that two automorphisms $\Phi, \Psi\in \Aut(G)$ are $(V_1, V_2)-$related, and we will write 
	\[
	\Phi \sim_{V_1, V_2}\Psi
	\]
	if there exist $\varphi_1\in \mathcal B\Aut(G, V_1),  \varphi_2\in \mathcal B\Aut(G, V_2)$ such that 
	\[
	\Psi= \varphi_1\circ \Phi\circ \varphi_2.
	\]
	The relation $\sim_{V_1, V_2}$ is clearly an equivalence relation on $\Aut(G)$. We denote by $Q\Aut(G)_{V_1, V_2}$ the quotient of $\Aut(G)$ by $\sim_{V_1, V_2}$.
	
	In other words $Q\Aut(G)_{V_1, V_2}$ is the set of double cosets 
	\[
	Q\Aut(G)_{V_1, V_2}= \mathcal B\Aut(G, V_1)\backslash \Aut(G)/  \mathcal B\Aut(G, V_2).
	\]
\end{definition} 
	\begin{remark}\label{rem: how_change_relatin_sim_V1V2}
		Replacing $V_1$ and $V_2$ by two spherical systems of generators in the same orbits, namely $V_1'=(\Psi_1, \sigma_1)\cdot V_1$ and $V_2'=(\Psi_2, \sigma_2)\cdot V_2$, then by Remark \ref{rem: how_change_V_if_change_ssg_on_the _orbit} we have 
		\[
		\Phi\sim_{V_1, V_2}\Psi \iff \Psi_1\circ \Phi\circ \Psi_2^{-1}\sim_{V_1', V_2'} \Psi_1\circ \Psi \circ \Psi_2^{-1}.
		\]
		Moreover, the bijection 
		$\Phi\mapsto \Psi_1\circ \Phi\circ \Psi_2^{-1}$ induces a bijection among the quotients 
		\begin{equation}\label{eq: bij_amonog_quot_QAutG}
			Q\Aut(G)_{V_1, V_2} \leftrightarrow Q\Aut(G)_{V_1', V_2'}, \qquad [\Phi]\mapsto [\Psi_1\circ \Phi\circ \Psi_2^{-1}]
		\end{equation}
		that only depends on $V_1,V_2,V_1'.V_2'$ and not on the choice of $\Psi_1,\Psi_2$.
	\end{remark}
We can finally state and prove the main theorem of this section:
	\begin{theorem}\label{thm: comp_inverse_image}
	We consider the map $\pi$ defined at \eqref{eq: defn_map_pi}. Let us fix a point $x\in \frac{\mathcal D^r(G)}{Aut(G)\times \mathcal B_r}\times \frac{\mathcal D^s(G)}{Aut(G)\times \mathcal B_s}$, and let us choose a pair of spherical systems of generators $V_1$ and $V_2$ such that $x=([V_1], [V_2])$. The following hold:
	\begin{enumerate}
		\item given $\Phi\in \Aut(G)$, then 
		\[
		[V_1, \Phi\cdot V_2]\in \frac{\mathcal D^r(G)\times \mathcal D^s(G)}{\Aut(G)\times \mathcal B_r\times \mathcal B_s}
		\]
		depends only by class of $\Phi$ in $Q\Aut(G)_{V_1, V_2}$.
		\item 
		The map 
		\begin{equation}\label{eq: bij_among_QAutG_and_fibre}
			\begin{split}
				Q\Aut(G)_{V_1, V_2} & \longrightarrow \pi^{-1}(x) \\
				[\Phi] & \mapsto [V_1, \Phi\cdot V_2]
			\end{split}
		\end{equation}
		is bijective. In particular, $\vert \pi^{-1}(x)\vert =\vert Q\Aut(G)_{V_1, V_2}\vert$. 
		\item If we replace $V_1$ by $V_1'$ in the same $\Aut(G)\times\mathcal B_r$-orbit, and $V_2$ by $V_2'$ in the same $\Aut(G)\times \mathcal B_s$-orbit, then the bijective maps in  \eqref{eq: bij_amonog_quot_QAutG} and \eqref{eq: bij_among_QAutG_and_fibre} form a commutative triangle
		\[\begin{tikzcd}
			{Q\Aut(G)_{V_1',V_2'}} \\
			&& {\pi^{-1}(x)} \\
			{Q\Aut(G)_{V_1,V_2}}
			\arrow[leftrightarrow, from=1-1, to=3-1]
			\arrow[leftrightarrow, from=1-1, to=2-3]
			\arrow[leftrightarrow, from=3-1, to=2-3]
		\end{tikzcd}\]
	\end{enumerate}
\end{theorem}
\begin{proof}
	\begin{enumerate}
		\item Let us consider an automorphism $\Phi'=\varphi_1\circ \Phi\circ \varphi_2$ in the same class of $\Phi$ in $Q\Aut(G)_{V_1, V_2}$, where $\varphi_1\in \mathcal B\Aut(G, V_1)$ and $\varphi_2\in \mathcal B\Aut(G, V_2)$. Then 
		\[
		\begin{split}
			[V_1, \Phi'\cdot V_2] & =[V_1, \left(\varphi_1\circ \Phi\circ \varphi_2\right)V_2] \\
			& =[\varphi_1^{-1}\cdot V_1, \left(\Phi\circ \varphi_2\right)\cdot V_2] \\
			&=[\sigma_1^{-1}\cdot V_1, \Phi\cdot (\sigma_2\cdot V_2)]\\
			& =[\sigma_1^{-1}\cdot V_1, \sigma_2\cdot \left(\Phi\cdot V_2\right)]=[V_1, \Phi\cdot V_2].
		\end{split}
		\]
		\item Point $1.$ proves that the map \ref{eq: bij_among_QAutG_and_fibre} is well-defined. Let us consider an element $[V_1', V_2']\in \pi^{-1}(x)$, hence $V_1'$ is in the same orbit of $V_1$ and $V_2'$ is in the same orbit of $V_2$. We write 
		\[
		V_1'=(\Psi_1, \sigma_1)\cdot V_1 \qquad \makebox{and} \qquad V_2'=(\Psi_2, \sigma_2)\cdot V_2,
		\]
		where $(\Psi_1, \sigma_1)\in \Aut(G)\times \mathcal B_r$, and $(\Psi_2, \sigma_2)\in \Aut(G)\times \mathcal B_s$. Then
		\[
		[V_1', V_2']=[\Psi_1\cdot V_1, \Psi_2\cdot V_2]=[V_1, \left(\Psi_1^{-1}\circ \Psi_2\right)\cdot V_2].
		\]
		This proves \eqref{eq: bij_among_QAutG_and_fibre} is surjective. 
		
		Let us consider $[\Phi_1]$ and $[\Phi_2]$ in $Q\Aut(G)_{V_1, V_2}$ such that 
		\[
		[V_1, \Phi_2\cdot V_2]=[V_1, \Phi_1\cdot V_2].
		\] 
		We are going to show that $[\Phi_2]=[\Phi_1]$. Since $(V_1, \Phi_2\cdot V_2)$ and $(V_1, \Phi_1\cdot V_2)$ share the same orbit, then there exists $(\Psi, \sigma_1, \sigma_2)\in \Aut(G)\times \mathcal B_r\times \mathcal B_s$ such that $(V_1, \Phi_2\cdot V_2)=(\Psi, \sigma_1, \sigma_2)\cdot (V_1, \Phi_1\cdot V_2)$. Then we have
		\[
		\Psi \cdot V_1=\sigma_1^{-1}\cdot V_1 \qquad \makebox{and} \qquad \left(\Phi_1^{-1}\circ\Psi^{-1}\circ \Phi_2\right)\cdot V_2=\sigma_2\cdot V_2.
		\]
		Therefore, $\varphi_1:=\Psi\in \mathcal B\Aut(G,V_1)$ and $\varphi_2:=\Phi_1^{-1}\circ\Psi^{-1}\circ \Phi_2\in \mathcal B\Aut(G,V_2)$. Finally, we have 
		\[
		\Phi_2= \varphi_1\circ \Phi_1\circ \varphi_2,
		\]
		which proves $[\Phi_2]=[\Phi_1]$, and so that \eqref{eq: bij_among_QAutG_and_fibre} is injective. 
		\item It is an immediate consequence from the definition of the map \eqref{eq: bij_amonog_quot_QAutG}. 
	\end{enumerate}
\end{proof}
Theorem \ref{thm: comp_inverse_image} gives not only a perfect enumeration  of the families of regular product-quotient surfaces corresponding to an {\bf ordered} pair of topological types of $G$-coverings of the projective line $(C_1,\phi_1)$ and $(C_2,\phi_2)$ but also how to realize these families. Indeed, given $\Psi\in \Aut(G)$, then 
$(C_1, \phi_1)$ and $(C_2,\phi_2\circ \Psi^{-1})$ define a product-quotient surface realizing an irreducible family. Theorem \ref{thm: comp_inverse_image} translates as each family given by topological types of $C_1$ and $C_2$ are obtained in this way via an automorphism of $\Aut(G)$. Furthermore, two automorphisms $\Psi_1$ and $\Psi_2$ define the same family if they are $(V_1,V_2)-$related, or equivalently if their class in $Q\Aut(G)_{V_1,V_2}$ is the same. 
\\
Thus, all families may be realized by a pair $(C_1, \phi_1)$ and $(C_2,\phi_2\circ \Psi^{-1})$ via a automorphism representative $\Psi$ for each class in $Q\Aut(G)_{V_1,V_2}$.

We consider \textbf{ordered} pairs of topological types since the Remark \ref{rem: exchanging_factors}, where we have observed that exchanging $C_1$ and $C_2$ defines an involution on $\bigcup \mathcal T^{r,s}(G)$ connecting isomorphic families. 

If we are interested in counting the families given by two different topological types of $G$-coverings, then it is sufficient to choose an order of them and then apply Theorem \ref{thm: comp_inverse_image}. 
\\
However, to enumerate the families of product-quotient surfaces associated to twice the same topological type we need to study how the exchange of the factors acts on $Q\Aut(G)_{V,V}$. 
\begin{proposition}
	The exchange of the factors acts on $Q\Aut(G)_{V,V}$ as the involution 
	\[
	Q\Aut(G)_{V,V}\to Q\Aut(G)_{V,V}, \qquad [\Phi]\mapsto [\Phi^{-1}].
	\]
\end{proposition}
\begin{proof}
	The exchange of the factors is a map from $\pi^{-1}([V], [V])$ to itself sending each $[V, \Phi\cdot V]$ to $[\Phi \cdot V, V]=[V, \Phi^{-1}\cdot V]$. 
\end{proof}
\begin{corollary}\label{cor: card_set_of_families}
	Let $C_1$ and $C_2$ be two $G$-coverings of $\mathbb P^1$ and let $V_1$ and $V_2$ be spherical systems of generators of them respectively. Then the cardinality of the set of families of product-quotient surfaces given by the topological types of $C_1$ and $C_2$ is equal to 
	\begin{enumerate}
		\item the cardinality of $Q\Aut(G)_{V_1, V_2}$, if $C_1$ and $C_2$ are not topological equivalent;
		\item the cardinality of $Q\Aut(G)_{V_1, V_1}/\left(\Phi\mapsto \Phi^{-1}\right)$, if $C_1$ and $C_2$ are  topological equivalent. 
	\end{enumerate}
\end{corollary}
Let us give an example how we use Theorem \ref{thm: comp_inverse_image} and Corollary \ref{cor: card_set_of_families}: 
\begin{example}\label{ex: PQ_group_S3xZ3square}
	Let $G=S_3\times \mathbb Z_3^2$. We are going to compute all regular product-quotient surfaces with quotient model $\left(C_1\times C_2\right)/G$ where the natural fibrations $ \lambda_1\colon C_1\to \mathbb P^1$ and $\lambda_2\colon C_2\to \mathbb P^1$ are both ramifying  over three points. 
	
	From Example \ref{ex: ssg_S3xZ3square_ramified_over_3_points}, then $C_1$ and $C_2$ are topological equivalent described by the algebraic data 
	\[
	V:=[(\tau, (1,0)), (\tau\sigma, (0,1)), (\sigma^2, (2, 2))].
	\]
	We need to compute the subgroup $\mathcal B\Aut(G, V)\leq \Aut(S_3\times \mathbb Z_3^2)$.
	
	Firstly we note that
	\[\Aut(S_3\times \mathbb Z_3^2)\cong \Aut(S_3)\times \GL_2(\mathbb Z_3).
	\]
	
	Hence every element of $\mathcal B\Aut(G, V)$ can be written as a pair $(\Psi, M)$, where $\Psi\in \Aut(S_3)$, and $M\in \GL_2(\mathbb Z_p)$.

	The action of $\mathcal{B}_3$ on $[(1,0), (0,1), (2, 2)]$ permutes its entries, since $\mathbb Z_3^2$ is abelian. Therefore, the automorphisms $M\in \GL_2(\mathbb Z_3)$ of braid type on it are those permuting its entries. Such automorphisms belong to the subgroup $\langle M_1, M_2\rangle \cong S_3$ generated by 
	\[
	M_1:=\begin{pmatrix}
		0 & 1\\
		1 & 0
	\end{pmatrix} \qquad M_2:=\begin{pmatrix}
		2 & 0\\
		2 & 1
	\end{pmatrix}.
	\]
	Let $(\Psi, M)$ be of braid type on $V$, and let $\eta$ be a braid in $\mathcal B_3$ such that $(\Psi, M)\cdot V=\eta\cdot V$.
	We observe that the signature of $V$ is $[6,6, 9]$: since the third number is different from the other two, and the automorphisms send elements in elements of the same order, then the permutation image of $\eta$ in $S_3$ fix the number three. This implies that $M$ fixes $(2,2)$, so $M \in \left\langle M_1 \right\rangle \cong \mathbb Z_2$.	 
	Therefore, 
	\[
	\mathcal B\Aut(G,V) \leq \Aut(S_3)\times \langle M_1\rangle\cong S_3\times \mathbb Z_2.
	\]
	Let us choose two generators of $\Aut(S_3)$: let $\Psi_1$ be the inner automorphism given by $\tau$ and let $\Psi_2$ be the inner automorphism of $\sigma^2$. 
	We observe that $(\Psi_1, \Id)$ and $(\Psi_2\circ \Psi_1, M_1)$ are of braid type on $V$, since they act on $V$ respectively as the braids $\sigma_1\sigma_2^2\sigma_1$ and $\sigma_1$. Since they generate the whole $ \Aut(S_3)\times \langle M_1\rangle$ then
	\[
	\mathcal B\Aut(G,V)= \Aut(S_3)\times \langle M_1\rangle\cong S_3\times \mathbb Z_2.
	\]
	Now we can compute $Q\Aut(S_3\times \mathbb Z_3^2)_{V,V}$, which as observed is the set of double cosets 
	\[
	Q\Aut(S_3\times \mathbb Z_3^2)_{V,V}=_{\mathcal B\Aut(G,V)}\backslash^{\left(\Aut(S_3)\times \GL_2 (\mathbb Z_3)\right)}/_{\mathcal B\Aut(G,V)}.
	\]
	Since $\mathcal B\Aut(G,V)=\Aut(S_3)\times \langle M_1\rangle$ contains the subgroup $\Aut(S_3)\times \lbrace 1\rbrace$, which is normal in $\Aut(S_3)\times \GL_2(\mathbb Z_3)$, then we have the following natural identification
	\begin{equation}\label{eq: identific_QAut(S3timesZpsquare)}
		Q\Aut(S_3\times \mathbb Z_3^2)_{V,V}\cong\  _{\langle M_1\rangle}\backslash^{\GL_2(\mathbb Z_3)}/_{\langle M_1\rangle}.
	\end{equation}
	More precisely, the correspondence sends $ [(\Id_{S_3},A)]\leftrightarrow [A]$. 
	
	From \eqref{eq: defn_map_pi} and Theorem \ref{thm: comp_inverse_image} we can conclude that 
	\[
	\mathcal T^{3,3}(S_3\times \mathbb Z_p^2)\cong Q\Aut(G)_{V,V}\cong \ _{\langle M_1\rangle}\backslash^{\GL_2(\mathbb Z_3)}/_{\langle M_1\rangle}.
	\]
	However, we are majorly interested to find the set of families of product-quotient surfaces given by the pair $V$, $V$. As proved in the Corollary \ref{cor: card_set_of_families}, it is sufficient to determine 
	\[
	Q\Aut(G)_{V_1, V_1}/\left(\Phi\mapsto \Phi^{-1}\right).
	\]
	
	This is the quotient of $\GL_2(\mathbb Z_p)$ by the simultaneous action of the three involutions $A\mapsto M_1A$, $A\mapsto AM_1$ and $A\mapsto A^{-1}$. These involutions generate a group of order $8$ isomorphic to a dihedral group. Hence 
	\begin{equation}\label{eq: T33(S3timesZpsquare)_bij_example}
		Q\Aut(G)_{V_1, V_1}/\left(\Phi\mapsto \Phi^{-1}\right)\cong \GL_2(\mathbb Z_3)/D_4.
	\end{equation}
	We have proved that families of regular product-quotient surfaces with quotient model $\left(C_1\times C_2\right)/G$ where $ \lambda_1\colon C_1\to \mathbb P^1$ and $\lambda_2\colon C_2\to \mathbb P^1$ are both ramifying over three points are in bijection with $\GL_2(\mathbb Z_3)/D_4$, a set of cardinality $10$. More precisely, these families are realized as follows: we consider two copies $(C_1,\phi)$, $(C_2,\phi)$ of the same curve $(C,\phi)$ defined in the Example \ref{ex: ssg_S3xZ3square_ramified_over_3_points} which is described by the algebraic data $V$. This pair of curves define a product-quotient surface realizing a first family. All the other families are realized by product-quotient surfaces each defined by a pair $(C_1,\phi)$ and $(C_2,\phi\circ (\Id, A^{-1}))$, where $A$ is a representative of a class of $\GL_2(\mathbb Z_3)/D_4$. 
\end{example}

\section{Finiteness of the classification problem 
}\label{sec: class_algorithm}
In this section we follow step-by-step the same arguments of \cite{BP12} and generalize the results of \cite[Prop. 1.14]{BP12} by removing the assumption $\chi=1$ there. 
\\
As a consequence of this, we describe an algorithm that produces for any fixed pair of natural integers $K^2$ and $\chi$ all regular product-quotient surfaces $S$ of general type with self-intersection $K^2_S=K^2$ and Euler characteristic $\chi(\mathcal O_S)=\chi$. 
\vspace{0,3cm}

Let $C_1$ and $C_2$ be two Riemann surfaces of respective genera $g_1,g_2\geq 2$ and let $G$ be a finite group acting faithfully on both of them. We consider the diagonal action of $G$ on the product $C_1\times C_2$, which gives a product-quotient surface $S$, the minimal resolution of singularities of the quotient model $X:=\left(C_1\times C_2\right)/G$. 

The singular points of the quotient model $X$ are images of points in $C_1\times C_2$ having non-trivial stabilizer by the diagonal action of $G$. Hence, $X$ has only finitely many singular points which are cyclic quotient singularities.
\\
A cyclic quotient singularity of type $\frac{1}{n}(1,a)$ is the singular point realized as the quotient of $\mathbb C^2$ by the action of the diagonal linear isomorphism of eigenvalues $\zeta_n$ and $\zeta_n^a$, with $\gcd(n,a)=1$. 
\\
We can attach to $X$ the so-called \textit{basket} of singularities:
\begin{definition}\cite[Def. 1.2]{BP12}
	A representation of the basket of singularities of $X$ is a multiset 
	\[
	\mathcal B(X):=\left\lbrace \lambda \times \left(\frac{1}{n}(1,a)\right): X \text{ has exactly } \lambda \text{ singularities of type } \frac{1}{n}(1,a)\right\rbrace.
	\]
\end{definition}
We use the word "representation" since $X$ may have several representatives of its basket, essentially since a singularity of type $\frac{1}{n}(1,a)$ is isomorphic to a singularity of type $\frac{1}{n}(1,a')$, where either $a=a'$ or $aa'\equiv 1$ mod $n$. 

In \cite{BP12} the authors used the minimal resolution of a cyclic quotient singularity as \textit{Hirzebruch-Jung string} to compute the correction terms to the self-intersection  and the topological characteristic of the product-quotient surface $S$. 
\\
We need to remind these these correction terms.
\begin{definition}(\cite[Definition 1.5]{BP12})
	Let $x$ be a singularity of type $\frac{1}{n}(1,a)$ with $\gcd(n,a)=1$, and let $1\leq a'<n$ be the inverse of $a$ modulo $n$, $a'=a^{-1}$. Write $\frac{n}{a}$ as a continued fraction 
	\[
	\frac{n}{a}=b_1-\frac{1}{b_2-\frac{1}{
			b_3-\dots}}=[b_1,\dots, b_l]
	\]
	We define the following correction terms
	\begin{itemize}
		\item $k_x:=k(\frac{1}{n}(1,a))=-2+\frac{2+a+a'}{n}+\sum_{i=1}^l(b_i-2)\geq 0$;
		\item $e_x:=e(\frac{1}{n}(1,a))=l+1-\frac{1}{n}\geq 0$;
		\item $B_x:=2e_x+k_x$. 
	\end{itemize}
	Let $\mathcal{B}$ be the basket of singularities of $X$. Then we denote by 
	\[
	k(\mathcal B):=\sum_x k_x, \qquad e(\mathcal B):=\sum_x e_x, \qquad B(\mathcal{B}):=\sum_x B_x.
	\]
\end{definition}
\begin{theorem}(\cite[Prop. 1.6 and Cor. 1.7]{BP12})\label{thm: invariants_PQ}
	Let  $\rho \colon S\to X$ be the minimal resolution of singularities of $X=(C_1\times C_2)/G$. 
	Then the self-intersection of the canonical divisor of $S$ and its topological Euler characteristic are equal to 
	\[
	K_S^2 =\frac{8(g_1-1)(g_2-1)}{\vert G\vert}-k(\mathcal{B}),\qquad \makebox{and} \qquad  	e(S) = \frac{4(g_1-1)(g_2-1)}{\vert G\vert}+ e(\mathcal B).
	\]
	Furthermore, it holds 
	\[
	K^2_S=8\chi(\mathcal O_S)-\frac{1}{3}B(\mathcal{B}).
	\]
\end{theorem} 
From now on we shall restrict to product-quotient surfaces $S$ of general type that are regular, namely $C_i/G\cong \mathbb P^1$. 
\\
As supposed from the previous Subsection \ref{subsec: families_of_product-quotient_surfaces}, we shall describe $S$ in a pure algebraic way by using a pair of spherical systems of generators 
\[
[g_1,\dots, g_r] \qquad \makebox{and} \qquad [h_1,\dots, h_s]
\]
of the pair of $G$-coverings $C_1$ and $C_2$ of $\mathbb P^1$. 
\begin{remark}
	In \cite[Subsec. 1.2]{BP12} is shown how to determine the number of cyclic quotient singularities (and their types) of the quotient model $X=\left(C_1\times C_2\right)/G$ from the algebraic data of a pair of spherical systems of generators. 
	
	In this way, we read the basket of singularities of $S$ form the pair $[g_1, \dots, g_r]$ and $[h_1,\dots, h_s]$, and then determine the invariants $K^2_S$ and $\chi(\mathcal O_S)$ by using Theorem \ref{thm: invariants_PQ}. 
\end{remark}
Finally, we states the preliminaries to extend \cite[Prop. 1.14]{BP12} to any natural integer $\chi$. 
\begin{definition}
	Fix an $r$-tuple of natural numbers $t:=[m_1, \dots, m_r]$, and a basket of singularities $\mathcal B$. Then we associate to these the following numbers:
	\[
	\begin{split}
		\Theta(t) & :=-2+\sum_{i=1}^r\left(1-\frac{1}{m_i}\right); \\
		\alpha(t, \mathcal{B}) &:= \frac{12\chi+ k(\mathcal{B})-e(\mathcal{B})}{6\Theta(t)}.
	\end{split}
	\]
\end{definition}
We recall
\begin{definition}
	The minimal positive integer $I_x$ such that $I_xK_X$ is Cartier in $x$ is called the \textit{index} of the singularity $x$. 
	
	The index of $X$ is the minimal positive integer $I$ such that $I$ is Cartier. In particular, $I=\text{lcm}_{x\in \Sing X}I_x$.  
\end{definition}
It is well known that the index of a cyclic quotient singularity $\frac{1}{n}(1,a)$ is 
\[
I_x=\frac{n}{\gcd(n, a+1)}. 
\]
By Lemma \cite[Lem. 1.10]{BP12}, fixed a pair of natural integers $(K^2,\chi)$, there are only finitely many basket of singularities $\mathcal B$ for which there exists a product-quotient surface $S$ with invariants $K_S^2=K^2$, $\chi(\mathcal O_S)=\chi$, and having a quotient model with a representation of the basket of singularities equal to $\mathcal B$.

We need to extend \cite[Prop. 1.14]{BP12} to any natural integer $\chi$ to bound, for fixed $K^2$, $\chi$, and $\mathcal B$, the possibilities for 
\begin{itemize}
	\item $\vert G\vert$;
	\item $t_1:=[m_1, \dots, m_r]$,
	\item $t_2:=[n_1, \dots, n_s]$,
\end{itemize}
of a product-quotient surface $S$ with $K_S^2=K^2$, $\chi(S)=\chi$, and basket of singularities of the quotient model $X=\left(C_1\times C_2\right)/G$ equal to $\mathcal B$ such that the pair of spherical systems of generators of $C_1$ and $C_2$ have respectively signature $t_1$ and $t_2$. 
\begin{proposition}\label{prop: fixed_Ksquare_and_chi_get_finitely_many}
	Fix $(K^2, \chi)\in \mathbb Z\times \mathbb Z$, and fix a possible basket of singularities $\mathcal B$ for $(K^2, \chi)$. Let $S$ be a product-quotient surface $S$ of general type such that 
	\begin{itemize}
		\item[i.] $K_S^2=K^2$;
		\item[ii.] $\chi(S)=\chi$;
		\item[iii.] the basket of singularities of the quotient model $X=\left(C_1\times C_2\right)/G$ equals $\mathcal B$. 
	\end{itemize}
	Then 
	\begin{itemize}
		\item[a)] $g(C_1)=\alpha(t_2, \mathcal B)+1$, $g(C_2)=\alpha(t_1, \mathcal B)+1$;
		\item[b)] $\vert G\vert =\frac{8\alpha(t_1, \mathcal B)\alpha(t_2, \mathcal B)}{K^2+k(\mathcal{B})}$;
		\item[c)] $r,s \leq \frac{K^2+k(\mathcal{B})}{2}+4$;
		\item[d)] $m_i$ divides $2\alpha(t_1, \mathcal B)I$, $n_j$ divides $2\alpha(t_2, \mathcal B)I$;
		\item[e)] there are at most $\vert \mathcal B\vert /2$ indices $i$ such that $m_i$ does not divide $\alpha(t_1, \mathcal B)$, and similarly for the $n_j$;
		\item[f)] $m_i\leq \frac{1+I\frac{K^2+k(\mathcal{B})}{2}}{f(t_1)}$, $n_i\leq \frac{1+I\frac{K^2+k(\mathcal{B})}{2}}{f(t_2)}$, where $I$ is the index of $X$, and $f(t_1):=\max(\frac{1}{6}, \frac{r-3}{2})$, $f(t_2):=\max(\frac{1}{6}, \frac{s-3}{2})$; 
		\item[g)] except for at most $\vert \mathcal B\vert/2$ indices $i$, the sharper inequality $ \ \ \ \ \ \ \ \ \ \ \ $ $m_i\leq \frac{1+\frac{K^2+k(\mathcal{B})}{4}}{f(t_1)}$ holds, and similarly for the $n_j$.
	\end{itemize}
\end{proposition}
\begin{remark}
	Note that $b)$ shows $t_1$ and $t_2$ determines the order of $G$. $c)$ and $f)$ implies there are only finitely many possibilities for the signatures $t_1, t_2$. Instead, $d), e)$, and $g)$ are strictly necessary to obtain an efficient algorithm. 
\end{remark}
\begin{proof}
	It is sufficient to prove $a)$ since the other points have the same proof as \cite{BP12}. 
	From Theorem \ref{thm: invariants_PQ}, then 
	\[
	\begin{split}
		\Theta(t_1)\alpha(t_1,\mathcal B)= & \frac{1}{2}\frac{24\chi+2k(\mathcal{B})-2e(\mathcal{B})}{6}  
		=\frac{24\chi-B(\mathcal{B})+3k(\mathcal B)}{6} \\ & =3\frac{8\chi -\frac{B(\mathcal B)}{3}+k(\mathcal B)}{12}  =\frac{K^2+k(\mathcal B)}{4},
	\end{split}
	\]
	and then by the Theorem \ref{thm: invariants_PQ} and Hurwitz' formula \ref{eq: Hurwitz_formula}, we have 
	\begin{multline*}
		\alpha(t_1, \mathcal B)=\frac{K^2+k(\mathcal B)}{4\Theta(t_1)}=\frac{8(g(C_1)-1)(g(C_2)-1)}{4\vert G\vert \left(-2+\sum_{i=1}^r\left(1-\frac{1}{m_i}\right)\right)}\\ =\frac{8(g(C_1)-1)(g(C_2)-1)}{4(2g(C_1)-2)}.
	\end{multline*}
\end{proof}
\subsection{Description of the classification algorithm}\label{subsec: descr_implem_class_algorithm}
Fixed a pair $(K^2, \chi)\in \mathbb N\times \mathbb N$, the next goal is to write a MAGMA script to find all minimal regular surfaces $S$ of general type with $K^2_S=K^2$, and $\chi(S)=\chi$, which are product-quotient surfaces. 
A commented version of the MAGMA code is available here:
\\
\url{https://github.com/Fefe9696/PQ_Surfaces_with_fixed_Ksquare_chi}

We describe here the strategy, and explain how the most important scripts work. Most of the scripts are modification of those in \cite{BP12}. Since those scripts were written under the assumption $\chi=1$, we generalize all of  them to allow any value of $\chi$. In the Introduction \ref{sec: Introduction} of the present paper we indicate the other main improvements we did. 
\\
We fix a couple $(K^2, \chi)$. Note that by minimality of $S$, and by Theorem \ref{thm: invariants_PQ}, then $K^2\in \lbrace 1, \dots, 8\chi\rbrace $, and the case $K^2=8\chi$ corresponds to surfaces isogenous to a product. 
\\
{\bf Step 1:} The script {\bf Baskets} lists all the \textit{possible basket of singularities} $\mathcal B$ for $(K^2, \chi)$. Indeed, there are only finitely many of them by \cite[Lem. 1.10]{BP12}. The input is $B(\mathcal B)=3(8\chi-K^2)$, so to get for instance all baskets for $(K^2, \chi)=(28, 4)$, we need $\textit{Basket}(12)$. 
\\
{\bf Step 2:} From the Proposition \ref{prop: fixed_Ksquare_and_chi_get_finitely_many} once we know the basket of singularities of $X=\left(C_1\times C_2\right)/G$, then there are finitely many possible signatures of a pair of spherical systems of generators of $C_1$ and $C_2$. {\bf ListOfTypes} computes them using the inequalities in the Proposition \ref{prop: fixed_Ksquare_and_chi_get_finitely_many}. Here the input is $K^2$, and $\chi$, so \textit{ListOfTypes} first computes \textit{Baskets($3(8\chi-K^2)$)}, and then computes for each basket all numerically compatible signatures. The output is a list of pairs, the first element of each pair being a basket, and the second element being the list of all signatures compatible with that basket.
\\ 
{\bf Step 3:} Every surface produces two signatures, one for each curve $C_i$, both compatible with the basket of singularities of $X$; if we know the signatures and the basket, then Proposition \ref{prop: fixed_Ksquare_and_chi_get_finitely_many} $b)$ tells us the order of $G$. {\bf ListGroups}, whose input is $K^2$, and $\chi$, first computes \textit{ListOfTypes$(K^2, \chi)$}. Then for each pair of signatures in the output, it calculates the order of the group. Next it searches for the groups of given order  which admit appropriate spherical systems of generators corresponding to both signatures. Here we use the database in  \cite{CGP23} if we are in one of the cases classified there, otherwise we use the function \textit{FindGenerators} developed in the work \cite{CGP23}.
\\
For each affirmative answer, it stores the triple (basket, pair of signatures, group) in a list, which is the main output. 

The script has some shortcuts:
\begin{itemize}
	\item Let $t_1$ and $t_2$  be the pair of signatures and let $\mathbb T(t_1)$ and $\mathbb T(t_2)$ be their respective orbifold groups (see the Remark \ref{rem: orbifold_morphism}). 
	Then the order of the abelianization $G^{ab}$  of $G$ has to divide the order the abelianization of  $\mathbb T(t_1)$ and $\mathbb T(t_2)$:
	\begin{equation}\label{eq: cond_abelianization_orders}
		\vert G^{ab}\vert \quad \makebox{divides} \quad  \vert \mathbb T(t_1)^{ab}\vert, \vert\mathbb T(t_2)^{ab}\vert .
	\end{equation}
	Indeed, the orbifold (surjective) homomorphisms $\mathbb T(t_1)\to G$ and $\mathbb T(t_2)\to G$ extend to surjective homomorphisms 
	\[
	\mathbb T(t_1)^{ab}  \to G^{ab}, \qquad \mathbb T(t_2)^{ab} \to G^{ab}.
	\]
	Hence  \textit{ListGroups} checks first if  $G$ satisfies \eqref{eq: cond_abelianization_orders}: if not, this case not occur. 
	\item If the pair of signatures $t_1$ and $t_2$ admits orbifold groups 
	$\mathbb T(t_1)$ and $\mathbb T(t_2)$ such that the orders of their abelianization are coprime numbers, then $G$ is forced to be a perfect group.
	This follows directly from the condition \eqref{eq: cond_abelianization_orders}.
	
	MAGMA knows all perfect groups of order $\leq 50000$, and then \textit{ListGroups} checks first if there are perfect groups of the right order: if not, this case can not occur. 
	\item If:
	\begin{itemize}
		\item[-] either the expected order of the group is $1024$ or bigger than $2000$, since MAGMA does not have a list of the finite groups of this order;
		\item[-] or the order is a number as \textit{e.g.}, $1728$, where there are too many isomorphism classes of groups;
	\end{itemize}
	then \textit{ListGroups} just stores these cases in  a list, secondary output of the script. These "exceptional" cases have to be considered separately. 
\end{itemize}
{\bf Step 4:} The basket of singularities of a surface described by a couple of spherical systems $[g_1, \dots, g_r]$ and $[h_1, \dots, h_s]$ depends only by the conjugacy classes of $g_i$ and $h_j$, from \cite[Rem. 4.7.2]{fede3}.
{\bf ExistingSurfaces} runs on the output of \textit{ListGroups}($K^2, \chi$), and throws away all triples giving rise only to surfaces whose singularities do not correspond to the basket. 
\\
{\bf Step 5:} Each triple (basket, pair of signatures, group) in the output \textit{ExistingSurfaces}$(K^2, \chi)$ gives many different pairs of compatible spherical systems of generators. On them there is the action of $ \Aut(G)\times \mathcal B_r\times \mathcal B_s$ described in Subsection \ref{subsec: families_of_product-quotient_surfaces}. Therefore, {\bf FindSurfaces} uses Theorem \ref{thm: comp_inverse_image} and Corollary \ref{cor: card_set_of_families} to pick up only one pair of spherical systems of generators for any family of product-quotient surfaces compatible with the triple (basket, pair of signatures, group).   Thus, the output is a list of (basket, sph1, sph2, group), where sph1 and sph2 are spherical systems of group compatible with pair of signatures and basket. 
\section{Regular product-quotient surfaces with $23\leq K^2\leq 32$ and $\chi=4$.}\label{sec: class_chi=4}
In this section we prove the main Theorem \ref{thm: intro_teo_classif_pg=3} presented in the introduction. 

We have run the function \textit{FindSurfaces} described in the previous Subsection \ref{subsec: descr_implem_class_algorithm} on each triple of the output of \textit{ExistingSurfaces}$(K^2,\chi)$, where $K^2\in \{23,\dots, 32\}$ and $\chi=4$. This has given all the families in tables \ref{table: class_PQ_pg=3} to \ref{table: class_PQ_pg=3_excep_minimality} of the appendix with the only exception of families $no.$ $267$ and $544$, which are the only cases occurred on those 
skipped by $ListGroups$ and stored in its secondary output. 

Thus, to prove the main Theorem \ref{thm: intro_teo_classif_pg=3} it remains to show that 
\begin{enumerate}
	\item among all the exceptional cases skipped by \textit{ListGroups}, only two cases occur, which are $no.$ $267$ and $544$;
	\item all the obtained families of tables  \ref{table: class_PQ_pg=3} to  \ref{table: class_PQ_pg=3_excep_minimality} are of general type and those on tables \ref{table: class_PQ_pg=3} to  \ref{table: class_PQ_pg=3_12} are also minimal. 
\end{enumerate}
\vspace{0,3cm}
This will be the content of the next two subsections \ref{subsec: excep_cases} and \ref{subsec: rational_curves_on_PQ}. 

\subsection{The exceptional cases}\label{subsec: excep_cases}
For each $K^2\in\{23,\dots, 32\}$, the list of cases skipped by \textit{ListGroups}($K^2,4)$ and stored in its secondary output 
can be found here:
\\
\url{https://github.com/Fefe9696/PQ_Surfaces_with_fixed_Ksquare_chi}

The main theorem of this subsection is the following: 
\begin{theorem}\label{thm: main_thm_excp_cases}
	There are exactly two groups $G$ admitting an appropriate pair of spherical systems of generators compatible with one of the triples of the secondary output of $ListGroups$($K^2,4$), for $K^2\in \{23,\dots, 32\}$: 
	\begin{table}[h]
		\centering
		\begin{tabular}{| c|c| c |c|c| l| c|}
			\hline
			$no.$ &	 $K^2_S$ & Sing($X$) & $t_1$ & $t_2$ & $G$ & $N$ \\
			\hline
			\hline
			$267$ &26&  $1/4, 1/2^2, 3/4$& $3^2,4$ &	$3^2,4$ & $G(1944,3875)$ &  2 \\ 
			$ 544 $ &24& $ 1/6, 1/2^2, 5/6$ & $ 2, 4, 6  $ & $ 2, 6, 8 $ & $ G(768, 1086051) $ & 2\\
			\hline
		\end{tabular}
	\caption{} \label{table: cases_occur_skipped_ListGroups}
\end{table}
\end{theorem}
A proof of this theorem can be found in the .txt files attached to this paper, one for each $K^2\in\{23,\dots, 32\}$. More precisely, in these files we provide a step-by-step proof of how to exclude the cases omitted by $ListGroups$ until we encounter the only two cases above that actually occur.

However, to illustrate the main strategy we have employed to exclude these cases, here we only discuss those with $K^2=32$, which already consist of a significant list of 152 cases. Therefore, we need to prove
\begin{theorem}\label{thm: skipped_cases_Ksquare_32}
	 No one of the cases skipped by $ListGroups(32, 4)$ gives a product-quotient surface $S$ with $K^2_S=32$ and $\chi(\mathcal O_S)=4$.  
\end{theorem}
\begin{proof}
	It  follows from propositions \ref{prop: exclude_groups_ord_lesseq_2000}, \ref{prop: from_1_to_18}, \ref{prop: from_19_to_62}, \ref{prop: 63_and_64}, and \ref{prop: 65} below. 
\end{proof}
The rest of this section is devoted to give a proof of the series of propositions used to prove Theorem \ref{thm: skipped_cases_Ksquare_32}. 

We use two MAGMA functions to prove these propositions and more in general Theorem \ref{thm: main_thm_excp_cases}:
\begin{itemize}
	\item \textbf{HowToExclude} takes in input a list of triples as those of the second output of  \textit{ListGroups} that have an order of the group different from $1024$ and less or equal to $2000$. 
	For each triple (basket,  $(t_1,t_2)$, $ord$) of the list it returns those groups with order $ord$ admitting a pair of spherical systems of generators of signatures $t_1$ and $t_2$. 
	This function uses such as \textit{ListGroups} the database and function \textit{FindGenerators} in \cite{CGP23}; 
	\item The function \textbf{HowToExcludePG} works similarly such as \textit{HowToExclude}. 
	Hence, it takes in input a list of triples (basket,  $(t_1,t_2)$, $ord$), where $ord$ is $\leq 50000$,  and returns those groups with order $ord$ that are perfect and admit a pair of spherical systems of generators of signatures $t_1$ and $t_2$. 
\end{itemize}
We also need the following
\begin{proposition}\label{prop: intermediate_quotients}
	Let $G$ be a finite group that admits a spherical system of generators of signature 
	$[a_1,a_2,a_3,b_1,\dots,b_k]$. Let us suppose $G$ have a normal subgroup $H$ of index a prime number $p\geq 2$ and that $p$ does not divide $b_1, \dots, b_k$. Then
	\begin{itemize}
		\item if $p$ does not divide only one among $a_1,a_2,a_3$, e.g. $p\nmid a_3$, then $H$ admits a spherical system of generators of signature $[a_1/p, a_2/p,a_3^p,b_1^p,\dots, b_k^p ]$;
		\item if $p$ divides each of $a_1,a_2,a_3$, then $H$ admits either a spherical system of generators having one of the following signatures:
		\begin{enumerate}
			\item $[a_1/p,a_2/p,a_3^p, b_1^p, \dots, b_k^p]$;
			\item $[a_1/p,a_2^p,a_3/p, b_1^p, \dots, b_k^p]$;
			\item $[a_1^p,a_2/p,a_3/p, b_1^p, \dots, b_k^p]$;
		\end{enumerate}
		or, if $p\neq 2$, it admits a generating vector of type
		\[
		\left[\frac{p-1}{2}; a_1/p,a_2/p,a_3/p,b_1^p, \dots, b_k^p\right].
		\]
		In other words, it there exists a $H$-covering of a curve of genus $\frac{p-1}{2}$ whose branch locus has ramification indices $a_1/p,a_2/p,a_3/p,b_1^p, \dots, b_k^p$. 
	\end{itemize}
\end{proposition}
	\begin{proof}
		By assumption, $G$ has a spherical system of generators $[g_1,g_2,g_3,h_1\dots, h_k]$ which defines a $G$-covering $C\to \mathbb P^1$ whose branch locus $v_1,v_2,v_3,q_1,\dots, q_k\in \mathbb P^1$ has ramification indices $a_1,a_2,a_3,b_1,\dots,b_k$ respectively. Furthermore, the existence of a normal subgroup $H$ of index $p$ gives the following triangular commutative diagram: 
		\[\begin{tikzcd}
			C \\
			{C/H} && {\mathbb P^1}
			\arrow["{/G}", from=1-1, to=2-3]
			\arrow["{/H}"', from=1-1, to=2-1]
			\arrow["{/\mathbb Z_p}"', from=2-1, to=2-3].
		\end{tikzcd}\]
		Note that $h_i\in H$ since $h_iH$ has order in $G/H\cong \mathbb Z_p$ that divides both $p$ and the order $b_i$ of $h_i$. Hence $q_1,\dots, q_k$ are not in the branch locus of $C/H\to \mathbb P^1$, which has then to ramify over at most $r\leq 3$ points with ramification indices $p$. 
		\\
		By Hurwitz formula \eqref{eq: Hurwitz_formula}, we  get 
		\begin{equation}\label{eq: Hurwitz_for_signatures}
			2g(C/H)-2=p\left(-2+r\frac{p-1}{p}\right)\implies g(C/H)=\frac{p-1}{2}(r-2).
		\end{equation}
		Hence $r$ is forced to be either equal to $2$ or $3$. If $r=2$, then $C/H\cong \mathbb P^1$, and we can assume without lost of generalities that $v_3$ is not in the branch locus, so in other words $g_3\in H$.
		\\
		We want to determine the signature of a spherical system of generators that defines $C\to C/H\cong \mathbb P^1$. Each point of the fibre of $q_i$ via $C/H\to \mathbb P^1$  is contained in the branch locus of $C\to C/H$ and has ramification index $b_i$, since $h_i\in H$. Note that the cardinality of the fibre is exactly $p$ for these points $q_i$.  The same holds for $v_3$, since also $g_3$ belongs to $H$. 
		\\
		Instead, the fibre of $v_i$ on $C/H$ consists of only one point, $i=1,2$. The ramification index of this point for $C\to C/H$ equals the order of $\langle g_i\rangle \cap H$, which is $a_i/p$. We therefore obtain the signature $[a_1/p,a_2/p,a_3^p, b_1^p, \dots, b_k^p]$. 
		
		The case $r=3$ can be discussed by using the same argument. 
	\end{proof}
\begin{remark}
	Since product-quotient surfaces with $K_S^2=32$ and $\chi(\mathcal O_S)=4$ are surfaces isogenous to a product, then their basket is always empty. For this reason, we shall avoid repeating which is the basket of the triples of the cases skipped by \textit{ListGroups}.  
\end{remark}
A first result is the following
\begin{proposition}\label{prop: exclude_groups_ord_lesseq_2000}
	There are exactly five groups $G$  of order different from $1024$ and less or equal than 2000  admitting an appropriate pair of spherical systems of generators compatible with one of the triples of the secondary output of  \textit{ListGroups}: 
\begin{center}
		\begin{tabular}{|c|c| l|}
			\hline
			$t_1$ & $t_2$ & $G$ \\
			\hline
			\hline
			$2,4,6$ & $2^3,4$ & $G(768, 1086051)$\\
			$2,4,6$ & $2^3,4$ & $G(768, 1086052)$\\
			$2,4,6$ & $2,4,20$ & $G(960, 5719)$\\
			$2,4,6$ & $2,4,12$ & $G(1152, 157849)$\\
			$2,4,5$ & $2,4,12$ & $G(1920, 240996)$\\
			\hline
		\end{tabular}
\end{center}
	However, no one of these cases gives product-quotient surfaces isogenous to a product. 
\end{proposition}
\begin{proof}
	We select in a list those triples of the secondary output of \textit{ListGroups} having a order of the group different from $1024$ and less or equal to $2000$. Then we run \textit{HowToExclude} on this list and we obtain the above table. 
	\\
	However, we use \textit{ExistingSurfaces} for each of the rows of the table to check that no one gives a product-quotient surface isogenous to a product. 
\end{proof}
As a consequence of the previous lemma, it remains to discuss  only $65$ of $152$ cases skipped by \textit{ListGroups}, that are those of Tables \ref{tab: tabellaA} and \ref{tab: tabellaB} below.
\begin{table}[h]
	\begin{minipage}[t]{0.5\textwidth}
		\vspace{0pt}
		\centering
		\begin{tabular}{|c|c|c| c|}
			\hline
			$no.$ & $t_1$ & $t_2$ & $\vert G\vert$ \\
			\hline
			\hline
			$1$	& $2,3,8$ & $2,5^2$ & $3840$\\  
			$2$	& $2,3,7$ & $4,4,4$ & $2688$\\
			$3$ &	$2,3,7$ & $2,3,18$ & $6048$\\
			$4$ &	$2,3,7$ & $2,4,8$ & $5376$\\
			$5$ &	$2,3,7$ & $3,3,5$ & $5040$\\
			$6$ &	$2,3,7$ & $2,5,6$ & $5040$\\
			$7$ &	$2,3,7$ & $2,8,8$ & $2688$\\
			$8$ &	$2,3,7$ & $3,3,15$ & $2520$\\
			$9$ &	$2,3,7$ & $2,3,7$ & $28224$\\
			$10$ &	$2,3,7$ & $2,5,30$ & $2520$\\
			$11$ &	$2,3,7$ & $2,3,10$ & $10080$\\
			$12$ &	$2,3,7$ & $2,2,2,4$ & $2688$\\
			$13$ &	$2,3,7$ & $2,6,15$ & $2520$\\
			$14$ &	$2,3,7$ & $3,5,5$ & $2520$\\
			$15$ &	$2,3,7$ & $2,3,30$ & $5040$\\
			$16$ &	$2,4,5$ & $3,3,4$ & $3840$\\
			$17$ &	$2,3,9$ & $2,4,5$ & $5760$\\
			$18$ &	$2,3,9$ & $2,5,6$ & $2160$\\
			$19$ & $2,3,12$ & $2,4,6$ & $2304$ \\  
			$20$ &	$2,3,10$ & $2,4,6$ & $2880$\\
			$21$ &	$2,3,8$ & $2,4,12$ & $2304$\\
			$22$ &	$2,3,8$ & $2,5,6$ & $2880$\\
			$23$ &	$2,3,22$ & $2,4,5$ & $2640$\\
			$24$ &	$2,3,12$ & $2,4,5$ & $3840$\\
			$25$ &	$2,3,14$ & $2,4,6$ & $2016$\\
			$26$ &	$2,3,8$ & $2,4,6$ & $4608$\\
			$27$ &	$2,3,18$ & $2,4,5$ & $2880$\\
			$28$ &	$2,3,10$ & $2,4,5$ & $4800$\\
			$29$ &	$2,3,54$ & $2,4,5$ & $2160$\\
			$30$ &	$2,4,5$ & $2,4,6$ & $3840$\\
			$31$ &	$2,3,30$ & $2,4,5$ & $2400$\\
			$32$ &	$2,4,5$ & $2,4,8$ & $2560$\\
			$33$ &	$2,3,8$ & $2,4,5$ & $7680$\\
			\hline
		\end{tabular}
		\caption{}
		\label{tab: tabellaA}
	\end{minipage}%
	\begin{minipage}[t]{0.5\textwidth}
		\vspace{0pt}
		\centering
		\begin{tabular}{|c|c|c| c|}
			\hline
			$no.$ & $t_1$ & $t_2$ & $\vert G\vert$ \\
			\hline
			\hline
			$34$ &	$2,3,14$ & $2,4,5$ & $3360$\\ 
			$35$ &	$2,3,8$ & $2,4,8$ & $3072$\\
			$36$ &	$2,3,8$ & $2,6,7$ & $2016$\\
			$37$ & $2,3,10$ & $2,3,10$ & $3600$ \\

			$38$	& $2,3,8$ & $2,3,18$ & $3456$\\

			$39$ &	$2,3,8$ & $2,3,54$ & $2592$\\
			$40$ &	$2,4,5$ & $2,5,6$ & $2400$\\
			$41$ &	$2,3,8$ & $2,3,22$ & $3168$\\   
			$42$ &	$2,3,12$ & $2,3,14$ & $2016$\\
			$43$ &	$2,3,8$ & $2,3,30$ & $2880$\\
			
			$44$ &	$2,3,8$ & $2,2,2,3$ & $2304$\\
			$45$ &	$2,3,8$ & $2,6,6$ & $2304$\\

			$46$ &	$2,3,8$ & $3,4,4$ & $2304$\\
			
			$47$ &	$2,3,10$ & $2,3,18$ & $2160$\\
			
			$48$ &	$2,3,10$ & $2,3,14$ & $2520$\\
			
			$49$ &	$2,3,10$ & $2,3,12$ & $2880$\\
			$50$	& $2,3,8$ & $2,3,14$ & $4032$\\
			$51$ &	$2,3,8$ & $2,3,8$ & $9216$\\
			
			$52$ &	$2,4,5$ & $2,4,5$ & $6400$\\
			$53$ &	$2,3,8$ & $2,3,12$ & $4608$\\
			$54$ &	$2,3,8$ & $2,3,10$ & $5760$\\
			$55$ &	$2,3,9$ & $3,3,5$ & $2160$\\  
			$56$ &	$2,3,9$ & $2,3,12$ & $3456$\\ 
			$57$ &	$2,3,12$ & $3,3,4$ & $2304$\\
			$58$ &	$2,3,9$ & $2,3,18$ & $2592$\\
			$59$ &	$2,3,9$ & $2,3,30$ & $2160$\\
			$60$ &	$2,3,9$ & $2,3,9$ & $5184$\\
			$61$ &	$3,3,4$ & $3,3,4$ & $2304$\\
			$62$ &	$2,3,9$ & $3,3,4$ & $3456$\\
			$63$ &	$2,4,6$ & $2,4,6$ & $2304$\\ 
			$64$ & $2,3,12$ & $2,3,12$ & $2304$ \\ 
			$65$ &	$2,4,8$ & $2,4,8$ & $1024$\\ 
			\hline
		\end{tabular}
		\caption{}
		\label{tab: tabellaB}
	\end{minipage}
\end{table}
\begin{remark}\label{rem: remind_order_abel}
	We remind that as observed in \eqref{eq: cond_abelianization_orders} the abelianization $G^{ab}$ of $G$ is a quotient of both the abelianizations of the orbifold groups $\mathbb T(t_1)$ and $\mathbb T(t_2)$. In particular, the order of $G^{ab}$ divides the greatest common divisor of the orders of $\mathbb T(t_1)^{ab}$ and $\mathbb T(t_2)^{ab}$. 
\end{remark}
\begin{remark}\label{rem: abelinizations_of_the_list}
	From Remark \ref{rem: remind_order_abel}, we automatically get that groups $G$ having group order and a pair of spherical system of generators compatible with one of the triples of Tables \ref{tab: tabellaA} and \ref{tab: tabellaB}
	\begin{enumerate}
		\item from $no. 1$ to $no. 18$ are perfect groups;
		\item from $no. 19$ to $no. 54$ are either perfect groups or $G^{ab}\cong \mathbb Z_2$; 
		\item from $no. 55$ to $no. 62$ are either perfect groups or $G^{ab}\cong \mathbb Z_3$;
		\item $no. 63$ are either perfect groups or $G^{ab}$ is isomorphic to $\mathbb Z_2$ or to $\mathbb Z_2\times \mathbb Z_2$;
		\item $no. 64$ are either perfect or $G^{ab}$ is isomorphic to $\mathbb Z_2$ or to $\mathbb Z_3$ or to $\mathbb Z_6$;
		\item $no. 65$ are either perfect or $G^{ab}$ is isomorphic to one among $\mathbb Z_2, \mathbb Z_2\times \mathbb Z_2$,$\mathbb Z_4$, $\mathbb Z_4\times \mathbb Z_2$;
	\end{enumerate}
\end{remark}
\begin{lemma}\label{lem: Lemma_PerfectGroups_G}
 There are no perfect groups $G$ having group order and a pair of spherical systems of generators of signatures compatible with one of the triples of Tables \ref{tab: tabellaA} and \ref{tab: tabellaB}. 
\end{lemma}
\begin{proof}
	We use \textit{HowToExcludePG} on the list of triples of Tables \ref{tab: tabellaA} and \ref{tab: tabellaB} to check that there are no perfect groups having compatible algebraic data. 
\end{proof}
\begin{proposition}\label{prop: from_1_to_18}
    There are no groups $G$ having group order and a pair of spherical systems of generators of signatures compatible with one of the triples of Tables \ref{tab: tabellaA} and \ref{tab: tabellaB} from $no.1$ to $no.18$.
\end{proposition}
\begin{proof}
	From Remark \ref{rem: abelinizations_of_the_list} and Lemma  \ref{lem: Lemma_PerfectGroups_G} we can automatically exclude triples of Table \ref{tab: tabellaA} from $no.1$ to $no.18$. 
\end{proof}
\begin{remark}\label{rem: rem_group_theory}We need the following classical remarks of group theory: 
	\begin{enumerate}
		\item 	Let $G$ be a finite group having a normal subgroup $H$ of index a prime number $p\geq 2$. If there is a element $g\notin H$ of order $p$, then 
		\[
		0\to H\to G\to \mathbb Z_p\to 0
		\]
		is a splitting exact sequence via the homomorphism section sending $\overline{1}\in\mathbb Z_p$  to $g$. In other words, $G=H\rtimes_\phi \mathbb Z_p$, where $\phi$ is an automorphism of $H$ of order $p$;
		\item Let $\pi\colon G\to Z$ be a surjective group homomorphism. If $Z$ admits a normal subgroup $T$ of index $k\in \mathbb N$, then $H:=\pi^{-1}(T)$ is a normal subgroup of $G$ of index $k$. More precisely,  $G/H\cong Z/T$. 
	\end{enumerate}
\end{remark}
\begin{proposition}\label{prop: from_19_to_62}
	There are no groups $G$ having group order and a pair of spherical systems of generators defining a product-quotient surface isgenous to a product and compatible with one of the triples from $no.19$ to $no.62$ of Tables \ref{tab: tabellaA} and \ref{tab: tabellaB}.
\end{proposition}
\begin{proof}
	From Remark \ref{rem: abelinizations_of_the_list} and Lemma \ref{lem: Lemma_PerfectGroups_G}, groups $G$ from $no.19$ to $no.62$ of Tables \ref{tab: tabellaA} and \ref{tab: tabellaB} have a commutator subgroup $G':=[G,G]$ of index either equal to $2$ or $3$. Hence we can apply Proposition \ref{prop: intermediate_quotients} to $H=G'$ and say that $G'$ has group order and a pair of spherical systems of generators compatible with one of the triples of Tables \ref{tab: tabellaA1} and \ref{tab: tabellaB1}. 
\begin{table}[h]
	\begin{minipage}[t]{0.5\textwidth}
		\vspace{0pt}
		\centering
		\begin{tabular}{|c|c|c| c|}
			\hline
			$no.$ & $t_1$ & $t_2$ & $\vert G'\vert$ \\
			\hline
			\hline
			$19(a)$ & $3,3,6$ & $2,2,2,3$ & $1152$ \\  
			(b)& $3,3,6$ & $3,4,4$ & $1152$ \\
			(c)& $3,3,6$ & $2,6,6$ & $1152$ \\
			$20(a)$ &	$3,3,5$ & $2,2,2,3$ & $1440$\\
			(b)	&$3,3,5$ & $3,4,4$ & $1440$\\
			(c) &	$3,3,5$ & $2,6,6$ & $1440$\\
			$21(a)$ &	$3,3,4$ & $2,2,2,6$ & $1152$\\
			(b)&	$3,3,4$ & $4,4,6$ & $1152$\\
			(c)&	$3,3,4$ & $2,12,12$ & $1152$\\
			$22$ &	$3,3,4$ & $3,5,5$ & $1440$\\
			$23$ &	$3,3,11$ & $2,5,5$ & $1320$\\
			$24$ &	$3,3,6$ & $2,5,5$ & $1920$\\
			$25(a)$ &	$3,3,7$ & $2,2,2,3$ & $1008$\\
			(b)&	$3,3,7$ & $3,4,4$ & $1008$\\
			(c)&	$3,3,7$ & $2,6,6$ & $1008$\\
			$26(a)$ &	$3,3,4$ & $2,2,2,3$ & $2304$\\
			(b)&	$3,3,4$ & $3,4,4$ & $2304$\\
			(c)&	$3,3,4$ & $2,6,6$ & $2304$\\
			$27$ &	$3,3,9$ & $2,5,5$ & $1440$\\
			$28$ &	$3,3,5$ & $2,5,5$ & $2400$\\
			$29$ &	$3,3,27$ & $2,5,5$ & $1080$\\
			$30(a)$ &	$2,5,5$ & $2,2,2,3$ & $1920$\\
			(b)&	$2,5,5$ & $3,4,4$ & $1920$\\
			(c)&	$2,5,5$ & $2,6,6$ & $1920$\\
			$31$ &	$3,3,15$ & $2,5,5$ & $1200$\\
			$32(a)$ &	$2,5,5$ & $2,2,2,4$ & $1280$\\
			(b)&	$2,5,5$ & $4,4,4$ & $1280$\\
			(c)&	$2,5,5$ & $2,8,8$ & $1280$\\
			$33$ &	$3,3,4$ & $2,5,5$ & $3840$\\
			$34$ &	$3,3,7$ & $2,5,5$ & $1680$\\
			\hline
		\end{tabular}
		\caption{}
		\label{tab: tabellaA1}
	\end{minipage}%
	\begin{minipage}[t]{0.5\textwidth}
		\vspace{0pt}
		\centering
		\begin{tabular}{|c|c|c| c|}
			\hline
			$no.$ & $t_1$ & $t_2$ & $\vert G'\vert$ \\
			\hline
			\hline
			$35(a)$ &	$3,3,4$ & $2,2,2,4$ & $1536$\\ 
			(b)&	$3,3,4$ & $4,4,4$ & $1536$\\
			(c)&	$3,3,4$ & $2,8,8$ & $1536$\\
			$36$ &	$3,3,4$ & $3,7,7$ & $1008$\\  
			$37$ & $3,3,5$ & $3,3,5$ & $1800$ \\ 
			
			$38$	& $3,3,4$ & $3,3,9$ & $1728$\\

			$39$ &	$3,3,4$ & $3,3,27$ & $1296$\\
			$40$ &	$2,5,5$ & $3,5,5$ & $1200$\\
			$41$ &	$3,3,4$ & $3,3,11$ & $1584$\\   
			$42$ &	$3,3,6$ & $3,3,7$ & $1008$\\
			$43$ &	$3,3,4$ & $3,3,15$ & $1440$\\
			$44$ &	$3,3,4$ & $2,2,3,3$ & $1152$\\
			$45(a)$ &	$3,3,4$ & $2,2,3,3$ & $1152$\\
			(b)&	$3,3,4$ & $3,6,6$ & $1152$\\

			$46$ &	$3,3,4$ & $2,2,3,3$ & $1152$\\
			
			$47$ &	$3,3,5$ & $3,3,9$ & $1080$\\

			$48$ &	$3,3,5$ & $3,3,7$ & $1260$\\
			
			$49$ &	$3,3,5$ & $3,3,6$ & $1440$\\
			
			$50$	& $3,3,4$ & $3,3,7$ & $2016$\\
			$51$ &	$3,3,4$ & $3,3,4$ & $4608$\\
			$52$ &	$2,5,5$ & $2,5,5$ & $3200$\\
			$53$ &	$3,3,4$ & $3,3,6$ & $2304$\\
			$54$ &	$3,3,4$ & $3,3,5$ & $2880$\\
			$55$ &	$2,2,2,3$ & $5,5,5$ & $720$\\ 
			$56$ &	$2,2,2,3$ & $2,2,2,4$ & $1152$\\ 
			$57$ &	$2,2,2,4$ & $4,4,4$ & $768$\\
			$58$ &	$2,2,2,3$ & $2,2,2,6$ & $864$\\
			$59$ &	$2,2,2,3$ & $2,2,2,10$ & $720$\\
			
			$60$ &	$2,2,2,3$ & $2,2,2,3$ & $1728$\\
			$61$ &	$4,4,4$ & $4,4,4$ & $768$\\
			$62$ &	$2,2,2,3$ & $4,4,4$ & $1152$\\
			\hline
		\end{tabular}
		\caption{}
		\label{tab: tabellaB1}
	\end{minipage}
\end{table}
	\begin{remark}\label{rem: excludePG_G'}
		We run \textit{HowToExcludePG} on the list of Tables \ref{tab: tabellaA1} and \ref{tab: tabellaB1} to see that there are no compatible perfect groups $G'$. 
	\end{remark}
\underline{from $no.19$ to $no.36$, and $no.55$}
\\
From Remark \ref{rem: remind_order_abel}, we see that triples of Tables \ref{tab: tabellaA1} and \ref{tab: tabellaB1} from $no.19$ to $no.36$ (with the exception of $no. 19(c),20(c),21(c),25(c),26(c)$) together with $no.55$ have $G'$ forced to be a perfect group, which is a contradiction with Remark \ref{rem: excludePG_G'}. 
\\
We run $HowToExclude$ on $no. 19(c),20(c),21(c),25(c)$ to prove that there are no groups compatible with those algebraic data. 
\\
Instead, we exclude $26(c)$ using the following 
\begin{remark}\label{rem: 768,4^3}
	From \cite[Lemma 4.11]{BCG08}, there are no groups of order $768$ having a spherical system of generators of signature $[4,4,4]$.
\end{remark}
 Indeed, we would get $G''=[G',G']$ of $26(c)$ is a group of order $768$ and from Proposition \ref{prop: intermediate_quotients} it should admit a spherical system of generators of signature $[4,4,4]$. 
 
 We have excluded all cases from $no.19$ to $no.36$ together with $no.55$ of Tables \ref{tab: tabellaA} and \ref{tab: tabellaB}. 

$\underline{no. 53,57,61}$
\\
Here $no. 53$ of Table \ref{tab: tabellaB} can be excluded by using Remark \ref{rem: 768,4^3} and the same argument of $26(c)$ applied to $no.53$ of Table \ref{tab: tabellaB1}. 
We also automatically exclude $no.57$ and $no.61$ of Table \ref{tab: tabellaB} by using Remark \ref{rem: 768,4^3} applied to $no.57$ and $no.61$ to Table \ref{tab: tabellaB1}. 

\underline{from $no.37$ to $no.49$, and $no.56,58,59,60,62$}
\\
We run $HowToExclude$ on the corresponding triples of Table \ref{tab: tabellaB1} to see that there are only $10$ groups, those of Table \ref{tab: 10groups}, 
compatible with the algebraic data.
\begin{table}[h]
	\begin{tabular}{|c|c|c| c|}
		\hline
		$no.$ & $t_1$ & $t_2$ & $ G'$ \\
		\hline
		\hline
		$58$ & $2,2,2,3$ & $2,2,2,6$ & $G(864,2225)$ \\  
		$58$ & $2,2,2,3$ & $2,2,2,6$ & $G(864,4175)$ \\ 
		$59$ & $2,2,2,3$ & $2,2,2,10$ & $G(720,764)$ \\ 
		$59$ & $2,2,2,3$ & $2,2,2,10$ & $G(720,771)$ \\ 
		$60$ & $2,2,2,3$ & $2,2,2,3$ & $G(1728,12317)$ \\ 
		$60$ & $2,2,2,3$ & $2,2,2,3$ & $G(1728,32133)$ \\ 
		$60$ & $2,2,2,3$ & $2,2,2,3$ & $G(1728,46099)$ \\ 
		$60$ & $2,2,2,3$ & $2,2,2,3$ & $G(1728,46119)$ \\ 
		$60$ & $2,2,2,3$ & $2,2,2,3$ & $G(1728,47853)$ \\ 
		$60$ & $2,2,2,3$ & $2,2,2,3$ & $G(1728,47900)$ \\ 
		\hline
	\end{tabular}
\caption{}
\label{tab: 10groups}
\end{table}
However, we run $ExistingSurfaces$ on this list to cheek that there are not surfaces isogenous to a product. 
\vspace{0,3cm}

For the remaining four cases $\underline{no.50,51,52,54}$ we remind Remark \ref{rem: excludePG_G'} and so we apply Proposition \ref{prop: intermediate_quotients} to the commutator $G''\lhd G'$, which has then to admit one of the algebraic data of Table \ref{tab: tabella_fourcases}.
\begin{table}[h]
		\vspace{0pt}
		\centering
		\begin{tabular}{|c|c|c| c|}
			\hline
			$no.$ & $t_1$ & $t_2$ & $\vert G''\vert$ \\
			\hline
			\hline
			$50$ & $4,4,4$ & $7,7,7$ & $672$ \\  
			$51$& $4,4,4$ & $4,4,4$ & $1536$ \\
			$52$ & $2^5$ & $2^5$ & $640$ \\
			$54$ &	$4,4,4$ & $5,5,5$ & $960$\\
			\hline
		\end{tabular}
		\caption{}
		\label{tab: tabella_fourcases}
\end{table}
We run $HowToExcludePG$ on this list to automatically exclude $no.50$ and $no.54$ of Table \ref{tab: tabellaB}. 

\underline{$no.52$}
\\
We run $HowToExclude$ and then $ExistingSurfaces$ for $no.52$ of Table \ref{tab: tabella_fourcases} to see that there are only four groups $G''(640,n)$ having a pair of spherical system of generators defining a product-quotient surface isogenous to product, where $n=7665, 8697, 12278, 15814$. 
\\
However, we remind that $G''$ has index $5$ in $G'$, which admits a spherical system of generators $[g_1,g_2,g_3]$ of signature $[2,5,5]$. Then $g_2\not\in G''$ and it has order $5$. This means from Remark \ref{rem: rem_group_theory}(1) that
\[
0\to G''\to G'\to \mathbb Z_5\to 0
\]
is a splitting exact sequence, 
so $G'=G''\rtimes_\phi \mathbb Z_5$ trough an automorphism $\phi$ of $G''$ of order $5$. We easily check that each of the obtained groups $G''(640,n)$ admits exactly four automorphisms of order $5$. However, for each of these automorphisms $\phi$ the semidirect product $G''(640,n)\rtimes_\phi \mathbb Z_5$ has abelianization $\mathbb Z_2^4\times \mathbb Z_5$, so no one of these groups can be $G'$ of $no. 52$ in Table \ref{tab: tabellaB1}, which has abelianization $\mathbb Z_5$.  
\\
This then excludes groups $G$ of $no.52$ of Table \ref{tab: tabellaB}. 

\underline{$no.51$}
\\
It remains to only discuss $no.51$ of Table \ref{tab: tabella_fourcases}. We run $ExSphSyst$ to each group of order 1536 to observe that
\begin{remark}
	There are no groups of order 1536 which admit a spherical system of generators of signature $[4,4,4]$. 
\end{remark}
\end{proof}
\begin{proposition}\label{prop: 63_and_64}
	There are no groups $G$ having group order and a pair of spherical systems of generators defining a product-quotient surface isogenous to a product and compatible with one of the triples $no.63$ and $no.64$ of Table \ref{tab: tabellaB}.
\end{proposition}
\begin{proof}
	From remarks \ref{rem: abelinizations_of_the_list}, \ref{lem: Lemma_PerfectGroups_G}, and \ref{rem: rem_group_theory}(2), then $G$ of $no.63$ admits a normal subgroup $H$ of index $2$, whilst $G$ of $no.64$ admits a normal subgroup $H$ of index either $2$ of $3$. We apply Proposition \ref{prop: intermediate_quotients} to $H$, which has then one of the following algebraic data:
\begin{table}[h]
	\begin{minipage}[t]{0.5\textwidth}
		\vspace{0pt}
		\centering
		\begin{tabular}{|c|c|c| c|}
			\hline
			$no.$ & $t_1$ & $t_2$ & $\vert H\vert$ \\
			\hline
			\hline
			$63$ & $2,2,2,3$ & $2,2,2,3$ & $1152$ \\  
			$63$& $2,2,2,3$ & $3,4,4$ & $1152$ \\
			$63$ & $2,2,2,3$ & $2,6,6$ & $1152$ \\
			$63$ &	$3,4,4$ & $3,4,4$ & $1152$\\
			\hline
		\end{tabular}
		\caption{}
		\label{tab: H_no.63.64_1}
	\end{minipage}%
	\begin{minipage}[t]{0.5\textwidth}
		\vspace{0pt}
		\centering
		\begin{tabular}{|c|c|c| c|}
			\hline
			$no.$ & $t_1$ & $t_2$ & $\vert H\vert$ \\
			\hline
			\hline
			$63$ &	$3,4,4$ & $2,6,6$ & $1152$\\
			$63$ &	$2,6,6$ & $2,6,6$ & $1152$\\
			$64$ &	$3,3,6$ & $3,3,6$ & $1152$\\
			$64$ &	$2,2,2,4$ & $2,2,2,4$ & $768$\\
			\hline
		\end{tabular}
		\caption{}
		\label{tab: H_no.63.64_2}
	\end{minipage}
\end{table}

We run $HowToExclude$ on this list to see that there are $90$ groups compatible with one of the algebraic data. However, we then run $ExistingSurfaces$ to check that there are no surfaces isogenous to a product. 
\end{proof}
\begin{proposition}\label{prop: 65}
	There are no groups $G$ having group order and a pair of spherical systems of generators defining a product-quotient surface isogenous to a product and compatible with the triple $no.65$ of Table \ref{tab: tabellaB}.
\end{proposition}
\begin{proof}
	From remarks \ref{rem: abelinizations_of_the_list}, \ref{lem: Lemma_PerfectGroups_G}, and \ref{rem: rem_group_theory}(2), then $G$ of $no.65$ admits a normal subgroup $H$ of index $2$. We apply Proposition \ref{prop: intermediate_quotients} to $H$, which has then one of the algebraic data of Table \ref{tab: H_no.65}. 
	\begin{table}[h]
		\vspace{0pt}
		\centering
		\begin{tabular}{|c|c|c| c|}
			\hline
			$no.$ & $t_1$ & $t_2$ & $\vert H\vert$ \\
			\hline
			\hline
			$65$ & $2,2,2,4$ & $2,2,2,4$ & $512$ \\  
			$65$& $2,2,2,4$ & $4,4,4$ & $512$ \\
			$65$ & $2,2,2,4$ & $2,8,8$ & $512$ \\
			$65$ &	$4,4,4$ & $4,4,4$ & $512$\\
			$65$ &	$4,4,4$ & $2,8,8$ & $512$\\
			$65$ &	$2,8,8$ & $2,8,8$ & $512$\\
			\hline
		\end{tabular}
		\caption{}
		\label{tab: H_no.65}
	\end{table}
We run $HowToExclude$ and then $ExistingSurfaces$ on Table \ref{tab: H_no.65} to see that there are only three groups $H(512,n)$ having a pair of spherical system of generators of signature $[4,4,4]$ defining a product-quotient surface isogenous to a product, where $n=325,335,351$. 
 \\
 Since $G$ admits a spherical system of generators $[g_1,g_2,g_3]$ of signature $[2,4,8]$ and $H(512,n)$ has signature $[4,4,4]$, then $g_1\notin H(512,n)$, so that from Remark \ref{rem: rem_group_theory}(1) 
 \[
 0\to  H(512,n)\to G\to \mathbb Z_2\to 0
 \] 
 is a splitting exact sequence. In other words, $G=H(512,n)\rtimes_\phi\mathbb Z_2$ via an automorphism $\phi$ of order $2$. For each of these three groups $H(512,n)$, there are respectively $128$,$384$,$128$  automorphisms $\phi$ of order two such that $H(512,n)\rtimes_\phi\mathbb Z_2$ has abelianization that is a quotient of $\mathbb Z_4^2$. In particular, for each of of these automorphisms the abelianization is $\mathbb Z_2\times \mathbb Z_4$. Furthermore, each of these groups admits a spherical system of generators of signature $[4,4,4]$. 
 \\
 However, we run $ExisitngSurfaces$ to prove that no one of them gives a product-quotient surface isogenous to a product. 
\end{proof}

\subsection{Rational (-1) curves on product-quotient surfaces}\label{subsec: rational_curves_on_PQ}
In this short subsection we investigate which surfaces among those obtained in Theorem \ref{thm: intro_teo_classif_pg=3} do not contain $(-1)$-curves, namely smooth rational curves with self-intersection $-1$. First of all we observe that 
\begin{remark}
	All surfaces $S$ obtained in Theorem \ref{thm: intro_teo_classif_pg=3} are surfaces of general type. Indeed, from Enriques-Kodaira classification of complex algebraic surfaces, if $q(S)$ is zero, then either $S$ is rational,  of $S$ is of general type, or $K^2_S\leq 0$. Therefore, since surfaces of Theorem \ref{thm: intro_teo_classif_pg=3} have $K^2_S\geq 23$, $q(S)=0$, and  $p_g(S)=3\neq 0$, then they are of general type.    
\end{remark} 

\begin{proposition}\label{prop: minimality_lemma}\cite[Lemma 6.9]{BP16}
Let $S$ be  a product-quotient surface of general type of quotient model $X$. Assume that the exceptional locus of the minimal resolution of singularities $\rho \colon S\to X$ consists of 
\begin{itemize}
	\item[i)] curves of self-intersection $(-3)$ and $(-2)$, or 
	\item[ii)] at most two smooth rational curves of self-intersection $(-3)$ or $(-4)$, and $(-2)$-curves.
\end{itemize}  
Then $S$ is minimal, so it does not contain $(-1)$-curves.  
\end{proposition}

\begin{corollary}
	Let $S$ be a product-quotient surface belonging to one of the families of tables  \ref{table: class_PQ_pg=3} to  \ref{table: class_PQ_pg=3_12} of Theorem \ref{thm: intro_teo_classif_pg=3}. 
Then $S$ is a minimal surface. 
\end{corollary}
\begin{proof}
	For each case of tables  \ref{table: class_PQ_pg=3} to  \ref{table: class_PQ_pg=3_12} (with the exception of $no. 186$ to $196$) 
	 the exceptional curves arising from the basket of singularities of the quotient model $X$ are either of type $i)$ or $ii)$ of Proposition \ref{prop: minimality_lemma}, so that $S$ is minimal. 
	\\
	Regarding the remain cases $no.$ $186$ to $196$, their basket of singularities is always equal to  $\{1/5,4/5\}$, so the minimality follows directly by \cite[Prop. 4.7(3)]{BP12}. 
\end{proof}
\section{The degree of the canonical map of product-quotient surfaces}\label{sec: degree_can_map}
In this section we investigate the degree of the canonical map of product-quotient surfaces, with a particular focus to those having geometric genus three. 

We briefly explain the strategy and the content of each subsection but first we give the following 
\begin{remark}\label{rem: higest_value_Ksquare}
	The degree of the canonical map of product-quotient surfaces is bounded from above by $32$. Indeed, product-quotient surfaces satisfy the inequality $K_S^2\leq 8\chi(\mathcal O_S)$, see Theorem \ref{thm: invariants_PQ}, and so replacing Bogomolov-Miyaoka-Yau inequality with $K_S^2\leq 8\chi(\mathcal O_S)$ in the proof of \cite[Prop. 4.1]{beauville}, we get deg$(\Phi_S)\leq 8\chi(\mathcal O_S)/\left(\chi(\mathcal O_S)-3\right)\leq 32$. 
\end{remark}
Let us consider a product-quotient surface $S$ given by a pair of curves $C_1$ and $C_2$ and a finite group $G$ acting (faithfully) on both of them. 


The diagonal action of $G$ on the product $C_1\times C_2$ induces a representation of $G$ on the spaces of $2$-forms of $C_1\times C_2$. Let us denote by $\vert K_{C_1\times C_2}\vert^G$ the linear subsystem of the canonical linear system of $C_1\times C_2$ given by the subspace $H^{2,0}(C_1\times C_2)^G$ of the $G$-invariant $2$-forms.  

In Subsection \ref{subsec: deg_can_map} we show the relationship between the degree of the canonical map of $S$ and the (schematic) base locus of $\vert K_{C_1\times C_2}\vert^G$. Indeed, it holds 
\begin{equation}\label{eq: formula_deg_phiS}
	\deg(\Phi_{S})=\frac{1}{\vert G\vert\cdot \deg(\Sigma)}\cdot \widehat{M}^2,
\end{equation}
where $\Sigma$ is the image of the canonical map of $S$, and $\widehat{M}$ is the base-point free linear system obtained blowing-up the base locus of  $\vert K_{C_1\times C_2}\vert^G$. 
\\
Note that whenever $p_g(S)=3$, then the image of the canonical map is $\mathbb P^2$, a surface of degree $1$, and so the knowledge of the base locus of  $\vert K_{C_1\times C_2}\vert^G$ is enough to compute $\deg(\Phi_{S})$ automatically by using formula \eqref{eq: formula_deg_phiS}. 

The strategy to investigate the base locus of  $\vert K_{C_1\times C_2}\vert^G$ is the following. The action of $G$ induces a representation on the space of $1$-forms $H^{1,0}(C_i)$ via pullback, called in literature  $\textit{canonical representation}$. By standard representation theory, the space of $1$-forms splits as a direct sum of isotypic components $H^{1,0}(C_i)^\chi$, $\chi\in Irr(G)$ irreducible character of $G$. The irreducible characters $\chi$ occurring in the character $\chi_{can}$ of the canonical representation are explicitly computable by the Chevalley-Weil formula, see \cite[Subsection 1.3]{G16}. 

As a consequence of this, the space of invariant $2$-forms $H^{2,0}(C_1\times C_2 )^G$ splits as a direct sum of invariant subspaces $\left(H^{1,0}(C_1)^\chi\otimes H^{1,0}(C_2)^{\overline{\chi}}\right)^G$, $\chi\in Irr(G)$. Therefore, the base locus of $\vert K_{C_1\times C_2}\vert^G$ is simply the intersection of the base loci of such invariant subspaces and then a computation of them solves the problem. 

Let us consider the natural inclusion 
\begin{equation}\label{eq: inclusion}
	\left(H^{1,0}(C_1)^\chi\otimes H^{1,0}(C_2)^{\overline{\chi}}\right)^G \subseteq H^{1,0}(C_1)^\chi\otimes H^{1,0}(C_2)^{\overline{\chi}}.
\end{equation}
The main Theorem \ref{thm: base_locus_H10C1chi_otimes_H10C2_overline_chi} of Subsection \ref{subsec: base_locus_PQ} computes the base locus of the linear subsystem of the bigger subspace $H^{1,0}(C_1)^\chi\otimes H^{1,0}(C_2)^{\overline{\chi}}$, which is discovered to being pure in codimension 1 and union of fibres (with multiplicities) for the natural projections $C_1\times C_2\to C_i$, $i=1,2$. 
\\
The formula to compute explicitly these fibres and their multiplicities is given through Theorem \ref{thm: Base_locus_formula} in Subsection \ref{subsec: base_locus_formula}. This theorem provides the base locus of the subsystem of the canonical system of a Riemann surface $C$ given by an isotypic component $H^{1,0}(C)^\chi$ of the action of a finite group $G$ on $C$.

Notice that, whenever $\chi$ is of degree one, then \eqref{eq: inclusion} is an equality. Thus, if  $S$ satisfies \hyperlink{Property $\left(\#\right)$}{Property $\left(\#\right)$} mentioned at the introduction for which 
\[
\left(H^{1,0}(C_1)^\chi\otimes H^{1,0}(C_2)^{\overline{\chi}}\right)^G\neq 0 \implies \deg(\chi)=1
\]
for any irreducible character $\chi$, then we know $\vert K_{C_1\times C_2}\vert^G$ is spanned by $p_g$ divisors which decomposes as union of fibres for the natural projections $ C_1\times C_2\to C_i$, $i=1,2$. Since two fibres either do not intersect or they intersect transversally at a point, this makes the base locus of $\vert K_{C_1\times C_2}\vert^G$ explicit. 
\begin{remark}
	Observe that \hyperlink{Property $\left(\#\right)$}{Property $\left(\#\right)$} always holds for $G$ abelian group, and it is sometimes satisfied for other non-abelian groups, since we are only  interested to those characters of $G$ for which \eqref{eq: inclusion} is not zero. 
\end{remark}
\begin{remark}
	In terms of representation theory, \hyperlink{Property $\left(\#\right)$}{Property $\left(\#\right)$} translates as 
	\[
	\langle \chi_{can}^1,\chi\rangle\neq 0 \qquad \makebox{and} \qquad \langle \chi_{can}^2,\overline{\chi}\rangle \neq 0 \implies \deg(\chi)=1
	\]
	for each irreducible character $\chi$, where $\chi_{can}^i$ is the character of the canonical representation of $C_i$, $i=1,2$. 
	\\
	Thus, once $\chi_{can}^1$ and $\chi_{can}^2$ are determined  using the Chevalley-Weil formula, verifying whether \hyperlink{Property $\left(\#\right)$}{Property $\left(\#\right)$} holds reduces to a simple numerical computation.
\end{remark}
Let us suppose now  \hyperlink{Property $\left(\#\right)$}{Property $\left(\#\right)$}
 is satisfied, namely each irreducible character $\chi$ of $G$ for which $\eqref{eq: inclusion}$ is not zero has degree one. 
\\
In Subsection \ref{subsec: base_locus_PQ} we explain how to compute the self-intersection of the mobile part $M$ of $\vert K_{C_1\times C_2}\vert^G$.
\\
Note that the difference  $M^2-\widehat{M}^2$ is the sum of the correction terms arising from each isolated base-point of $M$. 
\\
 To finish the computation of the degree, whenever $p_g(S)=3$, we use iteratively for each base point of $M$ the Correction Term formula \ref{thm:Correction_term_formula} in Subsection \ref{subsec: correction_term_formula}, which provides the correction term of each base point to the difference $M^2-\widehat{M}^2$.
 Such formula is a generalization of the formula presented in \cite{FG21} and it seems of independent interest, so that it is presented in a more general setting.
 \\
Once we have determined both $M^2$ and   $M^2-\widehat{M}^2$, then the degree of the canonical map of $S$ is obtained by rearranging formula \eqref{eq: formula_deg_phiS} as follows 
\[
\deg(\Phi_{S})=\frac{1}{\vert G\vert}\cdot \left(M^2-\left(M^2-\widehat{M}^2\right)\right).
\]

\subsection{Isotypic components of canonical representations of actions on curves}\label{subsec: base_locus_formula}

Let $C$ be a Riemann surface, $G< \Aut(C)$ be a finite group, $C':=C/G$ its quotient, and let  $\lambda\colon C\to C'$  be the quotient map.  
\\
$G$ acts on $H^{1,0}(C)$ via the canonical representation:
\[
\left(g\cdot \omega\right)_p :=(dg^{-1})_p\omega_{g^{-1}\cdot p},
\]
Let us denote by $\chi_{can}$ the character of the canonical representation, which takes the name of \textit{canonical character}. 
The canonical representation can be splitted as a direct sum of irreducible representations:
\[
H^{1,0}(C)=\bigoplus_{\chi \in Irr(G)}H^{1,0}(C)^\chi.
\]
Here $H^{1,0}(C)^\chi$ is the \textit{isotypic component} of $H^{1,0}(C)$ of character $\chi$, namely it is a $G$-invariant subspace such that the restriction of the canonical representation is isomorphic to $\langle \chi_{can},\chi\rangle$-times the irreducible representation given by the character $\chi$. 
\\
In terms of characters, the above splitting translates as 
\[
\chi_{can}=\sum_{\chi \in Irr(G)}\langle \chi_{can}, \chi\rangle\cdot \chi. 
\]
We shall use the algorithm developed in \cite{G16} and implemented in the  computational algebra system MAGMA to compute the canonical character $\chi_{can}$ of any Galois branched covering. 

The aim of this section is to investigate the base locus of the associated subsystem $\vert K_{C}\vert^\chi$ given by the isotypic component $H^{1,0}(C)^\chi$. Let us give first some preliminary results. 

{\bf Notation:} Given a point $q\in C'$, the divisor $\lambda^{-1}(q)$ is considered with the reduced structure. 
\begin{lemma}\label{Lemma:BsInv}
	Consider a $G$-invariant subspace $W\subseteq H^{1,0}(C)$. 
	For any $p\in \lambda^{-1}(q)$, let $t_p$ be the minimal order of vanishing of a $1$-form in $\vert W\vert$ at $p$. Then all $t_p$ are equal to the same number, denoted by $t_q$. 
	Therefore the base locus of $\vert W\vert$ is a union of orbits
	\[
	Bs(|W|)= \sum_{q} t_q \lambda^{-1}(q).
	\]
	Furthermore, there exists a general form $\omega\in W$ vanishing of order exactly $t_q$ at each $p\in \lambda^{-1}(q)$. 
\end{lemma}
\begin{proof}
	For every point $p\in \lambda^{-1}(q)$, it there exists a $1$-form $\omega_p$ in $W$ vanishing at $p$ with order $t_p$, by the definition of $t_p$. Given $g\in G$, then $g\cdot \omega_p$ belongs to the invariant subspace $W$ too, and it vanishes at $g\cdot p$ with multiplicity $t_p$, so that $t_{g\cdot p}\leq t_p$. Hence all $t_p$ are equal to the same number, denoted as $t_q$.  
	
	We observe that a generic linear combination $\omega$ of the obtained $\vert \lambda^{-1}(q)\vert$ $1$-forms $\omega_p$ vanishes with order $t_q$ at each point of $\lambda^{-1}(q)$. 
\end{proof}
\begin{remark}
	Let $\omega\in W$ be a $1$-form of the Lemma \ref{Lemma:BsInv}, with vanishing order $t_q$ at each point  $p\in \lambda^{-1}(q)$. Given $g\in G$, then $g\cdot \omega\in W$ is a $1$-form with vanishing order $t_q$ at each point  $p\in \lambda^{-1}(q)$. 
\end{remark}
Let $H^{1,0}(C)^\chi$ be the isotypic component of $H^{1,0}(C)$ of irreducible character $\chi$. 
\begin{lemma}\label{lem: f_merom_as_G-equiv_isom}
	Let $f\in \mathcal{M}(C/G)=\mathcal M(C)^G$ be  a non-zero  invariant meromorphic function. Denote by $ H^{1,0}(C)^\chi_f$ the subspace of $H^{1,0}(C)^\chi$ consisting of forms $\omega$ such that $f\omega$ is a holomorphic form. Then 
	\begin{equation}\label{eq: map_f_betw_H10(C)}
		f\colon H^{1,0}(C)^\chi_f\to f\cdot H^{1,0}(C)_f^\chi\subseteq H^{1,0}(C), \qquad \omega\mapsto f\omega
	\end{equation}
	is a $G$-equivariant isomorphism. In particular, $f\cdot H^{1,0}(C)_f^\chi$ is a $G$-invariant subspace of $H^{1,0}(C)^\chi$. 
\end{lemma}
\begin{proof}
	$H^{1,0}(C)^\chi_f$ is $G$-invariant: given $g\in G$ and $\omega\in H^{1,0}(C)^\chi_f$, then $f(g\cdot \omega)=g\cdot (f\omega)$ is holomorphic since $f$ is $G$-invariant, and $f\omega$ is holomorphic. This shows immediately also that the map of $\eqref{eq: map_f_betw_H10(C)}$ is $G$-equivariant. From Schur Lemma, then the image of $\eqref{eq: map_f_betw_H10(C)}$  is contained in $H^{1,0}(C)^\chi$.
	However, $f$ is not the zero function, so \eqref{eq: map_f_betw_H10(C)} is injective.
\end{proof}
\begin{definition}\label{defn: k_p_Weierstrass}
	Let $X$ be a Riemann surface and $q\in X$. Let us define 
	\[
	k_q:=\min \left\{m\in \mathbb{N} : h^0(X, mq)\geq 2\right\}
	\]
	be the \textit{minimal non-gap of} $q$. 
	$k_q$ is therefore the smallest  number such that $X$ admits a non-constant meromorphic function $f$ with only one pole at $q$, of order $-k_q$. 
\end{definition}
\begin{remark}
	From Riemann-Roch theorem we have 
	\[
	h^0(X,(g(X)+1)q)=h^0(X, K-(g(X)+1)q)+2\geq 2.
	\]
	Therefore 
	\[
	k_q\leq g(X)+1.
	\]
	In other words, $k_q$ is the minimum of the complement of the set of the Weierstrass gaps for $q$. In particular, $k_q=g(X)+1$, if $q$ is not a Weierstrass point, or $k_q< g(X)+1$, otherwise. 
\end{remark}
Let $q\in C'$ be a branch point of $\lambda$. The stabilizers of the points lying on $q$ are cyclic subgroups of $G$ and they are conjugated to each other. Thus the order of the stabilizers depends only on $q$, denoted as $m_q$. 
\\
We remind the definition 
\vspace{0,3cm}

\begin{definition}
	Let us fix a point $p\in \lambda^{-1}(q)$. Given a generator $h$ of $Stab(p)$, there exists a coordinate $z$ in $C$ such that the action of $h$ in a neighborhood of $p$ corresponds to $z\to \lambda z$, where $\lambda$ is one of the $m_q$-roots of the unity. This gives a bijection among the primitive $m_q$-roots of the unity and the generators of $Stab(p)$. We denote by \textit{local monodromy} of $p$ the unique generator of $Stab(p)$ acting by $z\to e^{\frac{2\pi i }{m_q}}z$.
\end{definition}
\begin{remark}
	The \textit{local monodromy} of another point $g\cdot p$ over $q$ is the conjugate $ghg^{-1}$ of $h$. In other words, the \textit{local monodromy} of points lying over $q$ are conjugated to each other. 
\end{remark}

Lemma \ref{Lemma:BsInv} applies to $H^{1,0}(C)^\chi$, so the base locus of $\vert K_C\vert^\chi$ is 
\[
Bs(\vert K_C\vert^\chi)= \sum_{q} t_q^\chi \lambda^{-1}(q),
\]
for some natural integers $t_q^\chi$, that we still need to determine. 

We denote by $\rho_\chi$ an irreducible representation of $G$ of character $\chi$. 
\\
We have the following 
\begin{lemma}\label{lem: t_q_eq_to}
	Let us fix a point $q\in C/G$ of ramification index $m_q$. Let $h$ be the local monodromy of a point $p\in \lambda^{-1}(q)$, hence $o(h)=m_q$.  It there exists
	\[
	a_q^\chi \in \lbrace j\in [0,\dots, m_q-1] \colon e^{\frac{2\pi i}{m_q}j}\in Spec (\rho_\chi(h))\rbrace 
	\]
	and a non-negative integer $0\leq k_q^{\chi}<k_q\leq g(C/G)+1$ such that 
	\[
	t_q^\chi=m_q-a_q^\chi-1+ k_q^\chi m_q,
	\]
	where $k_q$ is the minimal non-gap of $q$ in the Definition \ref{defn: k_p_Weierstrass}. 
	\\
	The values $a_q^\chi$ and $k_q^\chi$ depends only from $q$ and $\chi$ and not by the choice of $p\in \lambda^{-1}(q)$.
\end{lemma}
\begin{proof}
	We observe that the action on $H^{1,0}(C)^\chi$ of $h$ is diagonalizable, and its spectrum is contained in the set of the $m_q$-roots of the unity. 
	Hence the action of $h$ decomposes $H^{1,0}(C)^\chi$ as
	\[
	H^{1,0}(C)^\chi= \bigoplus_{j=0}^{m_q-1}V_j,
	\]
	where $V_j$ is the eigenspace of eigenvalue $\xi^j$, and $\xi$ is  the first $m_q$-root of the unity ($V_j$ may be zero, whenever $\xi^j$ is not an eigenvalue of $h$).
	
	Let $\omega_j \in V_j$ be an eigenvector. We determine the vanishing order of $\omega_j$ at the point $p$. By definition of local monodromy, it there exists a local coordinate $z$ such that the action of $h$ in a neighborhood of $p$ is $z\mapsto \xi z$. We write $\omega_j=f(z)dz$ locally around this neighborhood of $p$. We get 
	\[
	\begin{split}
		\xi^jf(z)dz & =h\cdot (f(z)dz) \\
		& =(h^{-1})^*(f(z)dz) \\
		& =f(\xi^{m_q-1}z)\xi^{m_q-1}dz.
	\end{split}
	\]
	Hence $f$ satisfies $f(\xi^{m_q-1}z)= \xi^{j+1}f(z)$, forcing it to be $f=z^{m_q-j-1}g(z^{m_q})$, for some holomorphic function $g$. Hence $\ord_p(\omega_j)$ is congruent to $m_q-j-1$ modulo $m_q$.
	
	Applying Lemma \ref{Lemma:BsInv} to $W=H^{1,0}(C)^\chi$ we find a form $\omega\in H^{1,0}(C)^\chi$ with vanishing order $t_q^\chi$ at each point of $\lambda^{-1}(q)$.
	Let us write $\omega$ as a $\omega=\sum_{j=0}^{m_q-1} \omega_j$, with $\omega_j\in V_j$. Since each $\omega_j$ has different order at $p$, then
	\[
	\begin{split}
		t_q^\chi & =\ord_p(\omega)  = \min_{\omega_j\neq 0}\{\ord_p(\omega_j)\}.
	\end{split}
	\]
	In other words, it there exists $j_0\in [0, \dots, m_q-1]$ such that $t_q^\chi= \ord_p(\omega_{j_0})$.
	
	Since $\omega_{j_0}$ is an eigenvector of eigenvalue $\xi^{j_0}$, then $t_q^\chi=\ord_p(\omega_{j_0})$ is congruent to $m_q-j_0-1$ modulo $m_q$; let us say $t_q^\chi=m_q-j_0-1+k_{j_0}m_q$, for some non-negative integer $k_{j_0}$.
	
	We claim that $k_{j_0}<k_q$. By contradiction, if $k_{j_0}\geq k_q$, then we use the definition of $k_q$ to pick up a meromorphic  function  $f\in \mathcal M(C/G)=\mathcal M(C)^G$ with only one pole at $q$ of order $\ord_q(f)=-k_q$. In this case, then $f\omega$ is a holomorphic form. Indeed, by definition of $f$, the only poles of $f\omega$ that may occur lie on $\lambda^{-1}(q)$, but the order of $f\omega$ at each $g\cdot p\in \lambda^{-1}(q)$ is 
	\[\begin{split}
		\ord_{g\cdot p}(f\omega) & =\ord_{g\cdot p}(\omega)+\ord_{g\cdot p}(f) \\
		& =t_q^\chi-k_qm_q \\
		& =m_q-j_0-1+(k_{j_0}-k_q)m_q\geq 0.
	\end{split}\]
	Furthermore, from Lemma  \ref{lem: f_merom_as_G-equiv_isom}, then $f\omega\in H^{1,0}(C)^\chi$. However, this would contradict the definition of $t_q^\chi$, since $\ord_p(f\omega)=t_q^\chi-k_qm_q<t_q^\chi$. 
	
	To summarize, we have proved 
	\[
	t_q^\chi =m_q-j_0-1+k_{j_0}m_q,
	\]
	where $j_0$ is one of the integers such that $\xi^{j_0}\in Spec(\rho_\chi(h))$, and $k_{j_0}<k_q$. 
	
	It is straightforward to see that such integers $j_0$ and $k_{j_0}$ do not depend  from the choice of $p\in\lambda^{-1}(q)$.
\end{proof}
\begin{theorem}(Base locus formula)\label{thm: Base_locus_formula}
	The base locus of $\vert K_C\vert^\chi$ is 
	\[
	Bs(\vert K_C\vert^\chi)= \sum_{q} \left(m_q-a_q^\chi-1+ k_q^\chi m_q\right) \lambda^{-1}(q),
	\]
	where the non-negative integers $a_q^\chi$ and $k_q^\chi$ are those defined in Lemma \ref{lem: t_q_eq_to}. 
\end{theorem}
\begin{proof}
	It suffices to apply Lemma \ref{lem: t_q_eq_to} to each point $q\in C/G$. 
\end{proof}
\begin{remark}\label{rem: direct_comp_t_qchi}
	Under suitable assumptions it is possible to determine exactly $a_q^\chi$ and $k_q^\chi$. 
	
	For instance, if $C/G\cong \mathbb P^1$, then $k_q=g(C/G)+1=1$, for any $q\in \mathbb P^1$. Hence $k_q^\chi=0$, and we get 
	\[
	t_q^\chi=m_q-a_q^\chi-1.
	\] 
	Moreover, if one of the following holds
	\begin{itemize}
		\item $\chi$ is an irreducible character of degree $1$, or
		\item the local monodromy $h$ is in the centre of $G$, 
	\end{itemize}
	then $\rho_\chi(h)=\frac{\chi(h)}{\chi(1)}\cdot \Id$ is a multiple of the identity. 
	
	This is obvious when the character has degree one. Instead, when the local monodromy is central, this is a result we take from  \cite{Cat18}.
	
	Under one of these two conditions, then
	$a_q^\chi\in [0,\dots, m_q-1]$ is the only integer such that $\chi(h)=e^{\frac{2\pi i}{m_q}a_{q}^\chi}\chi(1)$.
\end{remark}
%
We deduce then the following immediate consequence from Theorem \ref{thm: Base_locus_formula} and Remark \ref{rem: direct_comp_t_qchi}:
\begin{corollary}\label{cor: explicit_comp_base_locus_form}
	Assume $C/G\cong \mathbb P^1$, and $\chi$ is an irreducible character of degree $1$. Then 
	\[
	Bs(\vert K_C\vert^\chi)= \sum_{q} \left(m_q-a_q^\chi-1\right) \lambda^{-1}(q),
	\]
	where $a_q^\chi\in [0,\dots m_q-1]$ is the only non-negative integer such that $\chi(h)=e^{\frac{2\pi i }{m_q}a_q^\chi}$, with $h$ local monodromy of a point $p$ over $q$. 
\end{corollary}

\subsection{The canonical system of a product-quotient surface}\label{subsec: can_map_PQ} 
Let us consider a product-quotient surface $S$ given by a pair of curves $C_1$ and $C_2$ and a finite group $G$ acting (faithfully) on both of them. Let $X:=\left(C_1\times C_2\right)/G$ be the the quotient model of $S$. 
\\
According to the previous section, then $G$ induces the canonical representation on $H^{1,0}(C_i)$; let  $\chi_{can}^i$ be their canonical characters  respectively, $i=1,2$.
\begin{theorem}\label{ThmFormulaH^2,0(S)}
	Every $G$-invariant global holomorphic $2$-form of $C_1\times C_2$ extends uniquely to a global holomorphic $2$-form on the minimal resolution of the singularities $\rho\colon S\to X$ of $X$. It holds
	\begin{equation}\label{FormulaH^2,0(S)}
		H^{2,0}(S)=H^{2,0}(C_1\times C_2)^G=\bigoplus_{\chi \in Irr(G)}\left(H^{1,0}(C_1)^\chi\otimes H^{1,0}(C_2)^{\overline{\chi}}\right)^G.
	\end{equation} 
	Furthermore,
	\[
	p_g(S)=\sum_{\chi\in Irr(G)}\langle \chi_{can}^1, \chi\rangle \cdot \langle \chi_{can}^2 , \overline{\chi}\rangle.
	\]
\end{theorem}
\begin{proof}
	Denote by $X^\circ$ the smooth locus of $X$, i.e. the locus of the image of that points of $C_1\times C_2$ with trivial stabilizer. Each global holomorphic $2$-form of $X^\circ$ extends uniquely to a global holomorphic $2$-form of  $C_1\times C_2$, via the pullback map $\lambda_{12}^*\colon H^{2,0}(X^\circ)\to H^{2,0}(C_1\times C_2)$, resulting a monomorphism onto the invariant subspace $H^{2,0}(C_1\times C_2)^G$. On the other side, the minimal resolution of the singularities $\rho\colon S\to X$ is an isomorphism on $X^\circ$, hence $\left(\rho^{-1}\right)^*\colon H^{2,0}(S)\to H^{2,0}(X^\circ)$ is a monomorphism. Furthermore, each global holomorphic $2$-form on the smooth locus $X^\circ$ of $X$ extends uniquely to a global holomorphic $2$-form on $S$, by Freitag's theorem \cite[Satz 1]{FR71}, so $\left(\rho^{-1}\right)^* $ is an epimorphism too. 
	\\
	Thus $H^{2,0}(S)$ is sent isomorphically via $\lambda_{12}^*\circ (\rho^{-1})^*$ onto the invariant subspace $H^{2,0}(C_1\times C_2)^G\subseteq H^{2,0}(C_1\times C_2)$. 
	Finally, by applying K\"unneth formula and writing $H^{1,0}(C_i)$ as the direct sum of isotypic components, we get 
	\[
	H^{2,0}(C_1\times C_2)^G=\bigoplus_{\chi,\eta \in Irr(G)}\left(H^{1,0}(C_1)^\chi\otimes H^{1,0}(C_2)^{\eta}\right)^G.
	\]
	Formula \eqref{FormulaH^2,0(S)} follows just from Schur lemma. Indeed, the dimension of any piece of the sum is $\langle \chi_{can}^1, \chi\rangle \cdot \langle \chi_{can}^2 , \eta\rangle \cdot \langle \chi\eta, 1\rangle$. However $\langle \chi\eta, 1\rangle=\langle \chi ,\overline{\eta}\rangle$, which is equal to $1$ only for $\eta=\overline{\chi}$, and $0$ otherwise. 
\end{proof}
\begin{remark}
	Using an analogous proof such as that of Theorem \ref{ThmFormulaH^2,0(S)} one can say in general that
	\[
	H^{i,0}(S)=H^{i,0}(C_1\times C_2)^G
	\]
	by Freitag's theorem \cite[Satz 1]{FR71}.
	Hence, another immediate consequence firstly observed by Serrano in \cite[Prop. 2.2]{Ser96} is a formula for the irregularity of $S$:
	\[
	q(S)=g(C_1/G)+g(C_2/G). 
	\]
	In particular, $S$ is regular if and only if $C_i/G\cong \mathbb{P}^1$. 
\end{remark}
Let us remind the following classical lemma of representation theory:
\begin{lemma}\label{GenSpazioInvariante}
	Let us consider an irreducible representation $\phi_\chi$ afforded by a character $\chi$, of degree $n:=\chi(1)$. Consider a basis $v_1,\dots, v_n$ of $V$ and its dual basis $e_1,\dots, e_n$ of $V^*$.  Then $(V\otimes V^*)^G$ is one-dimensional and it is generated by $v_1\otimes e_1+\dots +v_n\otimes e_n$. 
\end{lemma}
We use the previous lemma to describe a basis of $(H^{1,0}(C_1)^\chi\otimes H^{1,0}(C_2)^{\overline{\chi}})^G$.
\begin{remark}\label{basis_inv_subs}
	Let us consider an irreducible representation $\phi_\chi\colon G\to GL(V)$ of character $\chi$. Let $n:=\chi(1)$ be the degree of $\phi_\chi$. Then $H^{1,0}(C_1)^\chi\otimes H^{1,0}(C_2)^{\overline{\chi}}$ is the direct sum of certain  number of copies of $V\otimes V^*$ (the exact number of copies is $\langle \chi_{can}^1, \chi\rangle \cdot \langle \chi_{can}^2 , \overline{\chi}\rangle$). Consequently its invariant subspace $(H^{1,0}(C_1)^\chi\otimes H^{1,0}(C_2)^{\overline{\chi}})^G$ is a direct sum of the same number of copies of the invariant subspace $(V\otimes V^*)^G$. 
	Let us fix a basis $\lbrace \omega_1,\dots, \omega_n \rbrace$ of $V$ and the (dual) basis $\lbrace \eta_1,\dots, \eta_n\rbrace$ on $V^*$.
	Hence, denote by $\lbrace \omega_1^k,\dots, \omega_n^k\rbrace$ the corresponding basis of the $k$-th copy of $V$ on $H^{1,0}(C_1)^\chi$, $k=1,\dots, \langle \chi_{can}^1,\chi\rangle$ [resp. by $\lbrace \eta_1^l,\dots, \eta_n^l\rbrace$ the corresponding basis of the $l$-th copy of $V^*$ on $H^{1,0}(C_2)^{\overline{\chi}}$, $l=1,\dots, \langle \chi_{can}^2, \overline{\chi}\rangle$].
	Lemma \ref{GenSpazioInvariante} applies for any copy of $(V\otimes V^*)^G$, so that
	\begin{equation}
		(H^{1,0}(C_1)^\chi\otimes H^{1,0}(C_2)^{\overline{\chi}})^G=\bigoplus_{k,l} \ \langle \omega_1^k\otimes \eta_1^l+\dots+ \omega_n^k\otimes \eta_n^l\rangle.
	\end{equation}
\end{remark}
\begin{definition}
	We denote by $\vert K_{C_1\times C_2}\vert^G$ the linear subsystem of the canonical system of $C_1\times C_2$  given by the subspace of invariant $2$-forms of $C_1\times C_2$.
\end{definition}
We give a theoretical description of the canonical map $\Phi_{K_S}$ of $S$. From Theorem \ref{ThmFormulaH^2,0(S)}, the (rational) map $\Phi_{K_S}\circ \lambda_{12}$ 
is induced by the linear subsystem $\vert K_{C_1\times C_2}\vert^G$. 
The situation is the following:
\[\begin{tikzcd}
	& X \\
	{C_1\times C_2} && S &&& {\mathbb P^{p_g-1}} \\
	\\
	{\mathbb P^{g_1-1}\times \mathbb P^{g_2-1}} && {\mathbb P^{g_1g_2-1}}
	\arrow[from=2-1, to=4-1]
	\arrow["Segre"{description}, from=4-1, to=4-3]
	\arrow["{\lambda_{12}}", dashed, from=2-1, to=2-3]
	\arrow["{\lambda_{12}}", from=2-1, to=1-2]
	\arrow["\rho"', from=2-3, to=1-2]
	\arrow["{\Phi_{K_{C_1\times C_2}}}"{description}, from=2-1, to=4-3]
	\arrow["{\Phi_{K_S}}", dashed, from=2-3, to=2-6]
	\arrow["proj"', dashed, from=4-3, to=2-6]
	\arrow["{\Phi_{\vert K_{C_1\times C_2}\vert^G}}"{description}, dashed, from=2-1, to=2-6, bend right=18pt].
\end{tikzcd}\]
Let us fix a basis of $H^{1,0}(C_1)$ and $H^{1,0}(C_2)$. Then $\Phi_{K_S}\circ \lambda_{12}$ is the composition of the product of the canonical maps of $C_1$ and $C_2$ with the Segre embedding in $\mathbb{P}^{g_1g_2-1}$, together with the projection map $proj$. This latter map sends a basis of $2$-forms of $C_1\times C_2$ to a basis of invariant $2$-forms defining $\Phi_{K_S}$.
\\
We can use Remark \ref{basis_inv_subs} to give an explicit description of $proj$, which is defined in coordinates as follows:  
\\
Let us fix coordinates $^\chi x_{ij}^{kl}$ on $\mathbb{P}^{g_1g_2-1}$, with $1\leq i,j \leq \chi(1)$, and $1\leq k\leq  \langle \chi_{can}^1, \chi\rangle $, $1\leq l\leq  \langle \chi_{can}^2, \overline{\chi}\rangle$. Then 
\begin{multline*}
	proj\left(\left(^\chi x_{ij}^{kl} : \chi, i,j,k,l\right)\right)=\\ \left(^\chi x_{11}^{kl} +\dots + \   ^\chi x_{nn}^{kl} : \quad \chi \in Irr(G), n=\chi(1), k,l\right).
\end{multline*}
\subsection{Base locus of the invariant subsystem $\vert K_{C_1\times C_2}\vert^G$}\label{subsec: base_locus_PQ}
Given an irreducible character $\chi\in Irr(G)$, we have the following series of inclusions
\[
\left(H^{1,0}(C_1)^\chi\otimes H^{1,0}(C_2)^{\overline{\chi}}\right)^G\subseteq H^{1,0}(C_1)^\chi\otimes H^{1,0}(C_2)^{\overline{\chi}}\subseteq H^{2,0}(C_1\times C_2).
\]
Let us define the associated subsystems of $\vert K_{C_1\times C_2}\vert$ given by these subspaces.
\begin{definition}
	We denote by $\vert K_{C_1}\vert^\chi\otimes \vert K_{C_2}\vert^{\overline{\chi}}$ and by $\left(\vert K_{C_1}\vert^\chi\otimes \vert K_{C_2}\vert^{\overline{\chi}}\right)^G$  the associated subsystems of 
	the canonical linear system of $C_1\times C_2$  given by $H^{1,0}(C_1)^\chi\otimes H^{1,0}(C_2)^{\overline{\chi}}$  and $\left(H^{1,0}(C_1)^\chi\otimes H^{1,0}(C_2)^{\overline{\chi}}\right)^G$ respectively. 
\end{definition}
Theorem \ref{ThmFormulaH^2,0(S)} permits us to describe the base locus of $\vert K_{C_1\times C_2}\vert^G$ in terms of the base locus of its pieces $\left(\vert K_{C_1}\vert^\chi\otimes \vert K_{C_2}\vert^{\overline{\chi}}\right)^G$, $\chi\in Irr(G)$. Precisely, we have 
\begin{equation}\label{eq: base_locus_KC1xC2G}
	Bs(\vert K_{C_1\times C_2}\vert^G)=\bigcap_{\langle \chi_{can}^1,\chi\rangle \neq 0, \ \langle \chi_{can}^2, \overline{\chi}\rangle \neq 0}Bs(\left(\vert K_{C_1}\vert^\chi\otimes \vert K_{C_2}\vert^{\overline{\chi}}\right)^G).
\end{equation}
{\bf Notation:} Let us denote by 
\[
B_q^{vert}:=\lbrace q \rbrace \times  C_2/G, \qquad \makebox{and} \qquad B_l^{hor}:= C_1/G\times \lbrace l\rbrace, 
\]
where $q\in C_1/G$ and  $l\in C_2/G$. Instead, $R_q^{vert}$ and $R_l^{hor}$ denote the reduced inverse images on $C_1\times C_2$ of $B_q^{vert}$ and $B_l^{hor}$:
\[
R_q^{vert}:=\frac{1}{m_q}\left(\lambda\circ \lambda_{12}\right)^*\left(\lbrace q \rbrace \times C_2/G\right), \qquad R_l^{hor}:=\frac{1}{m_l}(\lambda\circ \lambda_{12})^*\left(C_1/G\times \lbrace l\rbrace\right).
\]
\begin{remark}
	With this notation, then the branch locus of $\lambda\circ\lambda_{12}\colon C_1	\times C_2\to C_1/G\times C_2/G$ is the grid 
	\[
	B_q^{vert}:=\lbrace q \rbrace \times C_2/G, \qquad \makebox{and} \qquad B_l^{hor}:=C_1/G\times \lbrace l\rbrace
	\]
	with $q\in Crit(\lambda_1)$ and $l\in Crit(\lambda_2)$. 
\end{remark}
\textit{Base Locus formula} theorem \ref{thm: Base_locus_formula} provides a formula for the base locus of  $\vert K_{C_1}\vert^\chi\otimes \vert K_{C_2}\vert^{\overline{\chi}}$.
\begin{theorem}\label{thm: base_locus_H10C1chi_otimes_H10C2_overline_chi}
	The (schematic) base locus of the linear  subsystem $\vert K_{C_1}\vert^\chi$ $\otimes \vert K_{C_2}\vert^{\overline{\chi}}$ of $\vert K_{C_1\times C_2}\vert$ is pure in codimension $1$ and is equal to 
	\begin{equation}\label{eq:Corollary_Base_Locus}
		Bs(\vert K_{C_1}\vert^\chi\otimes \vert K_{C_2}\vert^{\overline{\chi}} )=\sum_{q\in Crit(\lambda_1)} t_q^\chi R_q^{vert}+ \sum_{l\in Crit(\lambda_2)}t_l^{\overline{\chi}}R_l^{hor}
	\end{equation}
	where
	$t_{q}^\chi$ and  $t_{l}^{\overline{\chi}}$ are the non-negative integers of Lemma \ref{lem: t_q_eq_to}.
\end{theorem}
\begin{corollary}\label{cor: base_locus_Kc1c2_chi_G}
	Let $\chi$ be a character of degree $1$. Then 
	\[
	(H^{1,0}(C_1)^\chi\otimes H^{1,0}(C_2)^{\overline{\chi}})^G=H^{1,0}(C_1)^\chi\otimes H^{1,0}(C_2)^{\overline{\chi}}
	\]
	and the base locus of its associated linear subsystem $\left(\vert K_{C_1}\vert^\chi\otimes \vert K_{C_2}\vert^{\overline{\chi}}\right)^G=\vert K_{C_1}\vert^\chi\otimes \vert K_{C_2}\vert^{\overline{\chi}}$ is given by the formula \eqref{eq:Corollary_Base_Locus} of Theorem  \ref{thm: base_locus_H10C1chi_otimes_H10C2_overline_chi}. 
	
	Assume furthermore that $C_i/G\cong \mathbb P^1$, for $i=1,2$. Then $t_q^\chi$ and $t_l^{\overline{\chi}}$ of \eqref{eq:Corollary_Base_Locus} are the unique non-negative integers with 
	$0\leq t_q^\chi\leq m_q-1$ and $0\leq t_l^{\overline{\chi}}\leq m_l-1$ satisfying 
	\[
	\chi(h)=e^{\frac{2\pi i }{m_q}\left(m_q-t_q^\chi-1\right)} \qquad \makebox{and} \qquad \chi(g)=e^{\frac{2\pi i }{m_l}\left(t_l^{\overline{\chi}}+1\right)},
	\]
	where $h$ is the local monodromy of a point over $q$, and $g$ is the local monodromy of a point over $l$. 
\end{corollary}
\begin{proof}
	The first sentence is straightforward, since every $v\otimes w\in H^{1,0}(C_1)^\chi\otimes  H^{1,0}(C_2)^{\overline{\chi}}$ is $G$-invariant
	\[
	g\cdot \left( v\otimes w\right)=(\chi(g)v)\otimes (\overline{\chi}(g)w)=\vert \chi(g)\vert v\otimes w=v\otimes w. 
	\]
	The rest of the thesis follows from Remark \ref{rem: direct_comp_t_qchi}.
\end{proof}
\begin{lemma}\label{lem: fixed_part_Kc1c2G}
	Suppose $S$ satisfies \hyperlink{Property $\left(\#\right)$}{Property $\left(\#\right)$}.
	Then the fixed part of  the linear system $\vert K_{C_1\times C_2}\vert^G$ is
	\begin{equation}\label{eq: fixed_part_Kc1c2G}
		\begin{split}
			Fix\left(\vert K_{C_1\times C_2}\vert^G\right)=\sum_{q\in Crit(\lambda_1)} & \left(\min_{\chi \colon \langle \chi_{can}^1,\chi\rangle \neq 0,  \langle \chi_{can}^2,\overline{\chi}\rangle \neq 0} t_q^\chi \right)R_q^{vert} +\\ & \sum_{l\in Crit(\lambda_2)}\left(\min_{\chi \colon \langle \chi_{can}^1,\chi\rangle \neq 0, \langle \chi_{can}^2,\overline{\chi}\rangle \neq 0} t_l^{\overline{\chi}} \right)R_l^{hor}.
		\end{split}
	\end{equation}
\end{lemma}
\begin{proof}
	The fixed part of $\vert K_{C_1\times C_2}\vert^G$ is the common divisor of the fixed parts of those pieces $\left(\vert K_{C_1}\vert^\chi\otimes \vert K_{C_2}\vert^{\overline{\chi}}\right)^G$ that are non-empty, for $\chi$ irreducible character. 
	By \hyperlink{Property $\left(\#\right)$}{Property $\left(\#\right)$} , then $\chi$ is of degree $1$, hence Corollary \ref{cor: base_locus_Kc1c2_chi_G} applies and the fixed part of $\left(\vert K_{C_1}\vert^\chi\otimes \vert K_{C_2}\vert^{\overline{\chi}}\right)^G$ 
	is amount to 
	\[
	\sum_{q\in Crit(\lambda_1)} t_q^\chi R_q^{vert}+ \sum_{l\in Crit(\lambda_2)}t_l^{\overline{\chi}}R_l^{hor}.
	\]
	The common divisor of these fixed parts is the right member of \eqref{eq: fixed_part_Kc1c2G}.
\end{proof}
Let $M$ be the mobile part of $\vert K_{C_1\times C_2}\vert^G$. By definition of $M$, then 
\[
M\equiv K_{C_1\times C_2}-Fix(\vert K_{C_1\times C_2}\vert^G).
\]
Suppose $S$ satisfies \hyperlink{Property $\left(\#\right)$}{Property $\left(\#\right)$}. Thus $Fix(\vert K_{C_1\times C_2}\vert^G)$ is a union of fibres from Lemma \ref{eq: fixed_part_Kc1c2G}. To compute $M^2$ is then sufficient to know the intersection product of 
\[
K_{C_1\times C_2}\cdot R_q^{vert}, \qquad K_{C_1\times C_2}\cdot R_l^{hor}, \qquad \left(R_q^{vert}\right)^2, \qquad \left(R_l^{hor}\right)^2, \quad  R_q^{vert}\cdot R_l^{hor}. 
\]
We compute them. 
\\
$R_q^{vert}$ can be written as sum of $\vert G\vert/m_q$ components $\lbrace g\cdot p\rbrace \times C_2$, with $p$ point over $q$, and $g\in G$. $\lbrace g\cdot p\rbrace \times C_2$ has self-intersection zero (since two points are always homologous on a connected variety, and then the fibres of $C_1\times C_2\to C_1$ are always numerically equivalent). Thus we can use \textit{genus formula} to get 
\[
K_{C_1\times C_2}\cdot\left( \lbrace g\cdot p\rbrace \times C_2\right)=2g(C_2)-2-\left( \lbrace g\cdot p\rbrace \times C_2\right)^2=2g(C_2)-2.
\]
The same reasoning works for an horizontal divisor $R_l^{hor}$. Thus, we have got 
\[
K_{C_1\times C_2}\cdot R_q^{vert}=\frac{\vert G\vert}{m_q}\left(2g(C_2)-2\right), \qquad K_{C_1\times C_2}\cdot R_l^{hor}=\frac{\vert G\vert}{m_l}\left(2g(C_1)-2\right).
\]
Analogously, 
\[
\left(R_q^{vert}\right)^2= \left(R_l^{hor}\right)^2 =0, \qquad \makebox{and} \qquad 
R_q^{vert}\cdot R_l^{hor}=\frac{\vert G\vert^2}{m_qm_l}.
\]
\subsection{A formula for the degree of the canonical map}\label{subsec: deg_can_map}
In the previous subsection we have seen that the (a priori rational) map $\Phi_{K_S}\circ \lambda_{12}$ is induced by the linear subsystem $\vert K_{C_1\times C_2}\vert^G $, which is generated by $p_g$ invariant $2$-forms defining $\Phi_{K_S}$:
\[\begin{tikzcd}
	{C_1\times C_2} & S & {\mathbb P^{p_g-1}}
	\arrow["{\lambda_{12}}", dashed, from=1-1, to=1-2]
	\arrow["{\Phi_S}", dashed, from=1-2, to=1-3]
	\arrow["{\Phi_{\vert K_{C_1\times C_2}\vert^G}}"', dashed, from=1-1, to=1-3, bend right=35pt]
\end{tikzcd}\]
We 
\textit{resolve  the indeterminacy} of  $\Phi_{\vert K_{C_1\times C_2}\vert^G}=\Phi_{K_S}\circ \lambda_{12}$ 
by a sequence of  blowups: 
\[
\xymatrix{
	\widehat{C_1\times C_2} \ar[r] \ar[dr]_{\Phi_{\widehat{M}}} & C_1\times C_2\ar@{-->}[d]^{\Phi_{\vert K_{C_1\times C_2}\vert ^G}} \\ 
& \mathbb P^{p_g-1}.
}
\]
Here the morphism $\Phi_{\widehat{M}}$ is induced by the base-point free linear system $ \widehat{M}$ obtained as follow: let $M $ be the mobile part of  $\vert K_{C_1\times C_2}\vert^G$. 
\\
We blow up the base-points of  $ M$, take the pullback of  $M$ and remove the fixed part of this new linear system. We repeat the procedure, until we obtain a  base-point free  linear system $\widehat{M} $.  
\begin{lemma}\label{lem: when_map_not_comp_with}
The map $\Phi_{K_S}$ is not composed with a pencil if and only if $\widehat{M}^2$ is positive.
\end{lemma}
\begin{proof}
The map $\Phi_{K_S}$ is composed with a pencil if and only if $\Phi_{\widehat{
		M}}$ is composed with a pencil. 
The image $\Sigma$ of $\Phi_{\widehat{
		M}}$ is a curve if and only if we are able to pick-up two general hyperplanes $H_1$ and  $H_2$ of $\mathbb P^{p_g-1}$ such that $H_{|_\Sigma}^2=H_1\cdot H_2\cdot \Sigma=0$. However, $\widehat{M}=\Phi_{\widehat{M}}^*(H)$, hence $H_{|_\Sigma}^2$ is zero if and only if $\widehat{M}^2$  is equal to zero. 
\end{proof}
Let us suppose $\widehat{M}^2>0$, so that 
$\Phi_{K_S}$ has image $\Sigma$ of dimension $2$. In this case, then $\Phi_{\widehat{M}}$ is a finite morphism, and by projection formula
\[
\widehat{M}^2=\deg (\Phi_{\widehat{M}})\deg(\Sigma)=\deg(\Phi_{K_S})\deg(\Sigma)\vert G\vert,
\]
which gives Formula \eqref{eq: formula_deg_phiS}.

\subsection{The correction term formula} 
\label{subsec: correction_term_formula}
As remarked in the introduction of this chapter, $M^2-\widehat{M}^2$  is 
the sum of the correction terms arising from each isolated base-point of $M$, the mobile part of the linear subsystem $\vert K_{C_1\times C_2}\vert^G$.

The contribution to the correction term of any isolated base-point may be easily computed whenever $S$ satisfies \hyperlink{Property $\left(\#\right)$}{Property $\left(\#\right)$}.

Let us fix a base-point $(p_1,p_2)\in C_1\times C_2$ of the mobile part $M$.  The point $p_1$ is over $q\in C_1/G$ and $p_2$ is over $l\in C_2/G$. Let us fix an irreducible character $\chi$. We can always choose a general basis of $H^{1,0}(C_1)^\chi$ such that each one-form of the basis has the minimum vanishing order $t_q^\chi$  at $p_1$, which is the natural integer computed in Lemma \ref{lem: t_q_eq_to}. 
\\
Similarly, we can choose a general basis of $H^{1,0}(C_2)^{\overline{\chi}}$ such that each one-form of the basis has minimum vanishing order $t_l^{\overline{\chi}}$ at $p_2$. The choice of this pair of bases gives via tensor product a natural basis of $H^{1,0}(C_1)^\chi\otimes H^{1,0}(C_2)^{\overline{\chi}}$, which is a $G$-invariant subspace since \hyperlink{Property $\left(\#\right)$}{Property $\left(\#\right)$} holds, namely $\chi$ is of degree one. This permits us to conclude that the divisors  spanning the linear  subsystem $\vert K_{C_1\times C_2}\vert^G$ can be written in a neighbourhood of $(p_1, p_2)$ as
\[
t_q^\chi R_q^{vert}+ t_l^{\overline{\chi}}R_l^{hor}, \qquad \chi \makebox{ \ such that \ }  \quad \langle \chi_{can}^1,\chi\rangle \neq 0, \langle \chi_{can}^2,\overline{\chi}\rangle \neq 0.
\]
Finally, it is sufficient to remove the fixed part of $\vert K_{C_1\times C_2}\vert^G$ computed in Lemma \ref{lem: fixed_part_Kc1c2G} to get how the divisors spanning $M$ are written in a neighbourhood of $(p_1, p_2)$. 
So, the linear system $M$ is spanned by $p_g$ divisors locally near $(p_1, p_2)$ of the form 
\[
a_1R_q^{vert}+b_1R_l^{hor}, \qquad \dots  \qquad a_{p_g}R_q^{vert}+b_{p_g}R_l^{hor}.
\]
Since we assumed that $(p_1,p_2)$ is a base-point and $M$ has not fixed components, then without loss of generality $a_1=b_2=0$. 
\\
Note that $R_q^{vert}$ and $R_l^{hor}$ are smooth and intersect transversally at $(p_1, p_2)$.

In Theorem \ref{thm:Correction_term_formula} we give a general formula to compute directly the contribution of $(p_1, p_2)$ to the correction term $M^2-\widehat{M}^2$ whenever $p_g$ is equal to three. 

The rest of this subsection proves Theorem \ref{thm:Correction_term_formula}. 
\vspace{0,3cm}

Let us consider a slightly more general setting: let  $ M $ be a (not necessarily complete) two-dimensional linear system on a surface $S$ spanned by $D_1$, $D_2$, and $D_3$. Assume that $M$ has only isolated base-points, smooth for $S$, and that in a neighborhood of a basepoint $p$ we can write  the divisors 
$D_i$ as 
\[
D_1=aH, \quad D_2=bK \quad \makebox{and} \quad D_3= cH+d K.
\]
Here $H$ and $K$ are reduced, smooth, and intersect transversally at $p$ and $a,b,c,d$ are non-negative integers, $b\leq a$. 

Let $\widehat{M}$ be the linear system obtained as follows:
we blow-up the basepoint $p$, take the pullback of the mobile part of $ M $ and remove the fixed part of this new linear system. If an infinitely near point of $p$ is a base-point for this linear system, then repeat the procedure, until we obtain a (not necessarily complete) linear system $\widehat{M}$
such that no infinitely near point of $p$ is a base point of $\widehat{M} $.  

\textbf{Definition.} The linear system $\widehat{M}$ is called \textit{strict transform} of $ M $ at $p$. 

Firstly, we present a stronger version of \cite[Lemma 2]{fede1}.
\begin{lemma}\label{FedericoLemma} 
	Assume that $bc+ad\geq ab$. Then $M^2-\widehat{M}^2 =ab$.
\end{lemma}
\begin{proof}
	We prove the lemma by induction on $(a,b)$, with $b\leq a$. Here we are considering the lexicographic order $\leq$ defined on the lower half plane $\Delta^\geq:=\{(a,b) \colon a\geq b\}\subseteq \mathbb{N}\times \mathbb{N}$ as follows:
	\[
	(a',b')\leq (a,b) \text{\ if \ and \ only \ if \ }  a'< a  \text{ \ or \ } a'=a \text{ \ and \ } b'\leq b.
	\]
	In this case, $\Delta^\geq$ admits the \textit{well-ordering principle} and so it holds the \textit{mathematical induction}. \\
	Suppose that $(a,b)=0$. Then $M$ is base-point free and so $\widehat{M}^2=M^2=M^2-ab$. Now suppose that the statement is true for $(a',b')<(a,b)$. We aim to prove it for $(a,b)$. We blow up the base-point $p$,  take the pullback of the divisors $D_i$, and remove the fixed part, which is the exceptional divisor  $bE$  of the blowup. In fact the pullback of $D_3$ contains $c+d$ times $E$ and $ c+d\geq b$, thanks to $b\leq a$ and to the assumption $bc+ad\geq ab$:
	\[
	a(c+d)\geq bc+ad\geq ab, \qquad \makebox{so} \qquad c+d\geq b.
	\]
	Restricted to  the preimage of our neighborhood of $p$,  these divisors are: 
	\[
	a\widehat{H}+(a-b)E, \qquad  b\widehat{K}\qquad \makebox{and} \qquad c\widehat{H}+d\widehat{K} +(c+d-b)E. 
	\]
	Here, $\widehat{H}$ and $\widehat{K}$  are the strict transforms of $H$ and $K$. Let $\widehat{M}$ be the linear system generated 
	by these three divisors, then $\widehat{M}^2=M^2-b^2$. If $a=b$ or $b=0$, then $\widehat{M}$ is base-point free and we are done. Otherwise, on the preimage, the linear system $\widehat{M}$ has precisely one 
	new base-point: the intersection point  of $\widehat{K}$ and  $E$. Locally near this point the three divisors spanning $\widehat{M}$ are: 
	\[
	(a-b)E, \qquad  b\hat{K}\qquad \makebox{and} \qquad d\hat{K} +(c+d-b)E.   
	\]
	
	We need to distinguish two cases, when $(a-b)< b$ or when $(a-b)\geq  b$. In the first case $(a-b)< b$ we get $(b,a-b)<(a,b)$. We define new coefficients $a':=b$, $b':=a-b$, $c':=d$ and $d':=c+d-b$. Otherwise if $(a-b)\geq b$, then $(a-b,b)<(a,b)$, and we define $a':=a-b$, $b':=b$, $c':=c+d-b$, and $d':=d$. For both cases, the new coefficients fulfill the inductive hypothesis, because:
	\\
	Thanks to $bc+ad\geq ab$, we have 
	\[
	\begin{split}
		b'c'+a'd' & =(a-b)d+b(c+d-b) \\ & =ad+bc-b^2 \\
		& \geq ab-b^2=(a-b)b\\
		& =a'b'.
	\end{split}
	\]
	By induction, the self-intersection of the new linear system $\widehat{M}$ is equal to 
	\[
	\widehat{M}^2=(M^2-b^2)-b(a-b)=M^2-ab.
	\]
\end{proof}
\begin{lemma}\label{FedericoLemma2} Assume that $bc+ad\leq ab$. Then $ M^2-\widehat{M}^2=ad+bc$.
	\end{lemma}
	\begin{proof}
		We prove the lemma by induction, once more on  $(a,b)$, with $b\leq a$. Thus we consider the lexicographic order $\leq $ on $\Delta^\geq$, as we have done in the proof of Lemma  \ref{FedericoLemma}.
		
		Suppose that $(a,b)=0$. Then $ M$ is base-point free and so $\widehat{M}=M^2=M^2-(0d+0c)$.
		Now suppose that the statement is true for $(a',b')<(a,b)$. Our aim is to prove it for $(a,b)$. We blow up the base-point $p$,  take the pullback of the divisors $D_i$, and remove the fixed part, which is the exceptional divisor  $(c+d)E$ of the blowup, if  $c+d\leq b$, or the divisor $bE$, otherwise. Hence we need to distinguish two cases. 
		
		Let us suppose first that $c+d\leq b \ (\leq a)$. Restricted to the preimage of our neighborhood of $p$, the divisors are 
		\[
		a\widehat{H}+(a-(c+d))E, \qquad  b\widehat{K}+(b-(c+d))E\qquad \makebox{and} \qquad c\widehat{H}+d\widehat{K}.
		\]
		Here, $\widehat{H}$ and $\widehat{K}$  are the strict transforms of $H$ and $K$. Let $\widehat{M}$ be the linear system generated 
		by these three divisors, then $\widehat{M}^2=M^2-(c+d)^2$. On the preimage, the linear system $\widehat{M}$ has precisely two new base-points: the intersection points of $\widehat{H}$ and $\widehat{K}$ with $E$. Locally near these points the three divisors spanning $\widehat{M}$ are respectively 
		\[
		a\widehat{H}+(a-(c+d))E, \qquad  (b-(c+d))E\qquad \makebox{and} \qquad c\widehat{H},
		\]
		and
		\[
		(a-(c+d))E, \qquad b\widehat{K}+(b-(c+d))E\qquad \makebox{and} \qquad  d\widehat{K}.
		\]
		We claim that for both points the coefficients of these three divisors satisfy the assumption of Lemma \ref{FedericoLemma}. 
		\\
		Let us verify it for the first point $\widehat{H}\cap E$: if $c\geq (b-(c+d))$, then define $a':=c$, $b':=b-(c+d)$, $c':=a$, and $d':=a-(c+d)$, otherwise define $a':=b-(c+d)$, $b':=c$, $c':=a-(c+d)$, and $d':=a$. For both the cases $d'\geq b'$ so that $b'c'+a'd'\geq a'd' \geq a'b'$.
		\\
		Regarding the second point $\widehat{K}\cap E$, we have: if $d\geq (a-(c+d))$, then define $a':=d$, $b':=a-(c+d)$, $c':=b$, and $d':=b-(c+d)$, otherwise define $a':=a-(c+d)$, $b':=d$, $c':=b-(c+d)$, $d':=b$. In the first case $c'\geq a'$, while in the second case $d'\geq b'$. Therefore we get $ b'c'+a'd'\geq a'b'$ for both cases.
		\\
		Thus Lemma \ref{FedericoLemma} applies for both points and  the self-intersection of  the new linear system $\widehat{M}$ at the final step is amount to
		\[
		\widehat{M}^2=(M^2-(c+d)^2)-(b-(c+d))c-(a-(c+d))d=M^2-(ad+bc).
		\]
		It remains to discuss the case $c+d\geq b$. 
		\\
		As we have already done before, we blow up the base-point $p$, take the pullback of the divisors $D_i$, and remove the fixed part, which this time is the exceptional divisor $bE$ of the blowup. Restricted to  the preimage of our neighborhood of $p$,  these divisors are: 
		\[
		a\widehat{H}+(a-b)E, \qquad  b\widehat{K}\qquad \makebox{and} \qquad c\widehat{H}+d\widehat{K} +(c+d-b)E. 
		\]
		Here $\widehat{M}^2=M^2-b^2$. If $b=0$ or $a=b$, then $\widehat{M} $ is base-point free. In the first case $b=0$, we get $ad=bc+ad\leq ab=0$, so $\widehat{M}^2=M^2-b^2=M^2=M^2-(ad+bc)$, and we are done. In the second case $a=b$, we get, thanks to the assumptions $ad+bc\leq ab$ and $b\leq c+d$, that
		\[
		\begin{split}
			a(c+d) & =ad+bc \\
			& \leq ab \\
			& \leq a(c+d)
		\end{split}, \qquad \makebox{so} \qquad c+d=b=a.
		\]
		Also in this case we are done, because $\widehat{M}^2=M^2-b^2=M^2-(ad+bc)$.
		\\
		It remains to consider when $a-b=0$ or $b=0$ does not hold. In this case, on the preimage, the linear system $\widehat{M}$ would have precisely one 
		new base-point, the intersection point  of $\widehat{K}$ and  $E$. Locally near this point the three divisors spanning $\widehat{M}$ are: 
		\[
		(a-b)E, \qquad  b\hat{K}\qquad \makebox{and} \qquad d\hat{K} +(c+d-b)E.   
		\]
		We need to distinguish two cases, when $(a-b)< b$ or when $(a-b)\geq  b$. In the first case $(a-b)< b$ we get $(b,a-b)<(a,b)$. We define new coefficients $a':=b$, $b':=a-b$, $c':=d$ and $d':=c+d-b$. Otherwise if $(a-b)\geq b$, then $(a-b,b)<(a,b)$, and we define $a':=a-b$, $b':=b$, $c':=c+d-b$, and $d':=d$. For both cases, the new coefficients fulfill the inductive hypothesis, because:
		\\
		Thanks to $bc+ad\leq ab$, we have
		\[
		\begin{split}
			b'c'+a'd' & =(a-b)d+b(c+d-b) \\ & =ad+bc-b^2 \\
			& \leq ab-b^2=(a-b)b\\
			& =a'b'.
		\end{split}
		\]
		By induction, the self-intersection of the new linear system $\widehat{M}$ is equal to 
		\[
		\begin{split}
			\widehat{M}^2 & =(M^2-b^2)-(a'd'+b'c')
			\\ &
			=M^2-b^2-(ad+bc-b^2)
			\\ & =M^2-(ad+bc).
		\end{split}
		\]
	\end{proof}
	By applying Lemma \ref{FedericoLemma} and Lemma  \ref{FedericoLemma2} it follows directly
	\begin{theorem}[Correction Term Formula]\label{thm:Correction_term_formula}
		Let  $M$ be a two-dimensional linear system on a surface $S$ spanned by $D_1$, $D_2$, and $D_3$. Assume that $M$ has only isolated base-points, smooth for $S$, and that in a neighborhood of a basepoint $p$ we can write  the divisors 
		$D_i$ as 
		\[
		D_1=aH, \quad D_2=bK \quad \makebox{and} \quad D_3= cH+d K.
		\]
		Here $H$ and $K$ are reduced, smooth, and intersect transversally at $p$ and $a,b,c,d$ are non-negative integers, $b\leq a$.
		Let $\widehat{M}$ be the \textit{strict transform} of $M$ along $p$. 
		Then
		\[
		M^2-\widehat{M}^2 =\min \left\{ ab,  \ ad+bc
			\right\}  .
			\]
		\end{theorem}
	
	\subsection{Example of the computation of the degree of the canonical map}\label{subsec: example_comp_deg_can_map}
In this section we give an example how to compute the degree of the canonical of a regular product-quotient surface of geometric genus three, whenever \hyperlink{Property $\left(\#\right)$}{Property $\left(\#\right)$} holds. In addition, in this way we also show the main steps in the operation of the MAGMA script for calculating the degree of the canonical map.
	
Let us consider the family of surfaces no. 1. in \cite[Thm 2.3]{fede2}, which have degree of the canonical map $18$. 

Surfaces $S$ of no.1 of  \cite[Thm 2.3]{fede2} can be described by the following pair of spherical systems of generators of the group $G=S_3\times \mathbb Z_3^2$:
\[
\begin{tabular}{c | c | c | c }
	& $q_1$ & $q_2$ & $q_3$   \\
	\hline
	\makebox{branch \ point} & $(1:1)$ & $(0:1)$  & $ (-1:1)$   \\
	\hline
	\makebox{generator} & $ g_1:=\left(\tau,(1,0)\right) $ & $g_2:=(\sigma^2,(2,2)) $  & $g_3:=(\sigma\tau,(0,1)) $  \\
\end{tabular}
\]
\[
\small\begin{tabular}{c | c | c | c| c }
	& $q_1$ & $q_2$  & $ q_3$  	& $q_4$   \\
	\hline
	\makebox{branch point} & $(1:1)$ & $(0:1)$  & $ (1: \lambda)$& $(-1:1)$     \\
	\hline
	\makebox{generator} & $ h_1:=\left(\sigma\tau,0\right) $ & $h_2:=(\sigma,(1,0)) $ & $h_3:=(\Id,(1,1)) $ & $h_4:=(\tau,(1,2))$  \\
\end{tabular}, 
\]
Here $\sigma$ and $\tau$ are a rotation ($3$-cycle) and a reflection (transposition) of the group $S_3$ respectively. 
\\
Instead, $q_i$ are the branch points of the pair of $G$-coverings $C_i\to\mathbb P^1$ defining $S$, and the respective generator of $q_i$ in the tables is the local monodromy of a point over $q_i$. 

Notice that the second covering $C_2\to \mathbb P^1$ depends from one parameter $\lambda$, with $\lambda\neq -1,1$ since $C_2$ is smooth. 

Consider the three irreducible characters of $S_3$, so the trivial character $1$, the character $\textit{sgn}$ computing the sign of a permutation, and the only $2$-dimensional irreducible character $\mu:=\frac{1}{2}\left(\chi_{reg}-sgn-1\right)$, where $\chi_{reg}$ is the character of the regular representation of $S_3$.
\\
Let us also fix a basis $e_1, e_2$ of $\mathbb Z_3^2$  and consider the dual characters $\epsilon_1$, $\epsilon_2$ of $e_1$ and $e_2$, i.e. the characters defined by
\[
\epsilon_i(ae_1+be_2):=\zeta_3^{a\delta_{1i}+b\delta_{2i}}, \qquad \zeta_3:=e^{\frac{2\pi i }{3}},
\]
where $\delta_{ij}$ is the Kronecker delta. 
\vspace{0,3cm}

We apply  Chevalley-Weil formula \cite[Thm. 1.3.3]{G16} to both the curves $C_1$ and $C_2$ defining $S$ to compute the canonical characters $\chi_{can}^1$ and $\chi_{can}^2$ respectively: 
\begin{equation*}
	\begin{split}
	\chi_{can}^1 & =\epsilon_1^2\cdot\epsilon_2^2+sgn\cdot\epsilon_1\cdot\epsilon_2+sgn\cdot\epsilon_2+sgn\cdot\epsilon_1+\mu\cdot\epsilon_1\cdot\epsilon_2+\mu\cdot\epsilon_1^2\cdot\epsilon_2+\mu\cdot \epsilon_1\cdot\epsilon_2^2; \\
		\chi_{can}^2 &= sgn \cdot \epsilon_1^2\cdot \epsilon_2+sgn\cdot\epsilon_1^2\cdot \epsilon_2^2+sgn\cdot \epsilon_1\cdot \epsilon_2+sgn\cdot \epsilon_1+sgn\cdot \epsilon_2^2+\mu\cdot \epsilon_1 \\ &+\mu\cdot\epsilon_2+2\mu\cdot \epsilon_2^2+sgn\cdot\epsilon_1^2+\epsilon_1^2+\mu\cdot \epsilon_1^2+\mu\cdot\epsilon_1\cdot\epsilon_2,
	\end{split}
\end{equation*}
We notice that the irreducible characters $\chi$ such that $\chi$ occurs on $\chi_{can}^1$ and $\overline{\chi}$ occurs on $\chi_{can}^2$ have degree one, so \hyperlink{Property $\left(\#\right)$}{Property $\left(\#\right)$}  is satisfied. These characters are precisely:
\[
sgn\cdot \epsilon_1\cdot \epsilon_2 ,\qquad sgn \cdot \epsilon_2,  \qquad \makebox{and} \qquad sgn \cdot \epsilon_1.
\]
From Theorem \ref{ThmFormulaH^2,0(S)} we have $H^{2,0}(S)=\left(H^{1,0}(C_1)\otimes H^{1,0}(C_2)\right)^{S_3\times \mathbb Z_3^2}$ decomposes into three pieces of dimension one: 
\[
H^{1,0}(C_1)^{sgn\cdot \epsilon_1\cdot \epsilon_2}\otimes H^{1,0}(C_2)^{sgn\cdot \epsilon_1^2\cdot \epsilon_2^2}, \quad H^{1,0}(C_1)^{sgn \cdot \epsilon_2}\otimes H^{1,0}(C_2)^{sgn \cdot \epsilon_2^2}, 
\]
\[
H^{1,0}(C_1)^{sgn \cdot \epsilon_1}\otimes H^{1,0}(C_2)^{sgn \cdot \epsilon_1^2}.
\]
Corollary \ref{cor: base_locus_Kc1c2_chi_G} determines which is respectively a generator of the associated linear subsystem given by each of these pieces: 
\[
\begin{split}
	& R_{(0,1)}^{vert}+R_{(1,\lambda)}^{hor}+2R_{(-1,1)}^{hor},\\
	& 2R_{(1,1)}^{vert}+ 2R_{(0,1)}^{hor},\\
	& 2R_{(-1,1)}^{vert}+ 4R_{(-1,1)}^{hor}.
\end{split}
\]
Thus, the above three divisors are spanning the linear system $\vert K_{C_1\times C_2}\vert^{S_3\times \mathbb Z_3^2}$. 

Notice then $\vert K_{C_1\times C_2}\vert^{S_3\times \mathbb Z_3^2}$ has not fixed part, so that 
\[
M^2=\left(2R_{(1,1)}^{vert}+ 2R_{(0,1)}^{hor}\right)^2=4\cdot 2 \cdot \left(R_{(1,1)}^{vert}\cdot R_{(0,1)}^{hor}\right)=8\frac{54}{6}\cdot \frac{54}{3}=24\cdot 54.
\]
	Furthermore,  $\vert K_{C_1\times C_2}\vert^{S_3\times \mathbb Z_3^2}$ has precisely $81$ (non reduced) isolated base-points $R_{(1,1)}^{vert}\cap R_{(-1,1)}^{hor}$.
	We can compute $M^2- \widehat{M}^2$ by applying \textit{Correction term formula} \ref{thm:Correction_term_formula}, recursively for each base-point of $\vert K_{C_1\times C_2}\vert^{S_3\times \mathbb Z_3^2}$. Indeed, in a neighbourhood of each of these base-points the three divisors are respectively 
	\[
	2R_{(-1,1)}^{hor}, \qquad 2R_{(1,1)}^{vert} \qquad \makebox{and} \qquad 4R_{(-1,1)}^{hor},
	\]
	and since $R_{(1,1)}^{vert}$ and $R_{(-1,1)}^{hor}$ are transversal, then we are in the situation of Theorem \ref{thm:Correction_term_formula}, with $H=R_{(-1,1)}^{hor}$ and $K=R_{(1,1)}^{vert}$,  $a=4, b=c=2$, and $d=0$. This implies $ad+bc=4\leq ab=8$. The correction term is $ab+cd=4$ for each of the $81$ base-points. Thus 
	\[
	M^2-\widehat{M}^2=4\cdot 81.\]
	The degree of the canonical map is therefore given by 
	\[
	\deg(\Phi_{K_S})=\frac{1}{54}\widehat{M}^2=\frac{1}{54}\left(M^2 - (M^2-\widehat{M}^2)\right)=\frac{1}{54}\left(54 \cdot  24- 4\cdot 81\right)=18. 
	\]
	
	\section{Some remarks on the computational complexity}\label{sec: comp_complexity}
	In this section we discuss the efficiency of the algorithm described in Subsection \ref{subsec: descr_implem_class_algorithm}. 
	\\
	As already said in the previous sections, we adopt the computational algebra system MAGMA \cite{BCP97} although all the calculations can be done through others computational algebra programs with a database of finite groups, such as GAP4. 
	\\
	Firstly, we compare the efficiency of our script with respect to the previous versions developed in \cite{BC04}, \cite{BCG08}, \cite{BP12}, \cite{BP16}, and \cite{Gle15}. 
	\\
	Thus, in the following table we firstly report the computation time and memory usage of \textit{FindSurfaces\_with\_Fixed\_Ksquare\_chi(ExistingSurfaces(ListGroups(1,$K^2$)))} for each $K^2\in \{-1,\dots, 8\}$, which returns automatically all regular surfaces with $\chi(\mathcal O_S)=1$ and $K^2_S=K^2$ that are product-quotient surfaces:
	\begin{table}[h]
		\vspace{0pt}
		\centering
		\begin{tabular}{|c||c|c| c|c| c| c| c| }
			\hline
			$K^2$ & $8$ & $7$ & $6$ & $5$ & $4$ &$3$ & $2$\\
			\hline
			time ($s$) & $332.79$ & $0.0$ & $712.28$ & $111.93$ & $1360.74$ & $917.95$ & $1268.29$ \\  
			\hline
			memory (MB)& $216.88$ & $32.09$ & $184.78$ & $184.75$ & $184.78$ & $184.75$ & $216.75$\\
			\hline
		\end{tabular}
		\caption{}
		\label{tab: comp_complex_chi=1}
	\end{table}
\begin{table}[h]
	\vspace{0pt}
	\centering
	\begin{tabular}{|c||c|c|c|}
		\hline
		$K^2$ & $1$ & $0$ & $-1$\\
		\hline
		time ($s$) & $812.59$ & $15153.71$ & $343947.72$\\  
		\hline
		memory (MB)& $216.78$ & $325.72$ & $431$\\
		\hline
	\end{tabular}
	\caption{}
	\label{tab: comp_complex_chi=1_2}
\end{table}
\begin{remark}
	We compared our results with respect to those of \cite{BC04} and \cite{BCG08} (for $K^2=8$) and those listed in the tables of \cite{BP12} (for $1\leq K^2\leq 8$). We noticed that there are two mistakes, since the authors forgot the possibly exchanging of the factors which provides only one irreducible family of surfaces instead of two,  so $N=1$, in the cases $G=\mathbb Z_5^2$ and $G=\mathbb Z_5^2\rtimes \mathbb Z_3$.
	
	The mistake found for $G=\mathbb Z_5^2$ was already mentioned in \cite[Remark 3.2 (3)]{BCF12}, while that for $G=\mathbb Z_5^2\rtimes \mathbb Z_3$ seems never discovered to our knowledge. 
\end{remark}
\begin{remark}
	Comparing the results for $K^2=0$ with respect those of \cite{BP16}, we noticed that \cite[Table 1]{BP16} does not contain the following two other cases: 
	\begin{table*}[h]
		\vspace{0pt}
		\centering
		\begin{tabular}{|c|c|c|c|}
			\hline
			$Sing(X)$ & $t_1$ & $t_2$ & $Id(G)$\\
			\hline
			\hline
			$1/4$, $1/2^4$, $3/4$& $2,4,6$ & $2,4,6$ & $\langle72,40\rangle$\\  
			$1/4$, $1/2^4$, $3/4$& $2,4,5$ & $2,4,6$& $\langle 120,34\rangle$\\
			\hline
		\end{tabular}
		\caption{}
		\label{tab: missed_cases}
	\end{table*}
  \textcolor{white}{..................} We verified that the MAGMA script of \cite{BP16} returns also these results, so the authors just forgot to include them in their list of Table 1. 
  
  We point out also that our code returns other three results than those of \cite[Table 1]{BP16} and \ref{tab: missed_cases}, listed in Table \ref{tab: missed_cases_2}.
  \begin{table}[h]
  	\vspace{0pt}
  	\centering
  	\begin{tabular}{|c|c|c|c|}
  		\hline
  		$Sing(X)$ & $t_1$ & $t_2$ & $Id(G)$\\
  		\hline
  		\hline
  		$2/5$, $1/2^4$, $3/5$& $2,4,5$ & $2,4,5$ & $\langle160,234\rangle$\\  
  		$1/3^2$, $1/2^2$, $2/3^2$ & $3,3,4$ & $3,3,4$& $\langle 48, 3\rangle$\\
  		$1/3^2$, $1/2^2$, $2/3^2$& $3,3,4$ & $2,3,7$& $\langle 168, 42\rangle$\\
  		\hline
  	\end{tabular}
  	\caption{}
  	\label{tab: missed_cases_2}
  \end{table}
These cases were not listed in Table 1 of  \cite{BP16} since they do not provide surfaces of general type. Indeed, the number $\xi(X)$ is respectively equal  to $\frac{1}{3}$, $\frac{2}{5}$ and $\frac{2}{5}$ for such cases, so that $\xi(X)<\frac{1}{2}$ and by \cite[Thm 5.3 and Cor 5.4]{BP16} they cannot give surfaces of general type.   

We also excluded manually the secondary output of $ListGroups(0,1)$ (with a similar approach such as that explained in Section \ref{sec: class_chi=4} for the case $(K^2,\chi)=(32,4)$) to prove the following
\begin{theorem}
	Let $S$ be a product-quotient surface with $K^2_S=p_g(S)=q(S)=0$, then one of the following holds: 
	\begin{enumerate}
		\item 	$S$ realizes one of the families of surfaces described in \cite[Table 1]{BP16}, Table \ref{tab: missed_cases}, Table \ref{tab: missed_cases_2}. Furthermore, all these surfaces are minimal and \textit{not} of general type;
		\item $S$ is the surface described in \cite[Prop 7.1]{BP16}. In particular, it is a surface of general type whose minimal model is a numerical Godeaux surface with torsion of order $4$. 
	\end{enumerate}
\end{theorem}
\end{remark}

\begin{remark}
	Regarding the classification obtained for $K^2=-1$, we get one case more than those two  found in \cite{BP16}, see Table \ref{tab: Ksquare=-1}. This happened because the script developed in \cite{BP16} looks for only surfaces of general type and so automatically exclude cases with  $\xi(X)<\frac{1}{2}$. However, the last case found by us has $\xi(X)=\frac{2}{5}$ and so has been automatically excluded. 
	
	Furthermore, we found two irreducible families sharing the same algebraic data of the  group $\mathbb Z_5^2$ instead of only one family  found in \cite{BP16}. 
	\begin{table}[h]
		\vspace{0pt}
		\centering
		\begin{tabular}{|c|c|c|c|}
			\hline
			$Sing(X)$ & $t_1$ & $t_2$ & $Id(G)$\\
			\hline
			\hline
			$1/5$,	$ 2/5^2$,$ 4/5$& $2,5,5$ & $3,3,5$ & $\langle 60, 5\rangle$\\  
			$1/5$, $1/2^4$, $4/5$& $2,4,5$ & $2,4,5$& $\langle 160, 234\rangle$\\
			$1/5^5$& $5,5,5$ & $5,5,5$& $\langle 25, 2\rangle$\\
			\hline
		\end{tabular}
		\caption{}
		\label{tab: Ksquare=-1}
	\end{table}
  We have also excluded manuallly the secondary output of $ListGroups(-1,1)$ to prove the following 
  \begin{theorem}
  	Let $S$ be a product-quotient surface with $K^2_S=-1$, $p_g(S)=q(S)=0$, then $S$ realizes one of the families of surfaces described in Table \ref{tab: Ksquare=-1}. Furthermore, the first two cases of the table give product-quotient surfaces that are minimal and not of general type. Instead, the last case with group $\mathbb Z_5^2$ gives two irreducible families of surfaces that are not minimal and whose minimal model is a numerical Godeaux surface with torsion of order $5$.  
  \end{theorem}
\end{remark}
%
%
Regarding the efficiency of our script to obtain the results of \cite{G16}, we see that

 \textit{FindSurfaces\_with\_Fixed\_Ksquare\_chi(ExistingSurfaces(ListGroups($K^2$,$\chi$)))} 
 \\
 for $K^2=16$  and $\chi=2$ has a computation time equal to $27306.09$ seconds, whilst the memory usage consists of $825$ megabyte. 
 \vspace{0,3cm}
 
 It remains to discuss the computational complexity of our program for the classification given in Section \ref{sec: class_chi=4} of  the present paper. In particular, we only give here details for the case $\chi=4$ and $K^2=32$:
 \begin{enumerate}
 	\item \textit{ListGroup(32,4)} has a computation time equal to 6075.11 seconds, and a memory usage of 325.78 megabyte;
 	\item The function \textit{ExistingSurfaces} computed on the main output of \textit{ListGroup(32,4)}  has  a computation time equal to 587.860 seconds and a memory usage of 182.66 megabyte;
 	\item The function \textit{FindSurfaces\_with\_Fixed\_Ksquare\_chi} computed on the list produced by \textit{ExistingSurfaces} requires 135494.74 seconds and 49756.03 megabyte.
 \end{enumerate}
\begin{remark}\label{rem: comp_diffi_Find_Surf}
	The computation at (3) is done by excluding from the list produced by \textit{ExistingSurfaces} six triples, which make the computation time too long. 
	
	They are triples $no.9,15,23,35,67,72$ of tables  \ref{table: class_PQ_pg=3} and  \ref{table: class_PQ_pg=3_2} of the appendix of this paper. 
	
	
	We ran \textit{FindSurfaces\_with\_Fixed\_Ksquare\_chi} for each of these triples separately. However, the computer used could not handle the extensive calculations and terminated the program automatically for triples $no.$ 15, 35, and 67. As a result, we are only able to confirm the existence of irreducible families of product-quotient surfaces with compatible algebraic data for these triples, but we cannot compute the exact number of them. Therefore, we marked an interrogative point in the $N$-column for each of these triples. Similar difficulties were encountered for other challenging cases with $K^2$ ranging from $23$ to $30$, where we consistently used the symbol '?' in the $N$-column.
	
	Instead, 
	\begin{itemize}
		\item for $no.$ $9$ the computation time is 341875.6 seconds and it requires 5893.69 megabyte;
		\item for $no.$ $23$ we need 137615.33 seconds and it requires 214.19 megabyte;
		\item for  $no.$ $72$ the computation time is 56131.82 seconds and it requires 214.19 magabyte. 
	\end{itemize}
\end{remark}
Let us give a brief comment on the computational complexity of the hardest steps to prove Theorem \ref{thm: skipped_cases_Ksquare_32} of Section \ref{sec: class_chi=4}.

To attain Proposition \ref{prop: exclude_groups_ord_lesseq_2000} we selected those triples of the secondary output of \textit{ListGroups} having a group order different from 1024, 1536, and less or equal to 2000. We used \textit{HowToExclude} script to exclude such list of triples. The computation time has been 22335.810 seconds and it has required 265,56 megabyte. 

Instead, to exclude those triples having a group order equal to $1536$, that are in total five, we always run separately \textit{HowToExclude} for each of them. The computation time for each of them was approximately of 162000 seconds (45 hours) and the memory usage of 740 megabyte. 

Regarding Proposition \ref{prop: from_19_to_62}, we encountered difficulties to prove that spherical systems of generators of a group of order $1536$ with signature $[4,4,4]$ do not exist. The function \textit{ExSphSystem} has required more or less one week of computation and a memory usage of 1072.44 megabytes. 

\appendix
\section{}

	In this appendix we list all regular product-quotient surfaces of general type with $23\leq K^2\leq 32$ and $p_g=3$. In particular, we list the following information in the columns of tables \ref{table: class_PQ_pg=3} to \ref{table: class_PQ_pg=3_excep_minimality}:
\begin{itemize}
	\item $K^2_S$ is the self-intersection of the canonical class of $S$;
	\item $G$ is the group, and $Id$ is the identifier of the group in the MAGMA database of small groups; hence the pair $\langle d,n\rangle$ of each row denotes that $G$ is the $n$-th group of order $d$ in the MAGMA database of small groups. Whenever $G$ has not an easy description, we simply denote it by $G(d,n)$, the group in the MAGMA database having identifier $\langle d,n\rangle $;
	\item Sing($X$) is the singular locus of the quotient model $X:=\left(C_1\times C_2\right)/G$ defining the product-quotient surface $S$. It is given as a sequence of rational numbers with multiplicities, describing the types of cyclic quotient singularities. For instance, $3/5^4$ means $4$ singular points of type $\frac{1}{5}(1,3)$; 
	\item $t_1$ and $t_2$ are the signatures of the corresponding spherical systems of generators, cf. Definition \ref{defn: ssg};
	\item $N$ is the number of irreducible families. Indeed our tables have 555 lines, but we collect in the same line $N$ families, which share all the other data. 
	We employ the symbol $?$ whenever we are unable to determine the exact number of families in a row due to computational time constraints or machine memory overflow;
	\item deg$(\Phi_S)$ contains, for each family of the row, the degree of the canonical map of a surface $S$ belonging to that family, whenever the computation of the degree can be done. For example, if there are $N$ irreducible families in a row, where $N = 3$, and the degrees listed in the deg$(\Phi_S)$ box for that row are $12$ and $16$, it indicates that the degree of the canonical map has been computed for surfaces from only two of the three families. Specifically, the degree is $12$ for one family and $16$ for the other.
	
	Furthermore, since the degree of the canonical map is not a topological invariant 
	then it may happen that surfaces belonging to the same family have distinct degrees of the canonical map. 
	In this case, we simply list sequentially all degrees of the canonical map of the surfaces belonging to that family. For instance, suppose deg$(\Phi_S$) of a row is $12, (18,16), 18$. This means surfaces of two of these three families have a degree of the canonical map that is constant on the family and equal respectively to $12$ and $18$, while the other family has surfaces with a degree of the canonical map either equal to $18$ or $16$. 
	\\
	The number $0$ means that the image of $\Phi_S$ has dimension $1$. 
\end{itemize}
For the groups occurring in tables \ref{table: class_PQ_pg=3} to  \ref{table: class_PQ_pg=3_excep_minimality} we use the following notation: 
\\
$\mathbb Z_n^k$ is $k$-times the cyclic group of order $n$.
\\
$S_n$ is the symmetric group of $n$ letters.
\\
$\mathcal{A}_n$ is the alternating group.
\\
ASL($n,k$) is the affine special linear group of $\mathbb Z_k^n$.
\\
$\PSL(2,n)$ is the group of $2\times 2$ matrices over $\mathbb Z_n$ with determinant $1$ modulo the subgroup generated by $-\Id$. 
\\
$\text{SO}(3,7)$ is the group of $3\times 3$ orthogonal matrices over $\mathbb F_7$ with determinant $1$. 
\\
He$p$ is the Heisenberg group of order $p^3$:
\[
\text{He}p:=\langle x,y,z | z^{-1}xyx^{-1}y^{-1}, x^p,y^p,z^p, xz=zx, yz=zy\rangle
\]
A $3$-dimensional representation of \text{He}$p$ (over the field $\mathbb Z_p$) is given by sending 
\[
x\mapsto 
.
\]
$Q_8$ is the quaternion group:
\[
Q_8:=\langle x,y | x^4, x^2y^{-2}, y^{-1}xyx\rangle.
\]
$K\wr H$ is the wreath product, so it is the semidirect product $K^H\rtimes H$, where $K^H$ is the set of \textit{functions} $f\colon H\to K$, with a group operation given by pointwise multiplication. Here $H$ is acting on $K^H$ via left multiplication: 
\[
h\cdot f:=f\circ h^{-1}, \qquad f\colon H\to K\in K^H. 
\] 
\begin{table*}
	\centering
	\begin{tabular}{| c|c| c |c|c| l| c| c|c | }
		\hline
		$no.$ &	 $K^2_S$ & Sing($X$) & $t_1$ & $t_2$ & $G$ & $Id$ &$N$ & $\text{deg}(\Phi_S)$\\
		\hline
		\hline
		$1$ &32 & \empty & $2^{6}$ &	$2^8$ & $\mathbb Z_2^3$ & $\langle 8, 5\rangle $ & 3 &$8, 16^2$ \\ 
		$2$ & 32& \empty & $2^{5}$ &	$2^{12}$ & $\mathbb Z_2^3$ & $\langle 8, 5\rangle $ & 3 & 0,4,8\\ 
		$3$ &32 & \empty & $3^4$ &	$3^7$ & $\mathbb Z_3^2$ & $\langle 9, 2\rangle $ & 2 & 6, 12  \\
		$4$ & 32 & \empty & $3^5$ &	$3^5$ & $\mathbb Z_3^2$ & $\langle 9, 2\rangle $ & 1 & 9 \\ 
		$5$ & 32 & \empty & $2^{3},4^2$ &	$2^3,4^2$ & $\mathbb Z_2^2\rtimes \mathbb Z_4$ & $\langle 16, 3\rangle $ & 2 & 16\\
		$6$ &32 & \empty & $2^{2},4^2$ &	$2^2,4^4$ & $\mathbb Z_2^2\rtimes \mathbb Z_4$ & $\langle 16, 3\rangle $ & 2 & \\
		$7$ & 32& \empty& $2^2,4^2$ &	$2^5,4^2$ & $\mathbb Z_2^2\rtimes \mathbb Z_4$ & $\langle 16, 3\rangle $ & 6 & 8\\
\hline
	\end{tabular}
	\caption{Minimal product-quotient surfaces of general type with $q=0$, $p_g=3$ and $K^2=32$} \label{table: class_PQ_pg=3} 
\end{table*}

\begin{table*}
	\centering
	\begin{tabular}{| c|c|c| c|c| l| c| c|c | }
		\hline
		$no.$ &$K^2_S$ & Sing($X$) &	 $t_1$ & $t_2$ & $G$ & $Id$ &$N$ & $\text{deg}(\Phi_S)$\\
		\hline
		\hline
		$43$ &32& \empty& $2,5^2$ &	$3^7$ & $\mathcal A_5$ & $\langle 60, 5\rangle $ & 2 & \\ 
		$44$ &32& \empty& $2^{8}$ &	$3^2,5$ & $\mathcal A_5$ & $\langle 60, 5\rangle $ & 1 & \\
		$45$ &32& \empty& $2^{4},3$ &	$5^3$ & $\mathcal A_5$ & $\langle 60, 5\rangle $ & 1 & \\ 
		$46$ &32& \empty& $3^4$ &	$5^3$ & $\mathcal A_5$ & $\langle 60, 5\rangle $ & 1 & \\ 
		$47$ &32& \empty& $2^{6}$ &	$3,5^2$ & $\mathcal A_5$ & $\langle 60, 5\rangle $ & 2 & \\  
		$48$ &32& \empty& $2^{2},4^2$ &	$2^2,4^2$ & $G(64,60)$ & $\langle 64, 60\rangle $ & 3 & 32 \\   
		$49$ &32& \empty& $2^{2},4^2$ &	$2^2,4^2$ & $\mathbb Z_4\rtimes (\mathbb Z_2^2\rtimes \mathbb Z_4)$ & $\langle 64, 71\rangle $ & 1 & \\
\hline
	\end{tabular}
	\caption{Minimal product-quotient surfaces of general type with $q=0$, $p_g=3$ and $K^2=32$} \label{table: class_PQ_pg=3_2}
\end{table*}

\begin{table*}
	\centering
	\begin{tabular}{| c|c|c| c|c| l| c| c|c | }
		\hline
		$no.$ &$K^2_S$ & Sing($X$) &	 $t_1$ & $t_2$ & $G$ & $Id$ &$N$ & $\text{deg}(\Phi_S)$\\
		\hline
		\hline
		$87$ &30& $1/2^2$& $2^3,4$ &	$2^{10},4$ & $\mathbb Z_2\times D_4$ & $\langle 16, 11\rangle $ &6 &0 \\ 
		$88$ &30& $1/2^2$& $2^{4},4$ &	$2^5,4$ & $\mathbb Z_2\times D_4$ & $\langle 16, 11\rangle $ & 2&  4\\ 
		$89$ &30& $1/2^2$& $2^{3},8$ &	$2^5,4$ & $\mathbb Z_2\times D_8 $ & $\langle 32, 39\rangle $ & 1  &  \\   
		$90$ &30& $1/2^2$& $2^{3},12$ &	$2^4,4$ & $S_3\times D_4$ & $\langle 48, 38\rangle $ &  1 & \\ 
		$91$ &30& $1/2^2$& $2^{3},4$ &	$2^3,6,12$ & $S_3\times D_4$ & $\langle 48, 38\rangle $ & 1 & \\
		$92$ &30& $1/2^2$& $2,4,6$ &	$2^{10},4$ & $\mathbb Z_2\times S_4$ & $\langle 48, 48\rangle $ & ? & \\
		$93$ &30& $1/2^2$& $2^2,3,4$ &	$2^4,4$ & $\mathbb Z_2\times S_4$ & $\langle 48, 48\rangle $ & 2 & \\
\hline
	\end{tabular}
	\caption{Minimal product-quotient surfaces of general type with $q=0$, $p_g=3$, and $K^2\in \{30,29,28\}$} \label{table: class_PQ_pg=3_3}
\end{table*}

\begin{table*}
	\centering
	\begin{tabular}{| c|c|c| c|c| l| c| c|c | }
		\hline
		$no.$ &$K^2_S$ & Sing($X$) &	 $t_1$ & $t_2$ & $G$ & $Id$ &$N$ & $\text{deg}(\Phi_S)$\\
		\hline
		\hline
		$129$ &28& $1/2^4$& $2^{11}$ &	$3,4^2$ & $S_4$ & $\langle 24, 12\rangle $ &  1& \\ 
		$130$ &28& $1/2^4$& $2^3,6^2$ &	$2^5$ & $\mathbb Z_2^2\times  S_3$ & $\langle 24,14\rangle $ & 3 & \\
		$131$ &28& $1/2^4$& $2^5$ &	$2^5,3$ & $\mathbb Z_2^2\times S_3$ & $\langle 24, 14\rangle $ &  1& \\   
		$132$ &28& $1/2^4$& $2,4^2,8$ &	$2^2,4^2$ & $\mathbb Z_4\wr \mathbb Z_2$ & $\langle 32, 11\rangle $ &  1& \\
		$133$ &28& $1/2^4$& $2^{3},4$ &	$2^3,4^3$ & $\mathbb Z_2^2\wr\mathbb Z_2$ & $\langle 32, 27\rangle $ & 4 & \\  
		$134$ &28& $1/2^4$& $2^3,4$ &	$2^6,4$ & $\mathbb Z_2^2\wr \mathbb Z_2$ & $\langle 32, 27\rangle $ &  30& \\
		$135$ &28& $1/2^4$& $2^{3},4$ &	$2^3,4^3$ & $\mathbb Z_2^2\rtimes  D_4 $ & $\langle 32, 28\rangle $ &  4&  \\
\hline
	\end{tabular}
	\caption{Minimal product-quotient surfaces of general type with $q=0$, $p_g=3$ and $K^2=28$} \label{table: class_PQ_pg=3_4}
\end{table*}

\begin{table*}
	\centering
	\begin{tabular}{| c|c|c| c|c| l| c| c|c | }
		\hline
		$no.$ &$K^2_S$ & Sing($X$) &	 $t_1$ & $t_2$ & $G$ & $Id$ &$N$ & $\text{deg}(\Phi_S)$\\
		\hline
		\hline
		$171$ &28& $1/2^4$& $2,6,10$ &	$2^3,4$ & $\mathbb Z_2\times S_5$ & $\langle 240, 189\rangle $ &2 & \\
		$172$ &28& $1/2^4$& $2,4,8$ &	$3,6^2$ & SO$(3,7)$ & $\langle 336, 208\rangle $ & 2 & \\
		$173$ &28& $1/2^4$& $2,4,8$ &	$2,6^2$ & $\mathbb Z_2\times \text{SO}(3,7)$ & $\langle 672, 1254\rangle $ &  2& \\
		$174$ &28& $1/2^4$& $2,4,6$ &	$2,8^2$ & $\mathbb Z_2\times \text{SO}(3,7)$ & $\langle 672, 1254\rangle $ & 2 & \\   
		\hline
		$175$ &28&  $3/5^2$& $2^3,5$ &	$3^3,5$ & $\mathcal A_5$ & $\langle 60, 5\rangle $ & 2 & \\
		$176$ &28& $3/5^2$& $2^6,5$ &	$3^2,5$ & $\mathcal A_5$ & $\langle 60, 5\rangle $ & 1& \\
		$177$ &28& $3/5^2$& $2^3,5$ &	$3,6,10$ & $\mathbb Z_2\times \mathcal A_5$ & $\langle 120, 35\rangle $ &1 & \\
\hline
	\end{tabular}
	\caption{Minimal product-quotient surfaces of general type with $q=0$, $p_g=3$ and $K^2\in\{28,27,26\}$} \label{table: class_PQ_pg=3_5}
\end{table*}


\begin{table*}
	\centering
	\begin{tabular}{| c|c|c| c|c| l| c| c|c | }
		\hline
		$no.$ &$K^2_S$ & Sing($X$) &	 $t_1$ & $t_2$ & $G$ & $Id$ &$N$ & $\text{deg}(\Phi_S)$\\
		\hline
		\hline
		$215$ &26& $1/4,1/2^2,3/4$& $2,4,7$ &	$3^3,4$ & $ \PSL(2,7)$ & $\langle 168, 42\rangle $ & 2& \\ 
		$216$ &26& $1/4,1/2^2,3/4$& $2,4,7^2$ &	$3^2,4$ & $\PSL(2,7)$ & $\langle 168, 42\rangle $ &4 & \\  
		$217$ &26& $1/4,1/2^2,3/4$& $3^2,4$ &	$3^3, 4$ & ASL$(2,3)$ & $\langle 216, 153\rangle $ &  4& \\
		$218$ &26& $1/4,1/2^2,3/4$& $2,4,5$ &	$3^3,4$ & $\mathcal A_6$ & $\langle 360, 118\rangle $ &8  & \\   
		\hline
		$219$ &26& $1/3^2,2/3^2$& $2^{8},3^2$ &	$3,4^2$ & $S_4$ & $\langle 24, 12\rangle $ &  1& \\ 
		$220$ &26& $1/3^2,2/3^2$& $2^2,3^2$ &	$2^3,3,4^2$ & $S_4$ & $\langle 24, 12\rangle $ &1  & \\ 
		$221$ &26&  $1/3^2,2/3^2$& $2^2,3^2$ &	$3,4^4$ & $S_4$ & $\langle 24, 12\rangle $ & 2& \\
\hline
	\end{tabular}
	\caption{Minimal product-quotient surfaces of general type with $q=0$, $p_g=3$ and $K^2=26$} \label{table: class_PQ_pg=3_6}
\end{table*}


\begin{table*}
	\hspace*{-\leftmargin}
	\begin{tabular}{| c|c|c| c|c| l| c| c|c | }
		\hline
		$no.$ &$K^2_S$ & Sing($X$) &	 $t_1$ & $t_2$ & $G$ & $Id$ &$N$ & $\text{deg}(\Phi_S)$\\
		\hline
		\hline
		$259$ &26& $1/3^2,2/3^2$& $2^3,3$ &	$3,8^2$ & $G(192,1494)$ & $\langle 192, 1494\rangle $ &1 & \\  
		$260$ &26& $1/3^2,2/3^2$& $2,5,6$ &	$3,8^2$ & $\SL(2,5)\rtimes \mathbb Z_2$ & $\langle 240, 90\rangle $ &1  & \\ 
		$261$ &26& $1/3^2,2/3^2$& $2,12^2$ &	$3,4^2$ & $\mathcal A_5\rtimes \mathbb Z_4$ & $\langle 240, 91\rangle $ &1 & \\
		$262$ &26&  $1/3^2,2/3^2$& $2,6^2$ &	$4^2,6$ & $\mathbb Z_2\times S_5$ & $\langle 240, 189\rangle $ &1  & \\
		$263$ &26&  $1/3^2,2/3^2$& $3^2,4$ &	$4,6^2$ & $G(384,4)$ & $\langle 384, 4\rangle $ &2 & \\ 
		$264$ &26&  $1/3^2,2/3^2$& $2,3,11$ &	$3,5^2$ & $\PSL(2,11)$ & $\langle 660, 13\rangle $ &6  & \\   
		$265$ &26& $1/3^2,2/3^2$& $2,3,7$ &	$3,13^2$ & $\PSL(2,13)$ & $\langle 1092, 25\rangle $ &  12& \\
\hline
	\end{tabular}
	\caption{Minimal product-quotient surfaces of general type with $q=0$, $p_g=3$ and $K^2\in\{26,25,24\}$} \label{table: class_PQ_pg=3_7}
\end{table*}

\begin{table*}
	\centering
	\begin{tabular}{| c|c|c| c|c| l| c| c|c | }
		\hline
		$no.$ &$K^2_S$ & Sing($X$) &	 $t_1$ & $t_2$ & $G$ & $Id$ &$N$ & $\text{deg}(\Phi_S)$\\
		\hline
		\hline

		$ 302 $ &24& $ 1/2^8 $ & $ 2^3, 4^2 $ & $ 2^4, 4 $ & $ \mathbb Z_2\times D_4 $ & $\langle 16, 11 \rangle$ & 2& \\
		$ 303 $ &24& $ 1/2^8 $ & $ 2^3, 4 $ & $ 2^{10} $ & $ \mathbb Z_2\times D_4 $ & $\langle 16, 11 \rangle$ & 27& 0\\
		$ 304 $ &24& $ 1/2^8 $ & $ 2^5 $ & $ 2^7 $ & $\mathbb Z_2\times D_4 $ & $\langle 16, 11 \rangle$ &4 &16 \\
		$ 305 $ &24& $ 1/2^8 $ & $ 2^4, 4 $ & $ 2^6 $ & $ \mathbb Z_2\times D_4$ & $\langle 16, 11 \rangle$ & 14&$8^2$ \\
		$ 306 $ &24& $ 1/2^8 $ & $ 2^3, 4 $ & $ 2^7, 4^2 $ & $ \mathbb Z_2\times D_4 $ & $\langle 16, 11 \rangle$ & 9&  \\
		$ 307 $ &24& $ 1/2^8 $ & $ 2, 4^3 $ & $ 2^6 $ & $ D_4\rtimes \mathbb Z_2 $ & $\langle 16, 13 \rangle$ & 1& \\
		$ 308 $ &24& $ 1/2^8 $ & $ 2^5 $ & $ 2^7 $ & $ \mathbb Z_2^4 $ & $\langle 16, 14 \rangle$ &13 &$8^5, 12^4, 16^4$  \\
\hline
	\end{tabular}
	\caption{Minimal product-quotient surfaces of general type with $q=0$, $p_g=3$ and $K^2=24$} \label{table: class_PQ_pg=3_8}
\end{table*}

\begin{table*}
	\hspace*{-\leftmargin}
	\begin{tabular}{| c|c|c| c|c| l| c| c|c | }
		\hline
		$no.$ &$K^2_S$ & Sing($X$) &	 $t_1$ & $t_2$ & $G$ & $Id$ &$N$ & $\text{deg}(\Phi_S)$\\
		\hline
		\hline

		$ 346 $ &24& $ 1/2^8 $ & $ 2, 6^2 $ & $ 2^4, 3^3 $ & $S_3\times S_3 $ & $\langle 36, 10 \rangle$ & 2& \\
		$ 347 $ &24& $ 1/2^8 $ & $ 2^2, 3^2, 6^2 $ & $ 2^3, 3 $ & $ S_3\times S_3$ & $\langle 36, 10 \rangle$ & 3& \\
		$ 348 $ &24& $ 1/2^8 $ & $ 2^2, 3, 6 $ & $ 2^2, 6^2 $ & $ S_3\times S_3 $ & $\langle 36, 10 \rangle$ &1 & \\
		$ 349 $ &24& $ 1/2^8 $ & $ 2^3, 3 $ & $ 3, 6^4 $ & $ S_3\times S_3 $ & $\langle 36, 10 \rangle$ & 2& \\

		$ 350 $ &24& $ 1/2^8 $ & $ 2^2, 3, 6 $ & $ 2^4, 3 $ & $ S_3\times S_3 $ & $\langle 36, 10 \rangle$ &1 & \\
		
		$ 351 $ &24& $ 1/2^8 $ & $ 2^3, 3, 6 $ & $ 2^3, 6 $ & $ S_3\times S_3$ & $\langle 36, 10 \rangle$ &2 & \\
		$ 352 $ &24& $ 1/2^8 $ & $ 2^2, 6^2 $ & $ 6^3 $ & $ \mathbb Z_6\times S_3 $ & $\langle 36, 12 \rangle$ & 1& \\
\hline
	\end{tabular}
	\caption{Minimal product-quotient surfaces of general type with $q=0$, $p_g=3$ and $K^2=24$} \label{table: class_PQ_pg=3_9}
\end{table*}

\begin{table*}
	\centering
	\begin{tabular}{| c|c|c| c|c| l| c| c|c | }
		\hline
		$no.$ &$K^2_S$ & Sing($X$) &	 $t_1$ & $t_2$ & $G$ & $Id$ &$N$ & $\text{deg}(\Phi_S)$\\
		\hline
		\hline

		$ 389$ &24& $ 1/2^8 $ & $ 2^2, 4^2 $ & $ 2^3, 8 $ & $ G(64, 130) $ & $\langle 64, 130 \rangle$ &1 & \\
		$ 390 $ &24& $ 1/2^8 $ & $ 2^2, 4^2 $ & $ 2^3, 8 $ & $ D_4\rtimes D_4 $ & $\langle 64, 134 \rangle$ &1 & \\
		$ 391 $ &24& $ 1/2^8 $ & $ 2, 4^3 $ & $ 2^3, 4 $ & $ D_4\rtimes D_4$ & $\langle 64, 134 \rangle$ & 1& \\
		$ 392 $ &24& $ 1/2^8 $ & $ 2^3, 4 $ & $ 2^4, 4 $ & $ \mathbb Z_2\wr \mathbb Z_2^2 $ & $\langle 64, 138 \rangle$ &5 & \\
		$ 393 $ &24& $ 1/2^8 $ & $ 2, 4^3 $ & $ 2^3, 4 $ & $ \mathbb Z_2\wr \mathbb Z_2^2 $ & $\langle 64, 138 \rangle$ &2 & \\
		$ 394 $ &24& $ 1/2^8 $ & $ 2^3, 4 $ & $ 2^4, 4 $ & $ \mathbb Z_4\rtimes D_8 $ & $\langle 64, 140 \rangle$ &1 & \\
		$ 395 $ &24& $ 1/2^8 $ & $ 2^2, 8^2 $ & $ 2^3, 4 $ & $ \mathbb Z_4\rtimes D_8 $ & $\langle 64, 140 \rangle$ & 1& \\
\hline
	\end{tabular}
	\caption{Minimal product-quotient surfaces of general type with $q=0$, $p_g=3$ and $K^2=24$} \label{table: class_PQ_pg=3_10}
\end{table*}

\begin{table*}
	\centering
	\hspace*{-\leftmargin}
	\begin{tabular}{| c|c|c| c|c| l| c| c|c | }
		\hline
		$no.$ &$K^2_S$ & Sing($X$) &	 $t_1$ & $t_2$ & $G$ & $Id$ &$N$ & $\text{deg}(\Phi_S)$\\
		\hline
		\hline

		$ 434 $ &24& $ 1/2^8 $ & $ 2, 6, 9 $ & $ 3, 6^2 $ & $ \mathbb Z_3\wr S_3 $ & $\langle 162, 10 \rangle$ &4 & \\
		$ 435 $ &24& $ 1/2^8 $ & $ 2, 7^2 $ & $ 2^2, 3^2 $ & $ \PSL(2,7) $ & $\langle 168, 42 \rangle$ &1 & \\
		$ 436 $ &24& $ 1/2^8 $ & $ 2^4, 7 $ & $ 3^2, 4 $ & $\PSL(2,7) $ & $\langle 168, 42 \rangle$ &2 & \\
		$ 437 $ &24& $ 1/2^8 $ & $ 2, 3, 8 $ & $ 2, 4^4 $ & $ G(192, 181) $ & $\langle 192, 181 \rangle$ &1 & \\

		$ 438 $ &24& $ 1/2^8 $ & $ 2, 4, 6 $ & $ 2^4, 4 $ & $ G(192, 955) $ & $\langle 192, 955 \rangle$ & 5& \\
		$ 439 $ &24& $ 1/2^8 $ & $ 2, 4, 6 $ & $ 2, 4^3 $ & $ G(192, 955) $ & $\langle 192, 955 \rangle$ & 1& \\
		
		$ 440 $ &24& $ 1/2^8 $ & $ 2, 4, 6 $ & $ 2^2, 6^2 $ & $ G(216, 87) $ & $\langle 216, 87 \rangle$ &1 & \\
\hline
	\end{tabular}
	\caption{Minimal product-quotient surfaces of general type with $q=0$, $p_g=3$ and $K^2\in \{24,23\}$} \label{table: class_PQ_pg=3_11}
\end{table*}

\begin{table*}
	\centering
	\hspace*{-\leftmargin}
	\begin{tabular}{| c|c|c| c|c| l| c| c|c | }
		\hline
		$no.$ &$K^2_S$ & Sing($X$) &	 $t_1$ & $t_2$ & $G$ & $Id$ &$N$ & $\text{deg}(\Phi_S)$\\
		\hline
		\hline

%
%
%

%
%

	
		$ 479 $ &23& $ 1/3^3, 2/3^3 $& $ 2, 6^2 $& $ 2^6, 3^3 $& $ \mathbb Z_2\times \mathcal A_4$& $\langle 24, 13 \rangle$ & 1& \\
		$ 480 $ &23& $ 1/3^3, 2/3^3 $& $ 3^4 $& $ 3^4 $& $ \text{He}3 $& $\langle 27, 3 \rangle$ &5 & \\
		
		$ 481 $ &23& $ 1/3^3, 2/3^3 $& $ 2, 3^3 $& $ 3^4 $& $ \mathbb Z_3\times \mathcal A_4 $& $\langle 36, 11 \rangle$ & 4& \\
		
		$ 482 $ &23& $ 1/3^3, 2/3^3 $& $ 3^2, 6 $& $ 3^6 $& $ \mathbb Z_3\times \mathcal A_4 $& $\langle 36, 11 \rangle$ &6 & \\
\hline
	\end{tabular}
	\caption{Minimal product-quotient surfaces of general type with $q=0$, $p_g=3$ and $K^2=23$} \label{table: class_PQ_pg=3_12}
\end{table*}

\begin{table*}
	\centering
	\hspace*{-\leftmargin}
	\begin{tabular}{| c|c|c| c|c| l| c| c |c|}
		\hline
		$no.$ &$K^2_S$ & Sing($X$) &	 $t_1$ & $t_2$ & $G$ & $Id$ &$N$& $\deg(\Phi_S)$\\
		\hline
		\hline
		
		$523$ &25& $1/7,2/7^2$& $2,4,7$ &	$3^3,7$ & $\PSL(2,7)$ & $\langle 168, 42\rangle $ &2 &\\
		$524$ &25& $1/7,2/7^2$ & $2,4,7$ &	$3,6,14$ & $\mathbb Z_2\times \PSL(2,7)$ & $\langle 336, 209\rangle $ &1& \\
		$525$ &25& $1/7,2/7^2$& $2,3,7$ &	$4,7,8$ & $G(1344, 814)$ & $\langle 1344, 814\rangle $ &8  &\\
		\hline 
		\hline
		$ 526 $ &24& $ 1/5, 1/3, 2/3, 4/5 $ & $ 2, 6, 10 $ & $ 2^2, 3, 5 $ & $ \mathbb Z_2\times \mathcal A_5 $ & $\langle 120, 35 \rangle$ & 1&\\
		\hline
		$ 527 $ &24& $ 1/6^2, 1/2^2, 2/3 $ & $ 2, 4, 6 $ & $ 4^4, 6 $ & $ \GL(2,\mathbb Z_4) $ & $\langle 96, 195 \rangle$ & 2 &\\

		$ 528 $ &24& $ 1/6^2, 1/2^2, 2/3 $ & $ 2, 4, 6 $ & $ 2^3, 4^2, 6 $ & $\GL(2,\mathbb Z_4)  $ & $\langle 96, 195 \rangle$ & 14 & \\
		$ 529 $ &24& $ 1/6^2, 1/2^2, 2/3 $ & $ 2, 5, 6 $ & $ 4, 6, 8 $ & $ \SL(2,5)\rtimes \mathbb Z_2$ & $\langle 240, 90 \rangle$ & 2 & \\
\hline
	\end{tabular}
	\caption{Remaining product-quotient surfaces of general type with $q=0$, $p_g=3$ and $K^2\in\{23,\dots, 32\}$ whose minimality is not established} \label{table: class_PQ_pg=3_excep_minimality}
\end{table*}

\textcolor{white}{..............}

\textcolor{white}{..............}

\textcolor{white}{..............}

\textcolor{white}{..............}

\textcolor{white}{..............}

\textcolor{white}{..............}

\textcolor{white}{..............}

\textcolor{white}{..............}

\textcolor{white}{..............}

\textcolor{white}{..............}

\textcolor{white}{..............}

\textcolor{white}{..............}

\textcolor{white}{..............}

\textcolor{white}{..............}

\textcolor{white}{..............}

\textcolor{white}{..............}
\begin{acknowledgements}
	The author expresses gratitude to Roberto Pignatelli for engaging discussions on the classification of product-quotient surfaces with low fixed invariants. 
	\\
	The author acknowledges support from the project "Classificazione di superfici di genere basso" in collaboration with the University of Trento.
	\\
	The author is partially supported by the INdAM – GNSAGA Project, “Classification Problems in Algebraic Geometry: Lefschetz Properties and Moduli Spaces” (CUP E55F22000270001).
	
\end{acknowledgements}


\end{document}